\documentclass
{article}
\usepackage[english]{babel}
\usepackage{setspace}
\usepackage{amssymb,amsthm,amscd,amsmath, txfonts,mathrsfs,enumerate}
\usepackage{mathtools}
\usepackage{amsmath,graphicx}
\usepackage{amsfonts}
\usepackage{amssymb}
\usepackage{latexsym}
\usepackage{eufrak}
\usepackage{makeidx}
\usepackage{hyperref}
\usepackage[normalem]{ulem}
\usepackage[dvipsnames]{xcolor}
\usepackage{tikz-cd}
\usepackage{dsfont}

\input xy
\xyoption{all}

\usepackage{vmargin}

\setmarginsrb{20mm}{12mm}{20mm}{18mm}%
     {10mm}{0mm}{10mm}{10mm}
     
\newtheorem{theorem}{Theorem}[section]
\newtheorem{lemma}[theorem]{Lemma}
\newtheorem{corollary}[theorem]{Corollary}
\newtheorem{proposition}[theorem]{Proposition}

\newtheorem{question}[theorem]{Question}
\newtheorem{fact}[theorem]{Fact}

\newtheorem{em-remark}[theorem]{Remark}
\newtheorem{em-remarks}[theorem]{Remarks}
\newenvironment{remark}{\begin{em-remark} \em }{\end{em-remark}}

\theoremstyle{definition}
\newtheorem{example}[theorem]{Example}

\newtheorem{definition}[theorem]{Definition}

\numberwithin{equation}{section}

\renewcommand{\phi}{\varphi}
\renewcommand{\theta}{\vartheta}

\newcommand\Pf{\mathcal P_{\mathrm{fin}}}

\def\Ab{{\mathbf{Ab}}}
\def\CompS{{\mathbf{CompSp}}}
\def\Comp{{\mathbf{CompAb}}}
\def\CompG{{\mathbf{CompGr}}}

\newcommand\Ker{\mathrm{Ker}}

\newcommand{\U}{\mathcal U}

\def\N{{\mathbb N}}

\def\T{{\mathbb T}}

\def\R{{\mathbb R}}
\def\Z{{\mathbb Z}}

\def\C{{\mathcal C}}

\def\U{{\cal U}}
\def\V{{\cal V}}
\def\W{{\cal W}}
\def\cov{\mbox{{\rm cov}}}

\def\lmod#1{#1\text{-}\mathrm{Mod}}

\newcommand{\dual}[1]{{#1}^\wedge}

\def\Ob{\mathrm{Ob}}

\def\id{\mathrm{id}}

\def\F{{\cal F}}
\def\V{{\cal V}}
\def\W{{\cal W}}

\def\LCA{{\mathbf{LCA}}}

\def\cov{\mbox{{\rm cov}}}

\def\supp{\mathrm{supp}}

\def\lre{\mathrm{LRep}}
\def\rre{\mathrm{RRep}}

\def\U{{\cal U}}

\def\F{{\cal F}}
\def\V{{\cal V}}
\def\W{{\cal W}}
\def\End{\mathrm{End}}
\def\M{\mathfrak M}

\def\nl{\bar{\lambda}}
\def\op{\mathrm{op}}

\def\L{\mathcal L}
\def\M{\mathcal M}
\def\cl{\mathrm{cl}}

\def\op{\mathrm{o}}

\def\top{\mathrm{top}}
\def\alg{\mathrm{alg}}
\def\fin{\mathrm{fin}}
\def\cov{\mathrm{cov}}
\def\hom{\mathrm{Hom}}

\def\lmod#1{#1\text{-}\mathrm{Mod}}
\def\d{\mathrm{d}}
\def\s{{\frak s}}
\def\P{\mathfrak P}

\newlength{\bibitemsep}
\newlength{\bibparskip}
\let\oldthebibliography\thebibliography
\renewcommand\thebibliography[1]{%
  \oldthebibliography{#1}%
  \setlength{\parskip}{\bibitemsep}%
  \setlength{\itemsep}{\bibparskip}%
}

\def\la#1{\mathrlap{\scalebox{.73}{\raisebox{1.35ex}{\hspace{6pt}$#1$}}}{\raisebox{-.25ex}{$\ \curvearrowright\ $}}}
\def\longla#1{\hspace{5pt}\mathrlap{\scalebox{.73}{\raisebox{1.4ex}{\hspace{-3pt}$#1$}}}{\raisebox{-.3ex}{$\ \curvearrowright\ $}}\hspace{6pt}}
\def\ra#1{\mathrlap{\scalebox{.73}{\raisebox{1.35ex}{\hspace{7pt}$#1$}}}{\raisebox{-.25ex}{$\ \curvearrowleft\ $}}}
\def\longra#1{\hspace{5pt}\mathrlap{\scalebox{.73}{\raisebox{1.4ex}{\hspace{-2.5pt}$#1$}}}{\raisebox{-.3ex}{$\ \curvearrowleft\ $}}\hspace{6pt}}

\makeindex

\title
{Ore localization of amenable monoid actions and applications towards entropy -- addition formulas and the bridge theorem}

\author{Dikran Dikranjan$^*$ \and Anna Giordano Bruno\footnote{The first and second named authors are members of GNSAGA - INdAM.} \and Simone Virili\footnote{The author was partially supported by MINECO (grant No. PID2020-113047GB-I00) and by the Generalitat de Catalunya as part of the research group ``Laboratori d'Interaccions entre Geometria, \`Algebra i Topologia (LIGAT)'' (grant No. 2021 SGR 01015). }}

\date{}

\begin{document}

\maketitle

\begin{abstract}
For a left action $S\la{\lambda}X$ of a cancellative right amenable monoid $S$ on a discrete Abelian group $X$, we construct its Ore localization $G\la{\lambda^*}X^*$, where $G$ is the group of left fractions of $S$; analogously, for a right action $K\ra{\rho}S$ on a compact space $K$, we construct its Ore colocalization $K^*\ra{\rho^*} G$. Both constructions preserve entropy, i.e., for the algebraic entropy $h_\alg$ and for the topological entropy $h_\top$ one has $h_\alg(\lambda)=h_\alg(\lambda^*)$ and $h_\top(\rho)=h_\top(\rho^*)$, respectively.

Exploiting these constructions and the theory of quasi-tilings, we extend the Addition Theorem for $h_\top$, known for right actions of countable amenable  groups on compact metrizable groups~\cite{Li}, to right actions $K\ra{\rho} S$ of cancellative right amenable monoids $S$ (with no restrictions on the cardinality) on arbitrary compact groups $K$.

When the compact group $K$ is Abelian, we prove that $h_\top(\rho)$ coincides with $h_\alg(\dual{\rho})$, where  $S\la{\dual{\rho}} X$ is the dual left action on the discrete Pontryagin dual $X=\dual{K}$, that is, a so-called Bridge Theorem.
From the Addition Theorem for $h_\top$ and the Bridge Theorem, we obtain an Addition Theorem for $h_\alg$ for left actions $S\la{\lambda} X$ on discrete Abelian groups, so far known only under the hypotheses that either $X$ is torsion~\cite{DFGB} or $S$ is locally monotileable~\cite{DFGS}. 

The proofs substantially use the unified approach towards  entropy based on the entropy of actions of cancellative right amenable monoids on appropriately defined normed monoids, as developed in \cite{DGB4} (for $\N$-actions) and \cite{VBT}.
\end{abstract}


%


\section*{Introduction}
\addcontentsline{toc}{section}{Introduction}

The notion of {\em topological entropy} $h_\top$ was first introduced by Adler, Konheim and McAndrew~\cite{AKM} in 1965 as an invariant of continuous self-maps of compact topological spaces; we denote by $\CompS$ (resp., $\mathbf{CompGr}$) the category of compact Hausdorff spaces (resp., groups) and continuous maps (resp., homomorphisms) {and in this paper when we write compact space (resp., group) we mean compact Hausdorff space (resp., group)}. Few years later, Bowen~\cite{B} introduced a notion of entropy for uniformly continuous self-maps of metric spaces (see also~\cite{Din}), later extended to general uniform spaces by Hood~\cite{Ho}. For a continuous {self-map $f\colon K\to K$ of a compact space $K$}, both Hood's extension of Bowen's entropy (taking $K$ with the unique compatible uniformity) and the original (Adler, Konheim and McAndrew's) notion of topological entropy coincide (see~\cite{DSV,GBV3} or Example~\ref{Gactionsextop2} for details).  

Let $S=(S,\cdot,{1})$ be a monoid, $K$ a compact space, and denote by $\End_{\CompS}(K)$ the monoid of continuous self-maps of $K$.
 A {\em right $S$-action} \(K\ra{\rho} S \)
is a monoid anti-homomorphism $\rho\colon S\to \End_{\CompS}(K)$, that is, letting, for each $s\in S$, $\rho_s=\rho(s)\in\End_{\CompS}(K)$,
\begin{enumerate}[({RA.}1)]
\item $\rho_{1}=\id_K$;
\item $\rho_{s t}=\rho_t\circ\rho_s$, for all $s, t\in S$.
\end{enumerate}
If $K$ is a compact group, we require that the monoid anti-homomorphism $\rho$ takes values in $\End_{\mathbf{CompGr}}(K)\subseteq \End_{\mathbf{CompSp}}(K)$, that is, we want each $\rho_s$ to be a continuous endomorphism of $K$. A continuous self-map $f\colon K\to K$ of a compact space $K$ (respectively, a continuous endomorphism of a compact group $K$) uniquely identifies a right action $\rho$ of the additive monoid $\N=(\N,+,0)$ on $K$ by setting $\rho(1)=f$ (and so $\rho(n)=f^n$ for every $n\in\N$, with the usual convention that $f^0=\id_K$). 
Using this terminology, the topological entropy can be seen as an invariant of the right $\N$-actions on compact spaces (respectively, compact groups). 

Of course, there is no reason to confine the theory of topological entropy to actions of the naturals. In fact, already in the 1980s, several authors developed a theory of topological entropy for actions of amenable groups (see, for example, \cite{Oll,OW,ST} and, for more recent results, \cite{CT,Den,Li}). A recent breakthrough is the extension of this entropy theory to actions of sofic groups by Bowen~\cite{Bo} (see also~\cite{li_kerr_book,Weisss}).

A foundational result for the theory of topological entropy of amenable group actions is the so-called Ornstein-Weiss Lemma. 
Roughly speaking, to define the topological entropy one has to study certain nets of real numbers (see \S\ref{sec:background} for details); {Lindenstrauss and Weiss~\cite{LW} (see also \cite{Krieger} for a different approach based on an idea from \cite{Gromov}) proved, using the theory of quasi-tilings of Ornstein and Weiss~\cite{OW}, that, for amenable group actions, all these nets converge.} Recently, Ceccherini-Silberstein, Coornaert and Krieger~\cite{CCK} were able to extend this convergence result to include any right action $K\ra{\rho}S$ of a cancellative and right amenable monoid $S$ on a compact space $K$, setting the bases for the study of the topological entropy of such right $S$-actions (see \S\ref{subs:top_ent}). We denote by $h_\top(\rho)$ the topological entropy of $\rho$, keeping the same notation of the case of right $\N$-actions (i.e., those induced by single continuous self-maps).

\smallskip
Any amenable group $G$ is, trivially, also a cancellative and right amenable monoid. On the other hand, any cancellative and right amenable monoid $S$ is a {\em left Ore monoid} and, therefore, it has a (necessarily amenable) group of left fractions $G=S^{-1}S$ (see \S\ref{amenable_implies_ore_subsection}). In other words, there is a natural (i.e., canonical, and even functorial) way to assign an amenable group $G$ to any cancellative and right amenable monoid $S$. We will see in \S\ref{sec_compact} that, in fact, the formation of the group of left fractions can be extended to right actions on compact spaces via the construction of the so-called {\em Ore colocalization} of a right $S$-action $K\ra\rho S$ on a compact space $K$, which is a right $G$-action $K^*\ra{\rho^*}G$ of the group of left fractions of $S$ (so necessarily an action by homeomorphisms) on a new compact space $K^*$; such colocalization can be thought of (in a sense that can be made precise) as the right $G$-action that ``best approximates'' $K\ra\rho S$. More precisely, $K^*\ra{\rho^*}G$ is the result of the following two successive modifications of $K\ra{\rho}S$. First, one restricts $\rho$ to a closed invariant subspace $\bar K$ of $K$, called the {\em surjective core} of $\rho$, to get a right $S$-action $\bar K\ra{\bar\rho}S$ by surjective continuous self-maps. This first construction already satisfies the following universal property: for any right $S$-action $K'\ra{\rho'}S$ on a compact space $K'$ by surjective continuous self-maps, any $S$-equivariant continuous map $f\colon K'\to K$  factors uniquely through the inclusion $\bar K\to K$. 
The second step then consists in building a surjective inverse system $\frak K$ of copies of $\bar K$, with transition maps suitably induced by $\bar\rho$; the inverse limit $K^*$ of  $\frak K$  admits a canonical right $S$-action $K^*\ra{}S$ by homeomorphisms, which is characterized by an analogous universal property among the right $S$-actions by homeomorphisms. This $S$-action  extends uniquely to a right $G$-action $K^*\ra{\rho^*}G$: the Ore colocalization of $K\ra{\rho}G$.
The above construction of the restricted action on the surjective core and of the Ore colocalization are remarkably well-behaved for several reasons: first, they can be directly applied (with no modification required) also to actions on compact (Abelian) groups by continuous endomorphisms, they are functorial and they both preserve the injectivity and surjectivity of equivariant morphisms. Finally and, perhaps, most importantly: they both preserve the topological entropy, that is, in the above notation, $h_\top(\rho)=h_\top(\bar \rho)=h_\top(\rho^*)$. 

\smallskip
In the sequel we discuss a deep property of the topological entropy. To this end we recall first that, for a right $S$-action $K\ra{\rho} S$ on a compact group $K$ and an $S$-invariant (not necessarily normal) subgroup $H$ of $K$, one can consider the two induced right $S$-actions 
$H\ra{\rho_H}S$ and $K/H\longra{\rho_{K/H}} S$, 
where $\rho_H$ is the restriction of $\rho$ to $H$ and $\rho_{K/H}$ is the action induced by $\rho$ on the compact space $K/H$ of left $H$-cosets in $K$. 
The equality 
\[
h_{\top}(\rho)=h_{\top}(\rho_{H})+h_{\top}(\rho_{K/H})\eqno{(\text{\sc at}_\top)}
\]
is known as Addition Theorem. It was established by Li~\cite{Li} (see also \cite{li_kerr_book,LL}) when $S$ is a countable amenable  group and $K$ is a compact metrizable group; 
in fact, this is a direct consequence of \cite[Theorem~6.1]{Li} (see also \cite[Theorem~13.47]{li_kerr_book}, including the brief discussion introducing the result).

The first of our main results  is the following extension of Li's Addition Theorem, in which we prove (\text{\sc at$_\top$})
for right $S$-actions on compact groups $K$. Therefore, our  (\text{\sc at$_\top$}) extends Li's Addition Theorem in three
different directions: first, $S$ is allowed to be non-countable, second, we pass from actions of amenable groups to actions of cancellative right amenable monoids, and third, the action is on compact groups that are not necessarily metrizable.

\medskip\noindent
{\bf Topological Addition Theorem.}
\emph{Let $S$ be a cancellative and right amenable monoid, $K$ a compact group, $K\ra{\rho} S$ a right $S$-action and $H$ a closed $S$-invariant subgroup of $K$. Then {$(\text{\sc at}_\top)$} holds.}

\medskip
The Addition Theorem is also called Yuzvinski's addition formula, since it was first proved by Yuzvinski~\cite{Yuz} for $\Z$-actions on compact metrizable groups. This result was extended to $\Z^d$-actions in by Lind, Schmidt and Ward \cite{LSW} (see also \cite{Sch}) and  further extensions (that now can be seen as particular cases of Li's result) can be found in \cite{Miles,MB}. These theorems are about measure entropy with respect to the Haar measure, but in all these cases it coincides with the topological entropy.

In \S\ref{ProofAT:top}, we first give a direct proof of {(\text{\sc at$_\top$})} in the case when $S$ is a group. 
For this, we use a new approach, that differs from the one adopted in \cite{Li,li_kerr_book} as, for example, it does not rely 
neither on the Variational Principle nor  on Bowen's definition of entropy in the metrizable case {(our approach is close to that used in \cite{WZ}, although 
these authors treat a different entropy).} After that, we deduce the general case of {(\text{\sc at$_\top$})} using {two important properties of the Ore colocalization: its ``exactness'' (i.e., taking the Ore colocalization commutes with restrictions and quotients of an action with respect to a closed invariant subgroup) and the invariance of $h_\top$ under this construction. These properties together allow us to reduce statements about the additivity of $h_\top$ for actions of cancellative right amenable monoids to the corresponding questions for actions of amenable groups.}

\medskip 
Dual to the notion of topological entropy is that of {\em algebraic entropy}. Quite surprisingly, the first definition of algebraic entropy for endomorphisms of torsion discrete Abelian groups was given by Adler, Konheim and McAndrew~\cite{AKM}, in a short final remark of the same paper where the topological entropy was introduced. This algebraic invariant was then studied for endomorphisms of any discrete Abelian group in~\cite{P1,Weiss} and, more recently, in~\cite{DGBpak,DGB,DGSZ,V1} (see~\cite{GBS,GBST,GSp} for the non-Abelian case). 

As we did above on the topological side, let us now introduce left actions of a monoid $S$ on discrete Abelian groups by endomorphisms. Denote by $\Ab$ the category of Abelian groups and group homomorphisms, fix a discrete Abelian group $X$ and let  $\End_\Ab(X)$ be the endomorphism ring of $X$. A {\em left $S$-action} 
\(S\la{\lambda} X\)
is a monoid homomorphism $\lambda\colon S\to \End_{\Ab}(X)$, that is, letting $\lambda_s=\lambda(s)\in\End_\Ab(X)$ for each $s\in S$,
\begin{enumerate}[({LA.}1)]
\item $\lambda_{{1}}=\id_X$;
\item $\lambda_{st}=\lambda_s\circ\lambda_t$, for all $s, t\in S$.
\end{enumerate}
An endomorphism $f\colon X\to X$ uniquely identifies a left action $\lambda$ of the additive monoid $\N=(\N,+,0)$ by setting $\lambda(1)=f$ as above. Using this terminology, the algebraic entropy can be seen as an invariant of the left $\N$-actions on Abelian groups. As shown in~\cite{DFGB}, the extension given by Ceccherini-Silberstein, Coornaert and Krieger~\cite{CCK} of the  Ornstein-Weiss Lemma allows one to extend the notion of algebraic entropy $h_\alg$ to left actions $ S\la{\lambda} X$ of cancellative and right amenable monoids $S$ on discrete Abelian groups $X$ (see \S\ref{subs_back_ent}). 

\smallskip
For any left action  $S\la{\lambda}X$ of a cancellative and right amenable monoid $S$ on a discrete Abelian group $X$, there is a canonical left action $G\la{\lambda^*}X^*$ of the group of left fractions $G=S^{-1}S$ (see \S\ref{amenable_implies_ore_subsection}) on a new discrete Abelian group $X^*$, obtained from $X$ and $\lambda$ via a suitable direct limit. This important construction, that we call {\em Ore localization}, is introduced and studied in detail in \S\ref{sec_discrete}. Such localization is built in a two-step process that closely parallels the construction of the Ore colocalization in the topological setting. Indeed, starting with $S\la{\lambda}X$, one has first to identify the smallest $S$-invariant subgroup  $Y\leq X$, called the {\em kernel of $\lambda$}, for which the action $S\la{\bar\lambda}\bar X=X/Y$ induced by $\lambda$ is by injective endomorphisms. One then proceeds to build an injective direct system $\frak X$ of copies of $\bar X$, with transition morphisms suitably induced by $\bar \lambda$, and to verify that the direct limit $X^*$ of $\frak X$ has a natural left $S$-action by automorphisms $S\la{}X^*$, that uniquely extends to a left $G$-action $G\la{\lambda^*}X^*$, the Ore localization of $ S\la{\lambda}X$. 
Let us remark that the action induced on the quotient over the kernel and the Ore localization satisfy suitable universal properties, that completely characterize them among the left actions by injective endomorphisms and automorphisms, respectively. In particular, both constructions are functorial. Furthermore, they produce exact functors (i.e., they both commute with restrictions and quotients of an action with respect to an invariant subgroup) and they preserve the algebraic entropy, i.e., in the above notation, $h_\alg(\lambda)=h_\alg(\bar \lambda)=h_\alg (\lambda^*)$.

\smallskip
Starting from the earliest stages after its introduction, the algebraic entropy was studied in connection with the topological entropy, by means of the Pontryagin duality. To make this idea more precise, let $S$ be a monoid, $K$ a compact Abelian group and $X=\dual K$ the (discrete) dual Abelian group of $K$.  Denote by $\Comp$ the full subcategory of $\CompG$ of compact Abelian groups. By the Pontryagin duality we then obtain a ring anti-isomorphism
\[\dual{(-)}\colon \End_{\Comp}(K)\to \End_\Ab(X),\quad\text{such that}\ \ \phi\mapsto\dual{\phi}.\]
So, starting with a right $S$-action $K\ra{\rho}S$, we get its dual left $S$-action $S\la{\dual\rho} X$ by letting
\[\dual\rho=\dual{(-)}\circ \rho.\]
Analogously, starting with a left $S$-action $S\la{\lambda} X$ on a discrete Abelian group $X$, we get a right $S$-action $K\ra{\dual\lambda} S$ on the compact Abelian group  $K=\dual X$.
By the Pontryagin duality (see \S\ref{subs:duality}), $\rho^{\wedge\wedge}$ is conjugated to $\rho$ and $\lambda^{\wedge\wedge}$ is conjugated to $\lambda$, so 
$h_\top(\rho^{\wedge\wedge}) = h_\top(\rho)$ and $h_\alg(\lambda^{\wedge\wedge}) = h_\alg(\lambda)$ (see \S\ref{subs_back_ent} and \S\ref{subs:top_ent}).

For a right $S$-action $K\ra{\rho}S$ on a compact Abelian group $K$, we say that {\em the Bridge Theorem holds for $\rho$} when the following equality holds: $$h_\top(\rho)=h_\alg(\dual\rho).$$

There are several instances of the Bridge Theorem in the literature; let us list here the most relevant to our situation: the first one was proved by M.\ D.\ Weiss~\cite{Weiss} for right $\N$-actions on totally disconnected compact Abelian groups (i.e., those for which the dual is a torsion Abelian group), Peters~\cite{P1} then verified the Bridge Theorem for right $\Z$-actions on compact metrizable Abelian groups. A completely different proof of Peters' result was given in~\cite{DGB}, also extending it to right $\N$-actions on arbitrary compact Abelian groups, while in~\cite{GB} Weiss' result was extended to semigroup actions. Finally, Peters' result was extended by Kerr and Li~\cite{li_kerr_book} to actions of countable amenable groups on compact metrizable Abelian groups. 
See also \cite{P2} and \cite{VBT} for $\Z$-actions and amenable  group actions, respectively, on general locally compact Abelian groups. 
Finally, let us mention that an instance of the Bridge Theorem for actions of sofic groups on compact metrizable Abelian groups has been proved by Liang~\cite{Liang}.

\smallskip
The second main goal of this paper is to prove the following extension of the known Bridge Theorems: 

\medskip\noindent
{\bf Bridge Theorem.} {\em Let $S$ be a cancellative and  right amenable  monoid, $K$ a compact Abelian group with a right $S$-action $K\ra{\rho} S$  and $X$ a discrete Abelian group with a left $S$-action $S\la{\lambda}X$. Then: 
\begin{enumerate}[(1)]
\item $h_\top(\rho)=h_\alg(\dual\rho)$; 
\item $h_\alg(\lambda)=h_\top(\dual\lambda)$.
\end{enumerate}}

The case of the Bridge Theorem when $S$ is a group is given in Theorem \ref{BT_for_invertible_actions}. This result is then extended to the general case in \S\ref{proofs}, by using the invariance of the topological and of the algebraic entropy under Ore colocalization and localization, respectively, and the fact that these two constructions are each other's dual with respect to the Pontryagin duality, that is, in the above notation, $\dual{(\rho^*)}=\lambda^*$ and $\dual{(\lambda^*)}=\rho^*$.

The proof  of Theorem \ref{BT_for_invertible_actions}, as well as that of (\text{\sc at$_\top$}), is heavily based on the uniform approach to entropy 
via the entropy of monoid actions on normed monoids {developed in \cite{DG_PC,DGB4} for $\N$-actions and then extended to the general case in \cite{tesi_simone,VBT}. This approach covers, beyond the algebraic and the various versions of the topological entropy,
also the measure entropy and many others (see  \cite{DGB4}). It is exposed with more details in \S\ref{AppA}.

\smallskip
As an application of the Addition Theorem for $h_\top$ and of the Bridge Theorem, we show that also $h_{\alg}$ satisfies an Addition Theorem for left $S$-actions $S\la{\lambda}X$ on discrete Abelian groups, so far known only under the hypotheses that either $X$ is torsion~\cite{DFGB} or $S$ is locally monotileable~\cite{DFGS}:

\bigskip\noindent
{\bf Algebraic Addition Theorem.} 
\emph{Let $S$ be a cancellative and right amenable monoid, $X$ an Abelian group,  $S\la{\lambda}X$ a left \mbox{$S$-action}  and $Y$ an $S$-invariant subgroup of $X$. Then, 
\[h_{\alg}(\lambda)=h_{\alg}(\lambda_{ Y})+h_{\alg}(\lambda_{X/Y}),\eqno{(\text{\sc at}_\alg)}
\] 
where $\lambda_Y$ and $\lambda_{X/Y}$ are the obvious left $S$-actions induced by $\lambda$ on $Y$ and $X/Y$, respectively.}

\bigskip
Unlike (\text{\sc at$_\top$}), (\text{\sc at}$_\alg$) fails if the group $X$ is not Abelian (a {simple} counterexample involving a $\Z$-action on a metabelian group can be found in \cite{GSp}).
 On the other hand, the first instance of (\text{\sc at}$_\alg$) for $\N$-actions on torsion Abelian groups was given in \cite{DGSZ}, while the general case for $\N$-actions on Abelian groups was settled in \cite{DGB}. Moreover, (\text{\sc at}$_\alg$) holds also for $\N$-actions on some special classes of non-Abelian groups \cite{GBS,GBST,Shloss}.

\medskip
The paper is organized as follows, starting with \S\ref{sec:background}, which contains preliminary results and definitions from Pontryagin-van Kampen duality, actions and entropy.

Then, \S\ref{sec_discrete} deals with the algebraic entropy for left actions of cancellative and right amenable monoids on discrete Abelian groups.
A reduction to actions by injective endomorphisms is carried out in \S\ref{reduction_to_injective}.  
In \S\ref{Ore_invariance} we describe the Ore localization of a left action, which is a functorial way to produce an action by automorphisms from any given action, and we establish the invariance of the algebraic entropy under Ore localization. Two different, much more categorical, interpretations of these constructions are then given in \S\ref{categorial_ore_loc_subs}.

The case of the topological entropy for right actions of cancellative and right amenable monoids on compact spaces and groups is discussed in \S\ref{sec_compact}. First, in \S\ref{subs_red_top_to_surj} we obtain a reduction of the computation of the topological entropy to the case of actions by surjective continuous self-maps.  Then in \S\ref{subs_dual_ore} we build the Ore colocalization of a right $S$-action on a compact space and we establish the invariance of the topological entropy under this construction. In \S\ref{categorical_coloc_subs} we give a categorical interpretation of the construction of the Ore colocalization.
 
 The aim of \S\ref{AppB} is to give a proof of (\text{\sc at$_\top$}) in \S\ref{ProofAT:top}. Before that, in \S\ref{N--} we study technical properties of the {minimum} cardinality of finite open covers of compact spaces related to projections and in \S\ref{evensec} we discuss special open covers of compact groups, that are used in the proof of \text{\sc at$_\top$}).

Finally, \S\ref{proofs} contains the proof of the Bridge Theorem and (\text{\sc at}$_\alg$). 

In the appendix, \S\ref{AppC}, we provide a self-contained proof of the Bridge Theorem for amenable group actions,  which is a particular case of the general result from \cite{VBT}  concerning amenable group actions on locally compact Abelian groups. 

\medskip
The readers who prefer to see the proof of (\text{\sc at$_\top$}) (resp., the Bridge Theorem) only for group actions can go directly to \S\ref{AppB} (resp., \S\ref{AppC}) keeping in mind that the background is recalled in \S\ref{sec:background}.

\subsection*{Acknowledgements}
We warmly thank Hanfeng Li for the useful discussion with the second and third named authors at the conference ``Entropies and soficity'' held in Lyon in January 2018, when he pointed out that a given cancellative right amenable monoid satisfies the left Ore condition and, therefore, it has an associated group of left fractions.

\section{Preliminary results and definitions}\label{sec:background}

Here we recall the necessary background. In \S\ref{subs:duality} we include some basic facts regarding the Pontryagin-van Kampen duality.
In \S\ref{amenable_implies_ore_subsection} we recall that a cancellative and right amenable monoid satisfies the left Ore condition and, therefore, it embeds in its group of left fractions. 
In \S\ref{representations} we introduce the formalism of left and right actions and
in \S\ref{AppA} the entropy for actions of right amenable cancellative monoids on normed monoids.
In \S\ref{subs_back_ent} and \S\ref{subs:top_ent}, we recall the definitions and few basic properties of the algebraic and the topological entropy, respectively. 

\subsection{The Pontryagin-van Kampen duality for locally compact Abelian groups}\label{subs:duality}

In what follows we denote by $\LCA$ the category of locally compact Abelian groups, with continuous group homomorphisms between them and we identify the categories $\Ab$ and $\Comp$ with the full subcategories of discrete (resp., compact) groups in $\LCA$. 
We denote by $\T=\R/\Z$ the torus group, which can be considered both as an object in $\Ab$, when taken with its discrete topology, but also as an object in $\Comp$, when taken with the topology inherited from the reals. It will be always clear from the context whether we take $\T$ as a discrete or as a compact group. 

Consider the contravariant functor $\dual{(-)}\colon \LCA \to \LCA$ that associates to $G \in \LCA$ the Abelian group of continuous homomorphisms $\dual G=\hom_\LCA(G,\T)$, which is a locally compact Abelian group when endowed with the compact-open topology. Furthermore, given a morphism $f\colon G\to H$ in $\LCA$, we define $\dual f\colon \dual H\to \dual G$ as follows:
\[
\dual f\colon \hom_\LCA(H,\T) \to  \hom_\LCA(G,\T),\quad\text{such that}\ \ \chi\mapsto \chi\circ f.
\]
The Pontryagin-van Kampen duality states that  for $G\in\LCA$ there is a natural topological isomorphism $\omega_G\colon G \to G^{\wedge\wedge}$, such that $\omega_G(x)(\chi)=\chi(x)$ for every $x\in G$ and $\chi\in\dual G$, i.e.,  $(-)^{\wedge}\colon\mathcal L\to\mathcal L$ is an involutive duality of  the category $\LCA$. Using this duality, we identify $G$ with $G^{\wedge\wedge}$ whenever needed or convenient.  As $\dual G$ is compact (resp., discrete) precisely when $G$ is discrete (resp., compact),  the restrictions of this functor to the categories $\Ab$ and $\Comp$ that, with a slight abuse of notation, we denote by the same symbol, give a duality between the categories $\Ab$ and $\Comp$. 

Given $G\in \LCA$, there is an anti-isomorphism between lattices of subobjects (i.e., closed subgroups of $G$ and  of $\dual G$):
\[
(-)^\perp\colon  \mathcal L (G) \longrightarrow  \L (\dual G), \quad\text{such that}\quad H\mapsto H^\perp =\{\chi\colon G\to \T: \chi(H)=0\}.
\] 
Furthermore, if $f\colon G\to G'$ is a morphism in $\LCA$ and $H\in \L (G)$, then $(f(H))^{\perp}=(\dual f)^{-1}(H^{\perp})$.
In particular, for $H=G$, this formula gives
\[\mathrm{Im}(f)^\perp=(f(G))^\perp=(\dual f)^{-1}(G^{\perp})=(\dual f)^{-1}(0)=\Ker(\dual f).\]
 See~\cite{Book} for these and other properties concerning the Pontryagin-van Kampen duality used in this paper.

\subsection{Right amenability and the left Ore condition}\label{amenable_implies_ore_subsection}

Denote by $\Pf(A)$ the family of all non-empty finite subsets of a set $A$.

A \emph{right F\o lner net} for a monoid $S$ is a net $\{F_i\}_{i\in I}$ in $\Pf(S)$, where $(I,\leq)$ is a directed set, such that, for every $s\in S$,
$$\lim_{i\in I}\frac{|F_is\setminus F_i|}{|F_i|}=0.$$
We say that a cancellative monoid $S$ is \emph{right amenable} if it admits a right F\o lner net \cite{Day2,N}.

\begin{lemma}\label{fol} 
Let $\{F_i\}_{i\in I}$ be a right F\o lner net for a cancellative monoid $S$ and take $\{s_i\}_{i\in I}\subseteq S$. Then, $\{s_iF_i\}_{i\in I}$ is a right F\o lner net for $S$.
\end{lemma}
\begin{proof}
Since $S$ is cancellative, $|s_i F_i|=|F_i|$ and $|s_iF_is\setminus s_iF_i|=|F_is\setminus F_i|$, for all $s\in S$ and $i\in I$. Hence, for each $s\in S$,
\[
\lim_{i\in I}\frac{|s_iF_is\setminus s_iF_i|}{|s_i F_i|}=\lim_{i\in I}\frac{|F_is\setminus F_i|}{|F_i|}=0, 
\]
where the second equality holds as $\{F_i\}_{i\in I}$ is right F\o lner. The fact that the limit on the left-hand side of the above formula goes to $0$ for all $s\in S$ tells us that also $\{s_iF_i\}_{i\in I}$ is right F\o lner, as desired.
\end{proof}

Recall that a cancellative monoid $S$ satisfies the {\em left Ore condition}, or that it is {\em left Ore}, if:
\begin{enumerate}[(LO)]
  \item for any pair of elements $s,t\in S$, the intersection $Ss\cap St$ is non-empty.
\end{enumerate}
The {\em right divisibility relation} on a monoid $S$ is the following partial preorder:
\begin{quotation}
for $s,s'\in S$ let $s\leq s'$ if and only if there exists $t\in S$ such that $s' = ts$.
\end{quotation}

Using the right divisibility relation, it is easy to see that the left Ore condition is equivalent to each of the following reformulations:

\begin{enumerate}
  \item[(LO$'$)]  $(S,\leq)$ is directed, that is, given $s,s'\in S$, there is $t\in S$ such that $t\geq s$ and $t\geq s'$;
  \item[(LO$''$)] the subset $Ss\subseteq S$ is cofinal in $(S,\leq)$ for all $s\in S$, that is, given $s,s'\in S$ there exists $t\in S$ such that $ts\geq s'$.
\end{enumerate}
In the following lemma we collect a couple of less obvious consequences of the left Ore condition:

\begin{lemma}\label{GO}
Let $S$ be a left Ore monoid and $s, s_1,\ldots,s_n\in S$. Then:
\begin{enumerate}[(1)]
    \item there exist $t_1,\ldots,t_n,t\in S$ such that $t_js_k=ts$ for every {$j=1,\ldots,n$};
    \item there exist $t_1,\ldots,t_n,t\in S$ such that $ts_j=t_js$ for every $j=1,\ldots,n$.
\end{enumerate}
\end{lemma}
\begin{proof}
(1) By (LO), there exist $l_1, k_1\in S$ such that $l_1s_1=k_1s$. Similarly, one can find $l_2,k_2\in S$ such that $l_2s_2=k_2k_1s$; in particular, $k_2l_1s_1=k_2k_1s$. Proceeding  by induction, we find $l_1,\ldots, l_n\in S$ and $k_1,\ldots,k_n\in S$ such that $l_js_j=k_jk_{j-1}\ldots k_1s$ for every {$j=1,\ldots,n$}. One concludes by letting  $t=k_n\ldots k_1$, and $t_j=k_n\ldots k_{j+1}l_j$.

\smallskip
(2) We first show that there exist $v_1,\ldots,v_n,t'_1,\ldots,t'_n\in S$ such that
\begin{equation}\label{GO_c}
(t'_j\ldots t'_1)s_j=v_js,\quad\text{for all $j=1,\ldots,n$}.
\end{equation}
Indeed, by (LO), there exist $t_1',v_1\in S$ such that $t_1's_1=v_1s$. Then there exist $t'_2,v_2\in S$ such that $t'_2t'_1s_2=v_2s$. Proceeding in this way by induction 
we find the last pair $v_n,t'_n\in S$ such that $(t'_n\ldots t'_1)s_n=v_ns$.
Put $t= t'_n\ldots t'_1$ and
\[u_j=  \begin{cases}
t_n'\ldots t'_{j+1}& \text{if $j< n$;}\\
1 &\text{if $n=1$.}
\end{cases}\]
Then,  $ts_j=u_j (t'_j\ldots t'_1)s_j=u_jv_js$, in view of \eqref{GO_c}. Letting $t_j= u_jv_j$, we deduce that $ts_j=t_js$, for every {$j=1,\ldots,n$}.
\end{proof}

If $S$ is a cancellative and left Ore monoid, it can be embedded in a group $G=S^{-1}S$ that we call  {\em group of left  fractions} of $S$  (see \cite[Proposition~1.26]{Pat}). Clearly $G$ is generated as a group by $S$.
Furthermore, it is well-known (see \cite[Proposition~1.23]{Pat}) that a cancellative and right amenable monoid $S$ is automatically left Ore. In the following lemma (which is \cite[Lemma~2.11]{DFGB}) we recall the useful fact that the right amenability of $S$ implies that its group of left fractions $G$ is amenable: 

\begin{lemma}\label{Lemma:ex-2.2}       
Let $G$ be a group and $S$ a submonoid of $G$ that generates $G$ as a group. 
\begin{enumerate}[(1)]
  \item If $S$ is right amenable (and, necessarily, cancellative), then $G$ is amenable.
  \item If $\{F_i\}_{i\in I}$ is a right F\o lner net for $S$, then it is also a right F\o lner net for $G$.
\end{enumerate}
\end{lemma}

Next we endow $G$ with the preorder relation $\leq_S$, called the \emph{$S$-preorder} on $G$, that makes it into a directed set:  
\begin{quotation}
$g_1\leq_Sg_2$ if and only if $g_2^{-1}g_1\in S$, for all $g_1,g_2\in G$.
\end{quotation}
In other words, for $g_1,g_2\in G$, $g_1\leq_Sg_2$ if and only if there exists $s\in S$ such that $g_1=g_2s$.

\begin{lemma}
Let $S$ be a cancellative left Ore monoid and $G=S^{-1}S$ its group of left fractions. Then, the preordered set $(G,\leq_S)$ is directed. 
\end{lemma}
\begin{proof}
Consider two elements $g_1,g_2\in G$ and write them as $g_1=s_1^{-1}t_1$ and $g_2=s_2^{-1}t_2$ for suitable $s_1,s_2,t_1,t_2\in S$ (this is always possible because $G=S^{-1}S$ is the group of left fractions of $S$). Using the left Ore condition, choose $u_1$ and $u_2\in S$ such that $u_1s_1=u_2s_2$ and let $g=s_1^{-1}u_1^{-1}=s_2^{-1}u_2^{-1}\in G$. Then,
\[g^{-1}g_1=(u_1s_1)(s_1^{-1}t_1)=u_1t_1\in S\quad \text{and}\quad g^{-1}g_2=(u_2s_2)(s_2^{-1}t_2)=u_2t_2\in S.\] 
Hence, we both have that $g_1\leq_S g$ and $g_2\leq_S g$, showing that $(G,\leq_S)$ is directed. 
\end{proof}

The main reason to use two different symbols for the right divisibility relation on $S$ and for the $S$-preorder on $G$ is that these two preorders do not coincide on $S$: actually, they are opposite to each other,  that is, $s_1\leq s_2$ in $S$ if and only if $s_2\leq_Ss_1$ in $G$; they coincide precisely when $S$ is a group 
(that is, the two preorders coincide on $S$ if and only if they are both trivial, if and only if $S=G$). In the following sections both preordered sets will turn out to be very important.

\subsection{Left and right representations}\label{representations}

Given a monoid $S$ and a category $\C$, we denote by $\lre(S,\C)$ and $\rre (S,\C)$ the categories of left and right representations in $\C$, respectively. 

The objects of $\lre(S,\C)$ are the left $S$-actions, that is, the pairs $(X,\lambda)$, denoted by $S\la{\lambda}X$ in the sequel, where $X$ is an object of $\C$ and $\lambda\colon S\to \End_\C(X)$ is a monoid homomorphism. A morphism $f\colon(X,\lambda)\to (X',\lambda')$ in $\lre (S,\C)$  is a morphism $f\colon X\to X'$ in $\C$ that is \emph{$S$-equivariant}, i.e., such that the following diagram commutes, for each $s\in S$:
\[\xymatrix@C=40pt{
X\ar[r]^{f}\ar[d]_{\lambda_s}&X'\ar[d]^{\lambda'_s}\\
X\ar[r]_{f}&X'.}\] 
If there exists an $S$-equivariant isomorphism between $X$ and $X'$, its inverse is automatically $S$-equivariant, and we say that the two left $S$-actions $S\la{\lambda}X$ and $S\la{\lambda'}X'$ are {\em isomorphic} or {\em conjugated}. In other words, two left $S$-actions are conjugated if and only if they are isomorphic in the category $\lre(S,\C)$. 

The monomorphisms in $\lre(S,\C)$ are exactly the $S$-equivariant morphisms that are monic in $\C$. In particular, the subobjects of a left $S$-action $S\la{\lambda} X$ are all of the form $\iota\colon Y \to X$, where $\iota$ is a monomorphism in $\C$, and $Y$  is {\em $S$-invariant} in $X$, that is,  for each $s\in S$, the composition $\lambda_s\circ \iota$ factors through $\iota$, i.e., there exists a (necessarily unique) morphism $(\lambda_Y)_s\colon Y\to Y$ such that $\lambda_s\circ \iota=\iota\circ (\lambda_Y)_s$. In this case, $Y$ is endowed with the unique possible action that makes $\iota$ into an $S$-equivariant morphism: 
\[S\overset{\lambda_{Y}}{\curvearrowright}Y\quad\text{with}\quad \lambda_Y(s)=(\lambda_Y)_s,\text{ for all $s\in S$}.\]

Objects, morphisms, conjugations and subobjects in the category $\rre(S,\C)$ are defined similarly. 

\smallskip
In the sequel, the category of left representations $\lre(S,\C)$ will only appear with $\C = \Ab$ and with $\C$ the category of normed monoids introduced in \S\ref{AppA} below. On the other hand, $\rre(S,\C)$ will appear with $\C = \CompS$, $\C = \CompG$ and $\C = \Comp$.

\begin{remark}
When $S$ is a cancellative and right amenable monoid, with $G=S^{-1}S$ its group of left fractions, one can consider $\lre(G, \Ab)$ simply as a full subcategory of $\lre(S, \Ab)$, by taking the restriction of a $G$-action to an $S$-action. In \S\ref{sec_discrete} we build, for each $S\la{\lambda} X \in \lre(S, \Ab)$, a canonical $G$-action $G\la{\lambda^*} X^* \in \lre(G, \Ab)$, functorially associated with $\lambda$, called the Ore localization of $\lambda$ (see Definition \ref{def_left_Ore_loc}). For a deeper analysis, using a heavier impact of categorical algebra, see \S\ref{categorial_ore_loc_subs}.

Similarly, $\rre(G,\CompS)$ (resp., $\rre(G,\CompG)$, $\rre(G,\Comp)$) can be viewed as a full subcategory of $\rre(S,\CompS)$ (resp., $\rre(S,\CompG)$, $\rre(S,\Comp)$). In \S\ref{sec_compact} we build  the Ore colocalization of a given right $S$-action $K\ra{\rho} S$ in $\rre(S, \CompS)$. This is a canonical right $G$-action $K^*\ra{\rho^*} G$ in $\rre(G, \CompS)$, functorially associated with $\rho$ (see Definition \ref{def_k_star}). As for the Ore localization, we give a categorical interpretation of the Ore colocalization in \S\ref{categorical_coloc_subs}.
\end{remark}

\subsection{Entropy of actions on normed monoids}\label{AppA}

Here we recall the notion of entropy for actions of cancellative and right amenable monoids on normed monoids, and the notions of asymptotic equivalence, introduced in \cite{tesi_simone} (see also \cite{VBT}), and {that of weak asymptotic equivalence, that} conveniently weaken the more rigid condition of being conjugated. Let $S$ be a fixed cancellative and right amenable monoid for the rest of this section.

\subsubsection{The category of normed monoids}

A {\em normed monoid} is a pair $(M,v)$, where $(M,+,0)$ is a commutative monoid and $v\colon M\to \R_{\geq0}$ is a function. 

Given a normed monoid $(M,v)$, we say that the norm $v$ is:
\begin{enumerate}[\rm --] 
      \item {\em monotone} provided $v(x)\leq v(x+y)$, for all $x$, $y\in M$; 
      \item {\em subadditive} provided $v(x+y)\leq v(x)+v(y)$, for all $x$, $y\in M$.
\end{enumerate}

A {\em homomorphism of normed monoids} $f\colon (M_1,v_1)\to (M_2,v_2)$ is a monoid homomorphism $f\colon M_1\to M_2$ such that $v_2(f(m))\leq v_1(m)$ for all $m\in M_1$. Furthermore, $f$ is an {\em isomorphism} if it is a monoid isomorphism and $v_2(f(m))= v_1(m)$ for all $m\in M_1$. We denote by $\mathfrak M$ the category of normed monoids defined in this way. 

\medskip
The algebraic and the topological entropy are based on the following normed monoids (see \S\ref{subs_back_ent} and \S\ref{subs:top_ent}). 

\begin{example}\label{nmexalg} 
Let $X$ be a discrete Abelian group.
\begin{enumerate}[(1)]
   \item Denote by $\Pf^0(X)$ the family of all finite subsets of $X$ that contain $0$. The pair $(\Pf^0(X),+)$ is a commutative monoid, with norm $v_{\frak F}$  defined by $v_{\frak F}(H)=\log|H|$, for all $H\in\Pf^0(X)$. The norm $v_{\frak F}$ is both monotone and subadditive. 
   \item The submonoid $\mathfrak F(X)$ of $\Pf^0(X)$, consisting of all finite symmetric subsets of $X$ that contain $0$, will also be used in the sequel; $\mathfrak F(X)$ is cofinal in $\Pf^0(X)$ with respect to the inclusion $\subseteq$. 
   \item Another submonoid of $\Pf^0(X)$ is $\L^{\fin}(X)$ consisting of all finite subgroups of $X$; clearly, $\L^{\fin}(X)\subseteq \mathfrak F(X)$ and $\L^{\fin}(X)$ is cofinal in $\Pf^0(X)$ (if and only if it is cofinal in $\mathfrak F(X)$) precisely when $X$ is torsion.
\end{enumerate}
\end{example}

\begin{example}\label{nmextop} 
Let $K$ be a compact space. Denote by $\cov(K)$ the family of all open covers of $K$ and recall that for $\mathcal U, \mathcal V \in \cov(K)$, 
$\mathcal U \vee \mathcal V=\{U\cap V : U\in\U, V\in\V\}\in\cov(K)$. For a subset $B$ of $K$ and $\mathcal U\in\cov(K)$, let 
\[
N_B(\mathcal U)=\min\{n\in\N_+\colon\mathcal U\ \text{admits a subcover of $B$ of size $n$}\};
\]
when $B=K$ we simply write $N(\mathcal U)$ instead of $N_K(\mathcal U)$. 

Then $(\cov(K),\vee)$ is a commutative monoid with a monotone and subadditive norm $v_{\cov}$ given by $v_{\cov}(\mathcal U) = \log N(\mathcal U)$, for all $\mathcal U \in \cov(K)$.
\end{example}

\begin{example}\label{nmextop2}
Now assume that $K$ is a compact group and $\mu$ is the Haar measure on $K$ (so that $\mu(K)=1$).
\begin{enumerate}[(1)]
   \item Let ${\frak U}(K)$ be the family of all symmetric compact neighborhoods of $1$ in $K$. Then the pair $({\frak U}(K),\cap)$ is a commutative monoid, with norm $v_{\frak U}$ defined by $v_{\frak U}(U)=-\log\mu(U)$, for each $U\in{\frak U}(K)$. Clearly, $v_{\frak U}$ is monotone, but not subadditive in general.
   \item Similarly, the larger family of {\em all} compact neighborhoods of $1$ in $K$ with $\cap$ as operation and $v_{\frak U}$ as above, is a normed monoid.
Then $\mathfrak U(K)$ is a submonoid of this monoid, cofinal with respect to $\supseteq$.
Our preference for $\mathfrak U(K)$ is motivated by its application  in \S\ref{AppC}, where the symmetry of its elements is needed. 
   \item One can also consider the submonoid $\L^{\op}(K)$ of $\mathfrak U(K)$ consisting of all open subgroups of $K$.
When $U\in\L^{\op}(K)$, clearly $U$ has finite index $[K:U]$ and $\mu(U)=1/[K:U]$, so $v_{\frak U}(U)=\log[K:U]$. In particular, when restricted to $\L^{\op}(K)$, the norm $v_{\frak U}$ is also subadditive. 
   \item  With $X=\dual K$, the map $\L^\fin(X)\to \L^{\op}(K)$ defined by $F\mapsto F^\perp$ is a monoid isomorphism (by the Pontryagin duality -- see \S\ref{subs:duality}) and $\log|F|=\log[K:F^\perp]$ for every $F\in\L^\fin(X)$; that is, we have an isomorphism of normed monoids $\L^\fin(X)\cong \L^{\op}(K)$.
   \item By identifying each $U\in\L^{\op}(K)$ with $\U_{ K}[U]=\{xU:x\in K\}\in\cov(K)$, $\L^{\op}(K)$ can also be considered as a normed submonoid of $\cov(K)$, since, for every $U,V\in\L^{\op}(K)$, $\U_{ K}[U\cap V]=\U_K[U]\vee\U_K[V]$ and $[K:U]=N(\U_K(U))$.
   \item The compact group $K$ is profinite precisely when $\L^{\op}(K)$ is cofinal in $\mathfrak U(K)$ or, equivalently, when $\L^{\op}(K)$ is cofinal in $\cov(K)$ (see Lemma~\ref{tech_AT}(1)).
\end{enumerate}
\end{example}

\subsubsection{Asymptotic domination}

Let $S$ be a cancellative and right amenable monoid. A left $S$-action $S\la{\alpha} M$ on a normed monoid $M$ is a monoid homomorphism $\alpha\colon S\to \End_{\mathfrak M}(M)$ (where $\End_{\mathfrak M}(M)$ is the monoid of all endomorphisms of normed monoids $M\to M$). For $x\in M$ and $F=\{f_1,\dots,f_k\}\subseteq S$, we let
\[T_F(\alpha,x)=\alpha_{f_1}(x)+\ldots+\alpha_{f_k}(x).\]

Recall from \S\ref{representations} that two left  $S$-actions $S\la{\alpha_1} M_1$ and $S\la{\alpha_1} M_1$ on the two normed monoids $(M_1,v_1)$ and $(M_2,v_2)$ are conjugated if there exists an $S$-equivariant isomorphism of normed monoids $f\colon M_1\to M_2$. Now we introduce three notions of ``equivalence'', all weaker than conjugation, between actions on normed monoids:

\begin{definition}
Let $S\la{\alpha_1} M_1$ and $S\la{\alpha_2} M_2$ be two left $S$-actions, where $M_j=(M_j,v_j)$ is a normed monoid for {$j=1,2$}. We say that:
\begin{enumerate}[(i)]
   \item $\alpha_2$ {\em dominates} $\alpha_1$  if, for each $x\in M_1$, there exists $y\in M_2$ such that, for each  $F\in\Pf(S)$,
\[v_1(T_F(\alpha_1,x))\leq v_2(T_F(\alpha_2,y));\]
   \item $\alpha_2$ {\em asymptotically dominates} $\alpha_1$  if, for each $x\in M_1$, there are two sequences $\{y_n\}_{n\in\N}$ in $M_2$ and $\{\varepsilon_n\}_{n\in \N}$ in $\R$ such that $\lim_{n\to\infty}\varepsilon_n=0$ and, for each $n\in\N$ and $F\in\Pf(S)$, 
\[\frac{v_1(T_F(\alpha_1,x))}{|F|}\leq \frac{v_2(T_F(\alpha_2,y_n))}{|F|}+\varepsilon_n;\]
   \item $\alpha_2$ \emph{weakly asymptotically dominates} $\alpha_1$ if, for every right F\o lner net $\s=\{F_i\}_{i\in I}$ for $S$ and for every $x\in M_1$, there exist a sequence $\{y_n\}_{n\in\N}$ in $M_2$ and a sequence $\{f_n\}_{n\in\N}$ of functions $f_n\colon\R_{\geq0} \to\R_{\geq0}$ such that:
\begin{itemize}
   \item[(iii$_1$)] $\{f_n\}_{n\in\N}$ converges uniformly to $\id_{\R_{\geq0}}\colon \R_{\geq0}\to\R_{\geq0}$ on every bounded interval $[0,C]$, that is, for every $\varepsilon>0$ there exists $\bar n\in \N$, such that $|f_n(r)-r|<\varepsilon$ for every $n\geq \bar n$ in $\N$ and for every $r\in [0,C]$; 
   \item[(iii$_2$)]  there exists $j\in I$ such that, for every $i\geq j$ in $I$ and for every $n\in\N$, 
$$
\frac{v_1(T_{F_i}(\alpha_1,x))}{|F_i|}\leq f_n\left(\frac{v_2(T_{F_i}(\alpha_2,y_n))}{|F_i|}\right);
$$
\end{itemize}
   \item $\alpha_1$ is {\em  equivalent} (resp., \emph{asymptotically equivalent}, \emph{weakly asymptotically equivalent}) to $\alpha_2$ if the two actions dominate (resp., asymptotically dominate, weakly asymptotically dominate) each other.
\end{enumerate}
\end{definition}

Clearly, conjugated actions are equivalent, and equivalent actions are asymptotically equivalent. Here are two examples when domination arises in a rather natural way. 

\begin{example}\label{exa:domini} 
Let $S\la{\alpha_1} M_1$ and $S\la{\alpha_2} M_2$ be two left $S$-actions on normed monoids, and assume that $f\colon M_1 \to M_2$ is an $S$-equivariant morphism in $\mathfrak M$.
\begin{enumerate}[(1)] 
  \item If $f$ is an embedding (i.e., $f$ is injective and $f\colon M_1 \to f(M_1)$ is an isomorphism in $\mathfrak M$), then $\alpha_2$ {\em dominates} $\alpha_1$ (to check it,  for $x \in M_1$ put $y=f(x))$. 
  \item If $f$ is surjective, then  $\alpha_1$ {\em dominates} $\alpha_2$ (to check it, for $x \in M_2$ pick $y\in M_1$ such that $x=f(y))$. 
\end{enumerate}
\end{example}

The following is another natural example of an equivalence.

\begin{example}\label{cofinal}
Let $S\la{\alpha_2} M_2$ be a left $S$-action on a normed monoid $M_2=(M_2,v_2)$, 
{$M_1$ an $S$-invariant submonoid of $M_2$, endowed with the norm $(v_2)_{\restriction M_1}$, and $S\la{\alpha_1} M_1$ the restriction of the action $\alpha_2$ to $M_1$. According to Example~\ref{exa:domini}(1),  $\alpha_2$ dominates $\alpha_1$. }

Now assume that $M_2$ is preordered, with the norm $v_2$ monotone with respect to the preorder (i.e., if $x\leq y$ in $M_2$, then $v_2(x)\leq v_2(y)$). If $M_1$ is cofinal in $M_2$, then also $\alpha_1$ dominates $\alpha_2$, and so we can conclude that $\alpha_1$ and $\alpha_2$ are equivalent. This permits to restrict from a normed monoid to a convenient cofinal $S$-invariant submonoid.
\end{example}

\begin{lemma}  Let $S\la{\alpha_1} M_1$ and $S\la{\alpha_2} M_2$ be two left $S$-actions on normed monoids. If $\alpha_2$ asymptotically dominates $\alpha_1$, then $\alpha_2$ weakly asymptotically dominates $\alpha_1$.
\end{lemma}
\begin{proof} 
Let $x\in M_1$. There exist two sequences $\{y_n\}_{n\in\N}$ in $M_2$ and $\{\varepsilon_n\}_{n\in\N}$ in $\R_{\geq0}$ such that $\lim_{n\to \infty}\varepsilon_n=0$ and, for every $F\in\Pf(S)$, 
$$ v_1(T_F(\alpha_1,x))\leq v_2(T_F(\alpha_2,y_n))+\varepsilon_n|F|.$$
For every $n\in\N$, define $f_n\colon\R_{\geq0}\to\R_{\geq0}$ by $f_n(r)=r+\varepsilon_n$ for every $r\in\R_{\geq0}$.
Then $f_n$ converges uniformly to $\id_{\R_{\geq0}}$ on all of $\R_{\geq0}$.
If $\s=\{F_i\}_{i\in I}$ is a right F\o lner net for $S$, then for every $i\in I$,
\[\frac{v_1(T_{F_i}(\alpha_1,x))}{|F_i|}\leq \frac{v_2(T_{F_i}(\alpha_2,y_n))+\varepsilon_n|F_i|}{|F_i|} = f_n\left(\frac{v_2(T_{F_i}(\alpha_2,y_n))}{|F_i|}\right).\qedhere\]
\end{proof}

\subsubsection{Entropy and its properties}

\begin{definition}[\cite{DGB4,VBT}] \label{Def:h}
Let $M=(M,v)$ be a normed monoid, $S\la{\alpha} M$ a left $S$-action and fix a right F\o lner net $\s=\{F_i\}_{i\in I}$ for $S$. The {\em $\s$-entropy} of $\alpha$ at a given $m\in M$ is 
\[H(\alpha,\s,m)=\limsup_{i\in I}\frac{v(T_{F_i}(\alpha,m))}{|F_i|}.\]
The {\em $\s$-entropy} of $\alpha$ is $h(\alpha,\s)=\sup_{m\in M}H(\alpha,\s,m)$.
\end{definition}

If the norm $v$ is also subadditive, then our definition of entropy can be sensibly improved:

\begin{lemma}\label{reallim} 
Let $M=(M,v)$ be a normed monoid with $v$ monotone and subadditive, and let $S\la{\alpha} M$ be a left $S$-action. For each $m\in M$ and each right F\o lner net $\s=\{F_i\}_{i\in I}$ for $S$, the limit superior defining $H(\alpha,\s,m)$ is a limit, which does not depend on the choice of $\s$.
\end{lemma}
\begin{proof}
 For $m\in M$, consider the function $f_m\colon \Pf(S)\to \R_{\geq0}$ such that $f_m(F)=v(T_F(\alpha,m))$. Each $f_m$ has the following properties, which allow one to conclude by \cite[Corollary~1.2]{CCK}: 
\begin{enumerate}[\rm --]
  \item {\em subadditivity}, that is, given $F,F'\in \Pf(S)$, we have:
\begin{align*}
f_m(F\cup F')&=v(T_{F\cup F'}(\alpha,m))\\
&=v(T_{F}(\alpha,m)+T_{F'\setminus F}(\alpha,m))\\
&\leq v(T_{F}(\alpha,m))+v(T_{F'\setminus F}(\alpha,m))&\text{by subadditivity of $v$;}\\
&\leq v(T_{F}(\alpha,m))+v(T_{F'}(\alpha,m))=f_m(F)+f_m(F')&\text{by monotonicity of $v$.}
\end{align*}
   \item {\em left subinvariance}, that is, given $F\in\Pf(S)$ and $s\in S$,
\[f_m(sF)=v(T_{sF}(\alpha,m))=v(\alpha_s(T_F(\alpha,m)))\overset{(*)}\leq v(T_{F}(\alpha,m))=f_m(F),\]
where $(*)$ follows since $\alpha_s$ is a homomorphism of normed monoids.
\qedhere
\end{enumerate}
\end{proof}


%

For normed monoids $M_1=(M_1,v_1)$, $M_2=(M_2,v_2)$ and  left $S$-actions $S\overset{\alpha_1}\curvearrowright M_1$, $S\overset{\alpha_2}\curvearrowright M_2$, define the coproduct $M_1\oplus M_2$ with the norm $v_1\oplus v_2$ such that $(v_1\oplus v_2)(x,y)=v_1(x)+v_2(y)$ for every $(x,y)\in M_1\oplus M_2$, and the left $S$-action 
$$S\longla{\alpha_1\oplus\alpha_2} M_1\oplus M_2,\quad \text{such that, for every $s\in S$, $(\alpha_1\oplus\alpha_2)_s(x,y)=((\alpha_1)_s(x),(\alpha_2)_s(y))$ for every $(x,y)\in M_1\oplus M_2$.}$$
If both $v_1$ and $v_2$ are monotone (resp., subadditive), then so is $v_1\oplus v_2$. 
In the proof of {(\text{\sc at$_\top$})} we use the following weak Addition Theorem (we omit its straightforward proof). 

\begin{proposition}\label{wAT} 
In the above notation, if $v_1$ and $v_2$ are both monotone and subadditive, then \[h(\alpha_1\oplus\alpha_2,\s)=h(\alpha_1,\s)+h(\alpha_2,\s)\] for any right F\o lner net $\s$ for $S$.
\end{proposition}

The key to the proof of both the Topological Addition Theorem and the Bridge Theorem is the following:

\begin{proposition}\label{wad}\label{conjumon}
Let $M_1=(M_1,v_1)$, $M_2=(M_2,v_2)$ be normed monoids, $S\overset{\alpha_1}\curvearrowright M_1$, $S\overset{\alpha_2}\curvearrowright M_2$ left $S$-actions, {and $\s=\{F_i\}_{i\in I}$ a right F\o lner net for $S$.} If $S\la{\alpha_2} M_2$ weakly asymptotically dominates $S\la{\alpha_1} M_1$, then $h(\alpha_1,\s)\leq h(\alpha_2,\s)$.
In particular, $h(\alpha_1,\s)= h(\alpha_2,\s)$ whenever $\alpha_1$ and $\alpha_2$ are weakly asymptotically equivalent.
\end{proposition}
\begin{proof} If $h(\alpha_2,\s)= \infty$, we are done, so assume $h(\alpha_2,\s)<\infty$. 
 We verify that for every $x\in M_1$, $H(\alpha_1,\s,x)\leq h(\alpha_2,\s)$. Let 
$\{y_n\}_{n\in \N}$, $\{f_n\}_{n\in\N}$ and $j\in I$ be as in the definition. Then, for every $n\in\N$, and $j\leq i \in I$
\begin{equation}\label{Eq:June26}
 \frac{v_1(T_{F_i}(\alpha_1,x))}{|F_i|} \leq  f_n\left(\frac{v_2(T_{F_i}(\alpha_2,y_n))}{|F_i|}\right)
\end{equation}
Fix an arbitrary $\varepsilon>0$ and put $C = h(\alpha_2,\s) + 1$. There exists $m\in \N$ such that 
\begin{equation}\label{Eq:June26**}
|f_n(r)-r|<\varepsilon \ \mbox{ for every }  n\geq m  \ \mbox{ and } r\in [0,C].  
\end{equation}
 Since 
$$
 \lim_{j\leq i \in I} \frac{v_2(T_{F_i}(\alpha_2,y_m))}{|F_i|} = H(\alpha_2,\s,y_m) \leq h(\alpha_2,\s) < h(\alpha_2,\s) + 1 = C, 
 $$
there exists $j_m \in I$ such that
\begin{equation}\label{Eq:June26*}
\frac{v_2(T_{F_i}(\alpha_2,y_m))}{|F_i|} \leq C \ \mbox{ for every } \ i \geq j_m. 
\end{equation}
Now pick a $j_0$ such that $j_0 \geq j$ and $j_0\geq j_m$. Then, combining (\ref{Eq:June26}), (\ref{Eq:June26**}) and (\ref{Eq:June26*})  we deduce that
$$ \frac{v_1(T_{F_i}(\alpha_1,x))}{|F_i|} \leq  \frac{v_2(T_{F_i}(\alpha_2,y_m))}{|F_i|} + \varepsilon \ \mbox{ for every } \ i \geq j_0.$$
After taking limits, this gives 
$$
H(\alpha_1,\s,x) \leq H(\alpha_2,\s,y_m)+\varepsilon \leq h(\alpha_2,\s) +\varepsilon.
$$
Since $\varepsilon>0$ was chosen arbitrarily, we conclude that $H(\alpha_1,\s,x) \leq h(\alpha_2,\s)$.
As this holds for each $x\in M_1$, we get the desired inequality $h(\alpha_1,\s)\leq h(\alpha_2,\s)$. The second assertion obviously follows from the first one.
\end{proof}

Example \ref{exa:domini} and the above proposition give: 

\begin{corollary}\label{Mono:mono}  With $S\la{\alpha_1} M_1$ and $S\la{\alpha_2} M_2$ as above, and $f\colon M_1 \to M_2$ an $S$-equivariant morphism in $\mathfrak M$, for $\s$ a right F\o lner net for $S$,
\begin{enumerate}[(1)]
  \item $h(\alpha_1,\s)\leq h(\alpha_2,\s)$, if $f$ is an embedding; 
  \item $h(\alpha_1,\s)\geq h(\alpha_2,\s)$, if $f$ is surjective. 
\end{enumerate}
\end{corollary}

\subsection{Background on the algebraic entropy}\label{subs_back_ent}

In this subsection we recall the definition and some of the basic properties of the algebraic entropy $h_\alg(\lambda)$ of the left action $S\la{\lambda} X$ of the cancellative and right amenable monoid $S$ on a discrete Abelian group $X$. 

Given $F\in\Pf(S)$ and $E\in\Pf^0(X)$, we let 
\[T_{F\,}(\lambda,E)=\sum_{s\in F}\lambda_s(E).\]

\begin{definition}\label{def:entropy}
For $E\in\Pf^0(X)$, the \emph{algebraic entropy of $\lambda$ with respect to $E$} is 
\begin{equation}\label{def_partial_h_alg}
H_{\alg}(\lambda,E)=\lim_{i\in I}\frac{\log|T_{F_i}(\lambda,E)|}{|F_i|}.
\end{equation}
The \emph{algebraic entropy of $\lambda$} is $h_{\alg}(\lambda)=\sup\{H_{\alg}(\lambda,E): E\in \Pf^0(X)\}.$  
\end{definition}

Here is an alternative way to arrive at this notion which allows us to easily obtain many of the properties of $h_{\alg}$. 

\begin{example}\label{Gactionsexalg}
Let $X$ be a discrete Abelian group.
\begin{enumerate}[(1)]
  \item The assignment $X \mapsto \Pf^0(X)$ from Example~\ref{nmexalg} can be completed to a covariant functor $\mathcal P\colon \Ab \to \mathfrak M$ by letting $\mathcal P(f)(E) = f(E)$ for any morphism $f\colon X \to Y$ in $\Ab$ and $E \in \Pf^0(X)$. If $j\colon Y\to X$ is a subgroup embedding in $\Ab$, 
then the map $\mathcal P(j)\colon \mathcal P(Y) \to \mathcal P(X)$ is an embedding in $\mathfrak M$, while the quotient map $q\colon X \to X/Y$ gives rise to a surjective morphism $\mathcal P(q) \colon \mathcal P(X) \to \mathcal P(X/Y)$ in $\mathfrak M$. 

  \item For a left $S$-action $S\la{\lambda} X$, the functor $\mathcal P$ induces a left $S$-action
$$S\la{\lambda_{\mathcal P}} \Pf^0(X),\quad\text{ with }\ (\lambda_{\mathcal P})_s(E)=\lambda_s(E)\ \text{for every $s\in S$ and $E\in\Pf^0(X)$}.$$
For any $F\in\Pf(S)$ and $E\in\Pf^0(X)$, $$T_F(\lambda_{\mathcal P},E)=\sum_{s\in F}\lambda_s(E)=T_F(\lambda,E).$$
So, just by definition, for a given right F\o lner net $\s$ for $S$, $H_{\alg}(\lambda,E)=H(\lambda_{\mathcal P},\s ,E)$ and $h_\alg(\lambda)=h(\lambda_{\mathcal P},\s)$.

Hence, the existence of the limit in \eqref{def_partial_h_alg} and its independence on the choice of $\s$ follow from Lemma~\ref{reallim}.
This freedom in the choice of the right F\o lner net will be fundamental in the proofs of the following sections. 

   \item The submonoid $\mathfrak F(X)$ of $\Pf^0(X)$ is  $S$-invariant, so we can consider the restriction $S\la{\lambda_{\frak F}}\mathfrak F(X)$ of $\lambda_\mathcal P$ to $\mathfrak F(X)$.
By Example~\ref{cofinal}, $\lambda_{\frak F}$ and $\lambda_{\mathcal P}$ are equivalent, since $\mathfrak F(X)$ is cofinal in $\Pf^0(X)$. 
Hence, for any right F\o lner net $\s$ for $S$, 
$$h_\alg(\lambda)=h(\lambda_{\mathcal P},\s)=h(\lambda_{\frak F},\s)$$
by Proposition~\ref{wad}.
Therefore, for the computation of $h_\alg$ one can use $\frak F(X)$ instead of the whole $\Pf^0(X)$, namely, $h_{\alg}(\lambda)=\sup\{H_{\alg}(\lambda,E): E\in \frak F (X)\}.$
    
    \item Also the submonoid $\L^{\fin}(X)$ of $\Pf^0(X)$ is $S$-invariant. Since $X$ is torsion precisely when $\L^{\fin}(X)$ is cofinal in $\Pf^0(X)$, in this case one can use $\L^{\fin}(X)$ for the computation of the algebraic entropy, that is, $h_\alg(\lambda)=h((\lambda_\mathcal P)_{\restriction \L^\fin(X)},\s)$ for a right F\o lner net $\s$ for $S$; in other terms, $h_{\alg}(\lambda)=\sup\{H_{\alg}(\lambda,E): E\in \L^{\fin} (X)\}$. 
\end{enumerate}
\end{example}

\begin{remark}\label{RemarkMay16}
Assume that $G$ is the group of left fractions of $S$ and $G\la{\lambda} X$ is a left $G$-action on the discrete Abelian group $X$. By Lemma \ref{Lemma:ex-2.2}, $G$ is amenable and every right F\o lner net for $S$ is also a right F\o lner net for $G$. 
Therefore, $h_{\alg}(\lambda) = h_{\alg}(\lambda_{\restriction S})$. 
\end{remark}

In case a subgroup $Y\leq X$ is $S$-invariant, that is, $\lambda_s(Y)\leq Y$ for all $s\in S$, one obtains the following induced left $S$-actions 
\[S\overset{\lambda_{Y}}{\curvearrowright}Y\quad\text{and}\quad S\overset{\lambda_{X/Y}}{\curvearrowright}X/Y.\]
Both the inclusion $Y\to X$ and the projection $X\to X/Y$ are examples of $S$-equivariant homomorphisms. 

For what follows we just need the following  general facts about $h_{\alg}$. We refer to~\cite{DFGB} for the proof of these and other general properties  of the algebraic entropy, a short alternative proof along the same lines of that of Proposition~\ref{mono-top} (see also Proposition~\ref{properties_h_top}) can be used as well.

\begin{proposition}\label{mon}
In our setting, the following statements hold true:
\begin{enumerate}[(1)]
   \item if $S\la{\lambda'} X'$ is a left action conjugated to $S\la{\lambda} X$, then $h_{\alg}(\lambda')=h_{\alg}(\lambda)$;
   \item if $Y$ is an $S$-invariant subgroup of $X$, then $h_{\alg}(\lambda)\geq \max\{h_{\alg}(\lambda_{Y}),\,h_{\alg}(\lambda_{X/Y})\}.$
\end{enumerate}
\end{proposition}

The first property says that the algebraic entropy is an invariant for left $S$-actions on Abelian groups, while the second one is a monotonicity property with respect to 
taking invariant subgroups and quotients, that can be seen as consequences of the Algebraic Addition Theorem.

\subsection{Background on the topological entropy}\label{subs:top_ent}


In this subsection we recall the needed background about the topological entropy $h_\top(\rho)$ of a right action $K\ra{\rho} S$ of the cancellative and right amenable monoid $S$ on a compact space $K$. 

Given two open covers $\mathcal U=\{U_i\}_{i\in I}$ and $\mathcal V=\{V_j\}_{j\in J}$ of $K$, one says that $\mathcal V$ \emph{refines} $\mathcal U$, denoted by $\mathcal V\succ\mathcal U$, if for every $j\in J$ there exists $i_j\in I$ such that $V_j\subseteq U_{i_j}$. 
For two given open covers $\mathcal U,\mathcal V\in \cov(K)$,
\begin{equation}\label{succ}
\text{if}\quad \mathcal V\succ\mathcal U\quad \text{then}\quad N(\mathcal V)\geq N(\mathcal U).
\end{equation}
If $K'$ is a compact space, $f\colon K'\to K$ a continuous map, and $\mathcal U=\{U_j\}_{j\in J}\in\cov(K)$, let 
\[f^{-1}(\mathcal U)=\{f^{-1}(U_j)\}_{j\in J}\in\cov(K').\]
Moreover, for $\mathcal U\in\cov(K)$ and $F\in\Pf(S)$, let 
\[\mathcal U_{\rho,F}=\bigvee_{s\in F}\rho_s^{-1}(\mathcal U).\]
\begin{definition}\label{def:top_entropy}
For $\mathcal U\in\cov(K)$, the \emph{topological entropy of $\rho$ with respect to $\mathcal U$} is 
\begin{equation}\label{def_partial_h_top}
H_{\top}(\rho,\mathcal U)=\lim_{i\in I}\frac{\log N(\mathcal U_{\rho,F_i})}{|F_i|}.
\end{equation}
The \emph{topological entropy of $\rho$} is $h_{\top}(\rho)=\sup\{H_{\top}(\rho,\mathcal U): \mathcal U\in \cov(K)\}.$
\end{definition}

For $h_\top$ we present the following two alternative descriptions.

\begin{example}\label{Gactionsextop} 
Let $K$ be a compact space.
\begin{enumerate}[(1)]
\item Analogously to Example~\ref{Gactionsexalg}, the assignment $K \mapsto {\cov}(K)$ from Example~\ref{nmextop} can be completed to a contravariant functor 
${\cov}\colon \CompS \to \mathfrak M$ 
such that ${\cov}(f)(\U) = f^{-1}(\U)$ for any morphism $f\colon K \to L$ in $\CompS$ and open cover $\U \in {\cov}(L)$. 
If $j\colon Y\to X$ is a subspace embedding in $\CompS$,  the map ${\cov}(j) \colon \cov (X) \to \cov (Y)$ is a surjective morphism in $\mathfrak M$, while a quotient map $q\colon X \to Y$ gives rise to an embedding  $\cov (q) \colon \cov( Y) \to \cov (X)$  in $\mathfrak M$. 

\item For a right $S$-action $K\ra{\rho}S$, the functor $\cov$ induces a left $S$-action 
$$S\la{\rho_\cov} \cov(K),\quad \text{with}\ (\rho_\cov)_s(\U)=\rho_s^{-1}(\U)\ \text{for every $s\in S$ and $\U\in\cov(K)$}.
$$
For any $F\in\Pf(S)$ and $\U\in\cov(K)$, $$T_F(\rho_\cov,\U)=\bigvee_{s\in F}\rho_s^{-1}(\U)=\U_{\rho,F}.$$ So, for any right F\o lner net $\s$ for $S$, 
$H_{\top}(\rho,\U) = H(\rho_\cov ,\s ,\U)$ and $h_\top(\rho)=h(\rho_\cov,\s)$.

Hence, the existence of the limit in \eqref{def_partial_h_top} and its independence on the choice of $\s$ follow from Lemma~\ref{reallim}. 
\end{enumerate}
\end{example}

 The following example in used in \S\ref{AppC}.

\begin{example}\label{Gactionsextop2} 
Let $K$ be a compact group.
\begin{enumerate}[(1)]
\item Consider the contravariant functor 
$\mathfrak U\colon \CompG\to \mathfrak M$ 
defined by $K\mapsto \mathfrak U(K)$ (see Example~\ref{nmextop2}) and by letting $\mathfrak U(f)(V)=f^{-1}(V)$ for every morphism $f\colon K\to L$ in $\CompG$ and $V\in\mathfrak U(L)$.

\item For a right $S$-action $K\ra{\rho}S$, the functor $\mathfrak U$ induces a left $S$-action 
$$S\la{\rho_{\frak U}} {\frak U}(K),\quad\text{with}\  (\rho_{\frak U})_s(U)=\rho_s^{-1}(U)\ \text{for every $s\in S$ and $U\in\mathfrak U(K)$}.$$
For any $F\in\Pf(S)$ and $U\in\mathfrak U(K)$, $$T_F(\rho_\mathfrak U,U)=\bigcap_{s\in F}\rho_s^{-1}(U)=C_F(\rho,U).$$
 In these terms, for a right F\o lner net $\s=\{F_i\}_{i\in I}$ for $S$,
$$h(\rho_\mathfrak U,\s)=\sup\{H(\rho_\mathfrak U,\s,U):U\in\mathfrak U(K)\}, \quad\text{with}\ H(\rho_\mathfrak U,\s,U)=\limsup_{i\in I}\frac{-\log \mu(C_F(\rho,U) )}{|F_i|}.$$

\item  The submonoid $\L^{\op}(K)$ {of $\mathfrak U(K)$} is $S$-invariant, so $\rho_\mathfrak U$ restricts to it. When $K$ is profinite, that is, $\L^{\op}(K)$ is cofinal in $\mathfrak U(K)$, by Example~\ref{cofinal} and Proposition~\ref{wad}, recalling that the norm $v_\mathfrak U$ restricted to $\L^{\op}(K)$ is subadditive and so the above limit superior becomes a limit by Lemma~\ref{reallim}, we get
$$h(\rho_\mathfrak U,\s)=h((\rho_\mathfrak U)_{\restriction \L^{\op}(K)},\s)=\sup\{H(\rho_\mathfrak U,\s,U):U\in\mathfrak \L^{\op}(K)\}, \quad\text{with}\ H(\rho_\mathfrak U,\s,U)=\lim_{i\in I}\frac{\log [K:C_F(\rho,U)]}{|F_i|},$$
and the quantities do not depend on the choice of the right F\o lner net $\s$ for $S$.

\item {Considering $\L^{\op}(K)$ as a submonoid of $\cov(K)$ (as in Example \ref{nmextop2}(5)), $\L^{\op}(K)$ is $S$-invariant also in $\cov(K)$}, so $\rho_\cov$ restricts to it and $h((\rho_\cov)_{\restriction \L^{\op}(K)},\s)=h((\rho_\mathfrak U)_{\restriction \L^{\op}(K)},\s)$ for a right F\o lner net $\s$ for $S$.
When $K$ is profinite, that is, $\L^{\op}(K)$ is cofinal both in $\mathfrak U(K)$ and in $\cov(K)$, we then deduce that
$h_\top(\rho)=h((\rho_\cov)_{\restriction \L^{\op}(K)},\s)=h((\rho_\mathfrak U)_{\restriction \L^{\op}(K)},\s)$ for a right F\o lner net $\s$ for $S$.

\item When $S=G$ is a group, $G\la{\rho_{\frak U}}{\frak U}(K)$ and $G\la{\rho_{\cov}}{\cov}(K)$ are equivalent (see Proposition~\ref{cor3}); hence, $h_{\top}(\rho) = h(\rho_{\frak U},\s)$ for any right F\o lner net $\s=\{F_i\}_{i\in I}$ for $G$, in view of Example~\ref{Gactionsextop} and Proposition \ref{conjumon}. In detail, 
$$
h_\top(\rho)=\sup\left\{\lim_{i\in I}\frac{-\log \mu(C_F(\rho,U))}{|F_i|} : U\in\frak U(K)\right\}.
$$
As a by-product we obtain a new proof of the well-known fact that the topological entropy defined by Adler, Konheim and McAndrew
and that defined by Bowen (namely, Hood's extension) coincide on $\CompG$, as mentioned in the very beginning of the introduction.
\end{enumerate}
\end{example}

\begin{remark}\label{RemarkMay16*}
Let $G$ be the group of left fractions of $S$ and  $G\ra{\rho} K$ a right $G$-action. Then $G$ is amenable and every right F\o lner net for $S$ is also a right F\o lner net for $G$, by Lemma \ref{Lemma:ex-2.2}. Therefore, $h_{\top}(\rho) = h_{\top}(\rho_{\restriction S})$. 
\end{remark}

\begin{proposition}\label{mono-top}
Let $K'$ be another compact space and $f\colon K\to K'$ an $S$-equivariant continuous map with respect to the right $S$-actions $K\ra{\rho}S$ and $K'\ra{\rho'}S$.
Then, the following statements hold true:
\begin{enumerate}[(1)]
   \item if $f$ is surjective, then $h_\top(\rho)\geq h_\top(\rho')$;
   \item if $f$ is injective, then $h_\top(\rho)\leq h_\top(\rho')$.
\end{enumerate}
\end{proposition}
\begin{proof}
(1) By Example~\ref{Gactionsextop}, $\cov(f) \colon \cov (K') \to \cov (K)$ is an embedding in $\mathfrak M$ so $h_\top(\rho) = h(\rho_\cov) \geq h(\rho'_\cov) = h_\top(\rho')$, in view of Corollary~\ref{Mono:mono}(1). 

\smallskip
(2) By Example~\ref{Gactionsextop}, $\cov(f)$ is a surjective morphism in $\mathfrak M$, so 
$h_\top(\rho) \! = \! h(\rho_\cov) \!\leq \! h(\rho'_\cov) \! = \! h_\top(\rho')$ by Corollary~\ref{Mono:mono}(2).
\end{proof}

As a consequence of the above proposition,  $h_\top(\rho)= h_\top(\rho')$, whenever the actions $\rho$ and $\rho'$ are conjugated.

\subsubsection{Topological entropy for linear actions on groups}

Suppose now that $K$ is a compact group and that the right $S$-action $K\ra{\rho}S$ is by continuous endomorphisms. In case $H$ is a closed (but not necessarily normal) subgroup of $K$ which is $S$-invariant, that is, $\rho_s(H)\leq H$ for each $s\in S$, there is an obvious right $S$-action $H\ra{\rho_{ H}}S$ that makes the inclusion $H\to K$ an $S$-equivariant injective continuous homomorphism. 

Denote by $\pi\colon K\to K/H=\{kH:k\in K\}$ the canonical projection to the space of left $H$-cosets. These objects always satisfy the following important properties:
\begin{itemize}
   \item when endowed with the quotient topology, $K/H$ is a compact space \cite[Theorems 5.21 and 5.22]{HR}, and $\pi$ is continuous, open and surjective;
   \item $K$ acts on the left on $ K/H$ and $\pi$ commutes with this action, i.e., $\pi$ is $K$-equivariant.
\end{itemize}
 Moreover, let $K/H\ra{\rho_{K/H}}\ \ S$ be the right action induced by $\rho$, with respect to which $\pi$ is $S$-equivariant.

\smallskip
As a consequence of Proposition~\ref{mono-top}, we obtain the following

\begin{proposition}\label{properties_h_top} 
In the above notation, the following statements hold true:
\begin{enumerate}[(1)]
  \item if $K'$ is another compact group and $K'\ra{\rho'}S$ is a right $S$-action conjugated to $K\ra{\rho}S$, then $h_{\top}(\rho')=h_{\top}(\rho)$;
  \item if $H$ is a closed $S$-invariant subgroup of $K$, then $h_{\top}(\rho)\geq \max\{h_{\top}(\rho_{ H}),\,h_{\top}( \rho_{K/H})\}$. 
\end{enumerate}
\end{proposition}

\section{Ore localization of actions on discrete Abelian groups}\label{sec_discrete}

Throughout this section we fix a cancellative and right amenable monoid $S$, its group of left fractions $G=S^{-1}S$, a discrete Abelian group $X$ and a left $S$-action $S\la{\lambda} X$. 
%
The aim of this section is to reduce the computation of $h_\alg(\lambda)$ to the computation of the algebraic entropy of a suitable left $G$-action $G\la{\lambda^*} X^*$, called the (left) Ore localization of $S\la{\lambda}X$. 

First, in \S\ref{reduction_to_injective}, we construct a left $S$-action by injective endomorphisms $S\la{\bar\lambda} \bar X$ and we prove that $h_\alg(\bar\lambda)=h_\alg(\lambda)$. Then, in \S\ref{Ore_invariance_subs}, we use $S\la{\bar\lambda} \bar X$ to construct the Ore localization $G\la{\lambda^*} X^*$ of $S\la{\lambda}X$, and we prove that $h_\alg(\lambda)=h_\alg(\lambda^*)$. Finally, in \S\ref{categorial_ore_loc_subs}, we sketch some categorical interpretations of the construction of the Ore localization.

\subsection{Reduction to actions by injective endomorphisms}\label{reduction_to_injective}

The goal of this subsection is to give an explicit construction that, starting with a left $S$-action $S\la{\lambda} X$ on an Abelian group $X$, produces a new left $S$-action $S\la{\bar\lambda} \bar X$ such that $\bar \lambda_s$ is injective for each $s\in S$. Furthermore, we see that the algebraic entropy does not distinguish between $\lambda$ and $\bar \lambda$, that is, $h_\alg(\lambda)=h_\alg(\bar \lambda)$. This fact allows us to restrict our attention to the algebraic entropy of those left $S$-actions that act by injective endomorphisms, knowing that such a technical simplification causes no loss in generality. 

The main ingredient in our construction is the kernel of $S\la{\lambda} X$, which is defined as follows:
\[\Ker(\lambda)= \{x\in X: \exists s\in S,\ \lambda_s(x) = 0\}=\bigcup_{s\in S}\Ker(\lambda_s).\]

\begin{lemma}\label{props_ker_lemma}
 The following statements hold true: 
\begin{enumerate}[(1)]
    \item $\Ker(\lambda)$ is a subgroup of $X$;
    \item $\lambda_s^{-1}(\Ker(\lambda))=\Ker(\lambda)$, for all $s\in S$.
\end{enumerate}
\end{lemma}
\begin{proof}
(1) Consider the direct system $\{\Ker(\lambda_s):s\in S\}$ of subgroups of $X$, indexed by $(S,\leq)$ with the right divisibility relation. Then, $\Ker(\lambda)=\bigcup_{s\in S}\Ker(\lambda_s)=\sum_{s\in S}\Ker(\lambda_s)$, which is therefore a subgroup of $X$.

\smallskip
(2) Given $s,t\in S$, we have that  $\lambda_s^{-1}(\Ker(\lambda_t))=\Ker(\lambda_{ts})\leq \Ker(\lambda)$. Furthermore, $Ss$ is cofinal in $(S,\leq)$ by (LO$''$), and so $\bigcup_{t\in S}\Ker(\lambda_{ts})=\Ker(\lambda)$. Therefore, $\lambda_s^{-1}(\Ker(\lambda))=\lambda_s^{-1}(\bigcup_{t\in S}\Ker(\lambda_t))=\bigcup_{t\in S}\lambda_s^{-1}(\Ker(\lambda_t))=\bigcup_{t\in S}\Ker(\lambda_{ts})= \Ker(\lambda).$
\end{proof}

By the above lemma, $\Ker(\lambda)$ is an $S$-invariant subgroup of $X$. Let $\bar X=X/\Ker(\lambda)$ and denote by $\pi_X\colon X\to \bar X$ the canonical projection. Then, $\lambda$ induces a left $S$-action on $\bar X$
\[ S\overset{\bar\lambda}{\curvearrowright}\bar X, \quad\text{such that}\quad \bar\lambda_s(\pi_X(x))=\bar\lambda_s(x+\Ker(\lambda))=\lambda_s(x)+\Ker(\lambda)=\pi_X(\lambda_s(x)),\quad \text{for all}\ s\in S,\  x\in X.\] 
Therefore, $\pi_X$ is $S$-equivariant.

\begin{corollary}\label{coro_act_by_inj}
In the above setting, the following statements hold true:
\begin{enumerate}[(1)] 
    \item $\bar \lambda$ acts on $\bar X$ by injective endomorphisms, that is, $\bar\lambda_s$ is injective for all $s\in S$;
    \item given a second left $S$-action $S\la{\lambda'}X'$ on an Abelian group $X'$, and an $S$-equivariant homomorphism $\phi\colon X\to X'$, there is a unique homomorphism $\bar\phi\colon \bar X\to \bar X'$ such that $\pi_{X'}\circ\phi=\bar \phi\circ\pi_{X}$. Furthermore, $\bar\phi$ is $S$-equivariant and it is injective (resp., surjective), whenever $\phi$ is is injective (resp., surjective);
    \item $h_\alg(\nl)\leq h_\alg(\lambda)$.
\end{enumerate}
\end{corollary}
\begin{proof}
(1) Given $s\in S$, we have that $\Ker(\bar\lambda_s)={\pi_X}(\lambda_s^{-1}(\Ker({\lambda})))={\pi_X}(\Ker(\lambda))=0$, by Lemma~\ref{props_ker_lemma}(2). Thus, the action of $S$ on $\bar X$ is by injective endomorphisms. 

\smallskip
(2) The uniqueness and existence of $\bar\phi$ are clear, it is well-defined since $\phi(\Ker(\lambda))\subseteq \Ker(\lambda')$, and by construction it is $S$-equivariant. Suppose now that $\phi$ is surjective. Then, $\bar \phi(\bar X)=\bar\phi(\pi_X(X))=\pi_{X'}(\phi(X))=\pi_{X'}(X')=\bar X'$, showing that $\bar \phi$ is surjective. On the other hand, if $\phi$ is injective, then $\Ker(\bar\phi)=\bar\phi^{-1}(\pi_{X'}(\Ker(\lambda')))=\pi_X(\phi^{-1}(\Ker(\lambda')))$, and so we have just to show that $\phi^{-1}(\Ker(\lambda'))\subseteq \Ker(\lambda)$. Indeed, assume that $x\in \phi^{-1}(\Ker(\lambda'))$. Then  $\phi(x) \in \Ker(\lambda')$, so  there is $s\in S$ such that $\lambda'_s(\phi(x))=0$. Then  $\phi(\lambda_s(x))=\lambda'_s(\phi(x))=0$, which means that $\lambda_s(x)=0$ by the injectivity of $\phi$. Therefore, $x\in \Ker(\lambda)$. 

\smallskip
(3) follows by Proposition~\ref{mon}(2), since the projection ${\pi_X}\colon X\to \bar X$ is surjective and $S$-equivariant.
\end{proof}

As a consequence of part (2) of the above corollary, the assignment  $(X,\lambda)\mapsto (\bar X,\bar \lambda)$ is part of an exact functor 
\[\overline{(-)}\colon\lre(S,\Ab)\to \lre(S,\Ab).\]
  
Now our aim is to sharpen the inequality in part (3) of the above corollary to the  equality 
$h_{\alg}(\lambda)=h_{\alg}(\nl)$. 
We need first the following technical lemma:

\begin{lemma}\label{lemma2} 
In the above setting, for each $F\in\Pf(X)$ there is an $s\in S$ such that, for any choice of $x,y\in F$,
\[(*)\qquad\qquad\qquad {\pi_X}(x)={\pi_X}(y) \qquad \Longrightarrow \qquad \lambda_s(x)=\lambda_s(y).\qquad\qquad\qquad\]
In particular, $|\lambda_s(F)| \leq |{\pi_X}(F)|$. 
\end{lemma}
\begin{proof}
Define a subset $A\subseteq F\times F$ as follows (in fact, $A$ is an equivalence relation on $F$):
\[A=\{(x,y)\in F\times F\colon {\pi_X}(x)={\pi_X}(y)\}.\]
Since $F$ is finite, $A$ is finite as well, so one can enumerate the elements of $A$ as $A=\{a_1,\ldots,a_n\}$, with $a_j=(x_j,y_j)$ for all $j=1,\ldots,n$. Now, for each {$j=1,\ldots,n$}, the equality ${\pi_X}(x_j)={\pi_X}(y_j)$ holds if and only if $x_j-y_j\in \Ker(\lambda)$ and, since $\Ker(\lambda)=\bigcup_{s\in S}\Ker(\lambda_s)$ by Lemma~\ref{props_ker_lemma}(1), there is an $s_j\in S$ with $x_j-y_j\in \Ker(\lambda_{s_i})$, that is, $\lambda_{s_j}(x_j) = \lambda_{s_j}(y_j)$. By Lemma~\ref{GO}(1), one can choose $s, t_1, \ldots, t_n\in S$ such that $s=t_1s_1=\ldots=t_ns_n$.  For each {$j=1,\dots, n$}, the choice of $s_j$, $t_j$ and $s$ allows us to show that
\[\lambda_s(x_j)=\lambda_{t_j}(\lambda_{s_j}(x_j))=\lambda_{t_j}(\lambda_{s_j}(y_j))=\lambda_s(y_j).\] 
Hence, we have found an $s\in S$ such that $\lambda_s(x)=\lambda_s(y)$, for all $(x,y)\in A$, as desired.
\end{proof}


\begin{proposition}\label{prop2} 
In the above setting, $h_{\alg}(\lambda) = h_{\alg}(\nl)$.
\end{proposition}
\begin{proof} 
The inequality $h_{\alg}(\bar\lambda)\leq h_{\alg}(\lambda)$ is proved  in Corollary~\ref{coro_act_by_inj}. 
Let $\{F_i\}_{i\in I}$ be a right F\o lner net for $S$ and let $E\in\Pf(X)$. 
By Lemma~\ref{lemma2} applied to $T_{F_i}(\lambda,E)$, for each $i\in I$, there exists $s_i\in S$ such that 
\[
|T_{s_iF_i}(\lambda,E)|= |\lambda_{s_i}(T_{F_i}(\lambda,E))|\leq |{\pi_X}(T_{F_i}(\lambda,E))|=|T_{F_i}(\bar \lambda,{\pi_X}(E))|.
\]
By Lemma~\ref{fol}, the net $\{s_iF_i\}_{i\in I}$ is right F\o lner for $S$ and so:
\begin{align*}
H_{\alg}(\lambda,E)&=\lim_{i\in I}\frac{\log|T_{s_iF_i}(\lambda,E)|}{|s_i F_i|}\leq \lim_{i\in I}\frac{\log|T_{F_i}(\bar\lambda,{\pi_X}(E))|}{|F_i|}=H_{\alg}(\bar\lambda,{\pi_X}(E)).
\end{align*}
Since $E$ was chosen arbitrarily, we conclude that $h_{\alg}(\lambda)\leq h_{\alg}(\bar\lambda)$.
\end{proof}

\subsection{Invariance under Ore localization for $h_\alg$}\label{Ore_invariance_subs}\label{Ore_invariance}

Given a left $S$-action $S\la\lambda X$ on the Abelian group $X$, our first goal is to construct a canonical left $G$-action $G\la{\lambda^*}X^*$ associated with $\lambda$ (see Definition~\ref{def_left_Ore_loc}). We do this in two steps: the group $X^*$ is introduced in Definition~\ref{def_x_star}, while the action $\lambda^*$ is described by its universal property in Lemma~\ref{def_lambda_star_lemma}.

\begin{definition}\label{def_x_star}
Let $S\la{\bar \lambda}\bar X=X/\Ker(\lambda)$ be the action induced by $\lambda$ on the quotient. Consider the following direct system of Abelian groups with index set $(G, \leq_S)$:
\begin{itemize}
  \item ${\frak X}=\{(X_{g}, \varepsilon_{gs,g}\colon X_{gs}\to X_g): g\in G,\ s\in S\}$, where $X_g = \bar X$ and $\varepsilon_{gs,g}=\bar\lambda_s\colon \bar X\to \bar X$, for all $s\in S$ and $g\in G$;
  \item denote the direct limit by $X^*=\varinjlim_{(G,\leq_S)}{\frak X}$;
  \item and let $\varepsilon_g=\varepsilon_g^X\colon \bar X=X_g\to X^*$ be the canonical morphism to the direct limit, for all $g\in G$.
\end{itemize}
\end{definition}
In particular, by definition of direct limit, the following relations hold, for all $s\in S$ and $g\in G$:
\begin{equation}\label{eq_bar_to_star_equation}
\varepsilon_g\circ \bar\lambda_s=\varepsilon_g\circ \varepsilon_{gs,g}=\varepsilon_{gs}.
\end{equation}

\begin{lemma}\label{def_lambda_star_lemma} Let $X^*$ and $\varepsilon_g\colon \bar X=X_g\to X^*$, for all $g\in G$, be as in Definition~\ref{def_x_star}. Then the following assertions hold true:
\begin{enumerate}[\rm (1)]
   \item $\varepsilon_g$ is injective,  for all $g\in G$. In particular, identifying $X_g$ with the image of $\varepsilon_g$, we have that $X^*=\bigcup_{g\in G}X_g$;
   \item there is a unique left $G$-action $G\la{\lambda^*}X^*$ such that the following diagram commutes, for all $g,h\in G$:
\begin{equation}\label{univ_prop_act_alg_comp_eq}
\xymatrix@C=50pt{
X_{g}\ar[r]^{\id_{\bar X}}\ar[d]_{\varepsilon_{g}}&X_{gh}\ar[d]^{\varepsilon_{gh}}\\
X^*\ar[r]_{\lambda^*_g}&X^*;}
\end{equation}
   \item $\varepsilon_1\colon \bar X\to X^*$ is $S$-equivariant when we let $S$ act on $\bar X$ via $\bar\lambda$ and on $X^*$ via $(\lambda^*)_{\restriction S}$;
   \item given a second left $S$-action $S\la{\lambda'}X'$ on an Abelian group $X'$ and an $S$-equivariant homomorphism $\phi\colon X\to X'$, there is a unique homomorphism $\phi^*\colon X^*\to (X')^*$ such that, for every $g\in G$, the following diagram commutes
\[\xymatrix@C=50pt{X^*\ar[r]^{\phi^*}\ar@{<-}[d]_{\varepsilon^X_g}&(X')^*\ar@{<-}[d]^{\varepsilon^{X'}_g}\\
\bar X\ar[r]_{\bar\phi}&\bar X'.}\]
Furthermore, $\phi^*$ is $G$-equivariant and if $\phi$ is injective (resp., surjective) then so is $\phi^*$.
\end{enumerate}
\end{lemma}
\begin{proof}
Part (1) follows from the fact that the transition maps in the direct system $\frak X$ of Definition~\ref{def_x_star} are all injective. 

\smallskip
(2) Let $g\in G$ and consider the following family of group homomorphisms $\{\lambda^*_{g,h}\colon X_h\to X^*:h\in G\}$ where $X_h=\bar X$ and $\lambda_{g,h}^*=\varepsilon_{gh}\colon \bar X\to X^*$. This family is compatible with the transition maps in $\frak X$, in fact, for each $h\in G$ and $s\in S$, we have that $\varepsilon_{ghs}=\varepsilon_{gh}\circ \bar\lambda_s$, by \eqref{eq_bar_to_star_equation}. Then, by the universal property of the direct limit, there exists a unique group homomorphism $\lambda_g^*\colon X^*\to X^*$, such that 
\begin{equation}\label{univ_prop_act_alg_comp_equation}
\lambda_g^*\circ \varepsilon_h=\lambda^*_{g,h}=\varepsilon_{gh}.
\end{equation}
To conclude, it is enough to verify that the family $\{\lambda_g^*:g\in G\}$ is actually a left $G$-action on $X^*$. Indeed, $\lambda_1^*=\id_{X^*}$ since $\id_{X^*}\circ \varepsilon_h=\varepsilon_{h}=\lambda^*_{1,h}$ for all $h\in G$, so $\id_{X^*}$ satisfies the universal property that defines $\lambda_1^* $ (i.e., \eqref{univ_prop_act_alg_comp_equation} with $g=1$). Similarly, given $g_1,g_2\in G$, we have that $\lambda_{g_1g_2}^*=\lambda_{g_1}^*\circ\lambda_{g_2}^*$ because 
\[\lambda_{g_1}^*\circ\lambda_{g_2}^*\circ \varepsilon_h=\lambda_{g_1}^*\varepsilon_{g_2h}=\varepsilon_{g_1g_2h}=\lambda^*_{g_1g_2,h},\]
for all $h\in G$. In particular, $\lambda_{g_1}^*\circ\lambda_{g_2}^*$ satisfies the universal property that defines $\lambda_{g_1g_2}^* $ (i.e., \eqref{univ_prop_act_alg_comp_equation} with $g=g_1g_2$).

\smallskip
(3) Given $s\in S$, we deduce by \eqref{univ_prop_act_alg_comp_equation} with $h=1$ and $g=s$, that $\lambda_s^*\circ \varepsilon_1=\varepsilon_{s}$. Furthermore, by \eqref{eq_bar_to_star_equation} with $g=1$, we deduce that $\varepsilon_1\circ \bar\lambda_s=\varepsilon_{s}$. Hence, $\lambda_s^*\circ \varepsilon_1=\varepsilon_1\circ \bar\lambda_s$, showing that $\varepsilon_1$ is $S$-equivariant.

\smallskip
(4) Existence and uniqueness of  $\phi^*$ are clear by the universal property of the direct limit, while $G$-equivariance follows by construction. By  Corollary~\ref{coro_act_by_inj}(2), $\bar\phi$ is injective (resp., surjective) whenever $\phi$ has the same property and, therefore, one can conclude by the exactness of direct limits in $\Ab$.
\end{proof}

\begin{definition}\label{def_left_Ore_loc}
Let $S\la{\lambda}X$ be a left $S$-action on the Abelian group $X$. The {\em (left) Ore localization} of $\lambda$ is the left \mbox{$G$-action}  $G\la{\lambda^*}X^*$, where $X^*$ is the group introduced in Definition~\ref{def_x_star} and $\lambda^*$ is the left $G$-action uniquely characterized in Lemma~\ref{def_lambda_star_lemma}(2).
\end{definition}

By Lemma~\ref{def_lambda_star_lemma}(4), the assignment $(X,\lambda)\mapsto (X^*,\lambda^*)$ is part of an exact functor
\[(-)^*\colon \lre(S,\Ab)\to \lre(G,\Ab).\]

The rest of this subsection is devoted to the proof of the equality $h_{\alg}({\lambda})=h_{\alg}({\lambda^*})$, verifying the paradigm that ``the algebraic entropy is invariant under Ore localization". 
Before that,  we need to establish the following easy consequence of Lemma~\ref{def_lambda_star_lemma}:

\begin{corollary}\label{coro_finite_shift}
In the above setting, given $E\in\Pf(X^*)$, there is $s\in S$ such that $\lambda^{*}_s(E)\subseteq \varepsilon_1(\bar X) \leq X^*$.
\end{corollary}
\begin{proof} By Lemma~\ref{def_lambda_star_lemma}(1) we know that $X^*$ can be written as a direct union of subgroups $X^*=\bigcup_{g\in G}\varepsilon_g(\bar X)$ and, in particular, there is some $g\in G$ such that $E\subseteq \varepsilon_g(\bar X)$. Write $g=s^{-1}t$ with $s,t\in S$, and note that $\lambda_s^*(E)\subseteq \lambda_s^*\varepsilon_g(\bar X)=\varepsilon_t(\bar X)$, where the last equality follows by Lemma~\ref{def_lambda_star_lemma}(2). To conclude, take into account that $t\leq_S1$ in $(G,\leq_S)$, and so $\varepsilon_t(\bar X)\subseteq \varepsilon_1(\bar X)$.
\end{proof}

We are finally ready for the proof of the invariance of the algebraic entropy under Ore localization:

\begin{theorem}[Invariance under Ore localization]\label{Th1}
In the above setting,  
$h_{\alg}(\lambda) = h_{\alg}(\lambda^*)$.
\end{theorem}
\begin{proof} 
By Proposition~\ref{prop2}, it is enough to verify that $h_{\alg}(\bar\lambda) = h_{\alg}(\lambda^*)$. 

Given $E\in \Pf(X^*)$, we know by Corollary~\ref{coro_finite_shift} that there exists $s\in S$ such that $\lambda^*_s(E)\subseteq  \varepsilon_1(\bar X)\leq X^*$.  Furthermore, by \cite[Lemma~2.7(b)]{DFGB}, we know that if $\{F_i\}_{i\in I}$ is a right F\o lner net for $S$, then so is $\{F_is\}_{i\in I}$ and they both are right F\o lner nets for $G$, by Lemma~\ref{Lemma:ex-2.2}(2). Therefore:
\begin{align*}
H_{\alg}(\lambda^*,E)=\lim_{i\in I}\frac{\log|T_{F_is}(\lambda^*,E)|}{|F_is|}=\lim_{i\in I}\frac{\log|T_{F_i}(\lambda^*,\lambda^*_s(E))|}{|F_is|}\overset{(*)}{=}\lim_{i\in I}\frac{\log|T_{F_i}( \bar\lambda,\varepsilon_1^{-1}(\lambda_s^{*}(E)))|}{|F_i|}=H_{\alg}(\bar\lambda,\varepsilon_1^{-1}(\lambda_s^{*}(E)))\leq h_{\alg}(\bar \lambda),
\end{align*}
where the equality $(*)$ is true since $|F_i|=|F_is|$ and since $\varepsilon_1$ induces a bijection between $T_{F_i}(\bar\lambda,\varepsilon_1^{-1}(\lambda_s^{*}(E)))$ and  $T_{F_i}(\lambda^*,\lambda^*_s(E))$ (use that $\varepsilon_1$ is injective and $S$-equivariant by Lemma~\ref{def_lambda_star_lemma}(1) and (3)). This proves that  $h_{\alg}(\lambda^*) \leq h_{\alg}(\bar\lambda)$. 

For the proof of the converse inequality, note that the injective homomorphism $\varepsilon_1 \colon \bar X \to X^*$ is $S$-equivariant, by Lemma \ref{def_lambda_star_lemma}(3).  Hence, $h_{\alg}(\bar\lambda) \leq h_{\alg}(\lambda^*_{\restriction S})$.  By Remark \ref{RemarkMay16}, $h_{\alg}(\lambda^*_{\restriction S})= h_{\alg}(\lambda^*)$, therefore, $h_{\alg}(\bar\lambda)\leq h_{\alg}(\lambda^*)$.
\end{proof}

\subsection{Categorical interpretations of the Ore localization}\label{categorial_ore_loc_subs}

{Given a left action $S\la{\lambda}X$ on a discrete Abelian group $X$, in the previous subsections we have made the deliberate choice of giving a direct (and very explicit) construction of $S\la{\bar\lambda}\bar X$ and $S\la{\lambda^*}X^*$. This approach has the advantage of making the text much more accessible and, furthermore, some degree of concreteness was needed for the proof of the equalities $h_\alg(\lambda)=h_\alg(\bar\lambda)=h_\alg(\lambda^*)$. On the other hand, it is also possible to express these constructions in a more categorical language and, in this way, obtain a better understanding of  the whole process that, starting from $S\la{\lambda}X$, produces its Ore localization $S\la{\lambda^*}X^*$.

\medskip
First of all, we observe that the category $\lre(S,\Ab)$ is  equivalent to 
$\lmod{\Z[S]}$, the category of left modules over the monoid ring $\Z[S]$. In fact, a left $\Z[S]$-module ${}_{\Z[S]}X$ is nothing else but an Abelian group ${}_\Z X$ with a specified left $\Z[S]$-action, which is a ring homomorphism $\lambda\colon\Z[S]\to \End_{\Ab}(X)$,  uniquely determined by the monoid homomorphism $\lambda_{\restriction S}\colon S\to \End_{\Ab}(X)$ together with $\Z$-linearity.

Now, inside $\lre(S,\Ab)\cong\lmod{\Z[S]}$, one can identify two reflective subcategories, where a full subcategory $\C'$ of a category $\C$ is said to be reflective if the inclusion functor $\C'\to \C$ has a left adjoint functor $\C\to \C'$, called the reflector (see \cite[\S IV.3]{maclane} for more details): 
\[
 \lre_{\rm bij}(S,\Ab)\subseteq\lre_{\rm inj}(S,\Ab)\subseteq \lre(S,\Ab).
\] 
Here $\lre_{\rm bij}(S,\Ab)$ and $\lre_{\rm inj}(S,\Ab)$ are the full subcategories of $\lre(S,\Ab)$ of those $S\la{\lambda}X$ such that each $\lambda_s$  is bijective or injective, for all $s\in S$, respectively. Moreover,  $\lre_{\rm bij}(S,\Ab)\cong \lre(G,\Ab)\cong \lmod{\Z[G]}$, with $G=S^{-1}S$.

To verify the reflectivity of these subcategories we need to build a left adjoint to each of the two inclusion functors $\lre_{\rm inj}(S,\Ab)\to \lre(S,\Ab)$ and $\lre_{\rm bij}(S,\Ab)\to \lre(S,\Ab)$. In \S\ref{reduction_to_injective} and \S\ref{Ore_invariance_subs}, we have built the corresponding reflectors ``manually", verifying that the $S$-equivariant projection $\pi_X\colon X\to \bar X$ is a reflection in $\lre_{\rm inj}(S,\Ab)$ of the left $S$-action $(S\la\lambda X)\in\lre(S,\Ab)$, while the composition $\varepsilon_1\circ\pi_X\colon X\to X^*$ of $\pi_X$, followed by the $S$-equivariant injection $\varepsilon_1\colon \bar X \to X^*$, is a reflection of $S\la\lambda X$ in $\lre_{\rm bij}(S,\Ab)$.

Alternatively, this problem can be  solved via the abstract machinery of hereditary torsion pairs. Recall that, for a ring $R$, a full subcategory $\mathcal T$ of $\lmod R$ is said to be a hereditary torsion class if it is closed under submodules, quotients, extensions and arbitrary direct sums. Given $\mathcal T$, the corresponding torsionfree class is $\F=\{M\in \lmod R:\hom_R(T,M)=0,\ \forall T\in\mathcal T\}$, and $\tau=(\mathcal T,\F)$ is said to be a hereditary torsion pair (see \cite[\S VI.2-3]{Ste}). In this situation, $\F$ is always a reflective subcategory of $\lmod R$. Moreover, $\tau$ induces a so-called Gabriel topology $\frak F_\tau:=\{{}_RI\leq R:R/I\in \mathcal T\}$ on $R$ (see \cite[\S VI.4-5-6]{Ste}), with which we can define the corresponding Gabriel localization $\lmod {(R,\frak F_\tau)}$ of $\lmod R$ (also commonly denoted by $\lmod R/\mathcal T$), as follows. This is the full subcategory of the $\tau_s$-closed (or $\tau_s$-local) modules, that is, those $M\in \lmod R$ such that the canonical map $M\cong\hom_R(R,M)\to \hom_R(I,M)$ is an isomorphism for all $I\in\frak F_\tau$ (see \cite[\S IX.1-2]{Ste}). Then,  $\lmod {(R,\frak F_\tau)}$ is a Giraud subcategory of $\lmod R$ (see \cite[\S X.1-2]{Ste}), 
 that is, a reflective subcategory whose reflector $Q_\tau\colon \lmod R\to \lmod R/\mathcal T$, called the Gabriel $\tau$-localization or $\tau$-quotient functor, is an exact functor. Being a reflection, $Q_\tau$ is ``surjective'' (on isomorphism classes of objects) while, for any left $R$-module $M$, $Q_\tau(M)=0$ if, and only if, $M\in\mathcal T$. Hence, we obtain a sort of ``short exact sequence'' $0\to \mathcal T\to \lmod R\to \lmod R/\mathcal T\to 0$ of Abelian categories, given by the exact embedding $\mathcal T \to \lmod R$, followed by exact $\tau$-quotient functor, where the ``image'' of  the former is precisely the ``kernel'' of the latter; this explains both the notation $\lmod R/\mathcal T$ and also the fact that $Q_\tau$ is called a quotient functor.
  
 Let us briefly indicate how this approach can be applied to our concrete situation:
\begin{itemize}
\item first of all, there is a hereditary torsion pair $\tau_S=(\mathcal T_S,\F_S)$ in $\lre(S,\Ab)$, with $\mathcal T_S$ the class of those actions $S\la{\lambda}X$ such that $\Ker(\lambda)=X$, and $\F_S$ the class of those actions $S\la{\lambda}X$ such that $\Ker(\lambda)=0$ (i.e., $\F_S=\lre_{\rm inj}(S,\Ab)$). 
Then, the inclusions $\mathcal T_S\to  \lre(S,\Ab)$ and $\F_S\to  \lre(S,\Ab)$ have a right and a left adjoint, respectively,
\[
t_S\colon \lre(S,\Ab)\to \mathcal T_S\quad\text{and}\quad (1:t_S)\colon \lre(S,\Ab)\to \F_S,
\]
(see the definitons in \eqref{May13}) which are usually referred to as the $\tau_S$-torsion radical and the $\tau_S$-torsionfree coradical;
\item furthermore, $\lre_{\rm bij}(S,\Ab)$ is a Giraud subcategory of $\lre(S,\Ab)$, which is equivalent to the 
$\tau_S$-localization $\lre(S,\Ab)/\mathcal T_S=\lmod{(\Z[S],\frak F_{\tau_s})}$. In particular, the inclusion $\lre_{\rm bij}(S,\Ab)\to \lre(S,\Ab)$ has the following exact left adjoint functor (that we describe explicitly below):
\[
Q_S\colon \lre(S,\Ab)\to \lre(S,\Ab)/\mathcal T_S\cong\lre_{\rm bij}(S,\Ab).
\]
\end{itemize}
In fact, $\tau_S$ is not just a hereditary torsion pair, but it is a very special and well-behaved one:  the subset $S\subseteq \Z[S]$ is a multiplicative subset that satisfies the left Ore condition for rings. By \cite[Example~2 in \S{VI.4} and Example in \S{VI.9}]{Ste}, given ${}_{\Z[S]}X\in\lmod{\Z[S]}$, 
\begin{equation}\label{May13}
t_S(X)=\{x\in X:sx=0\text{ for some $s\in S$}\}\quad\text{and}\quad (1:t_S)(X)= X/t_S(X).
\end{equation}
This means that, in the language of left $S$-actions, $t_S(X)=\Ker(\lambda)$ and $(1:t_S)(X)= \bar X$. Thus, our ``reduction to actions by injective endomorphisms'' can be viewed as a ``reduction to $\tau_S$-torsionfree $\Z[S]$-modules''. 

Now, to construct $Q_S(X)$, one can first note that $\Z[G]\cong S^{-1}\Z[S]$ is the ring of left $S$-fractions of $\Z[S]$. In particular, each element $r\in \Z[G]$ can be written in the form $g r'$ for some $r'\in \Z[S]$ and $g\in G$. Therefore, considering $g \Z[S]$ as a subgroup of $\Z[G]$ in the obvious way, one can write $\Z[G]$ as a direct union of Abelian groups $\Z[G]=\bigcup_{g\in G}g \Z[S]$, where there is an inclusion $\iota_{g,g'}\colon g \Z[S]\hookrightarrow g' \Z[S]$ if, and only if, $g\leq_Sg'$, for all $g,g'\in G$. Furthermore, by \cite[Lemma~1.6 and Example~2 in {\S}IX.1]{Ste},
\[
Q_S(X)\cong \Z[G]\otimes_{\Z[S]}X\cong \Z[G]\otimes_{\Z[S]}(1:t_S)(X)= \Z[G]\otimes_{\Z[S]}\bar X. 
\]
Hence, as an Abelian group, $Q_S(X)$ can be seen to be isomorphic to $\varinjlim_{(G,\leq_S)}X_g$, where $X_g=g\Z[S]\otimes \bar X=g\otimes \bar X\cong \bar X$ and where the transition maps are given by $\iota_{g,g'}\otimes \id_{\bar X}\colon X_g\to X_{g'}$. Therefore, $Q_S(X)$ is a direct limit of a direct system in $\Ab$ which is isomorphic to the direct system $\frak X$ of Definition~\ref{def_x_star}. In particular, $Q_S(X)\cong X^*$, showing that our ``reduction to actions by bijective endomorphism'' can be viewed as a ``reduction to $\tau_S$-local left $\Z[S]$-modules''.

\medskip
Even if the above torsion theoretic interpretation is very neat, the theory of Gabriel localizations is specific to Abelian categories and, in fact, it works best for Grothendieck categories, which are the cocomplete Abelian categories with a generator and exact directed colimits (in fact, by a famous result of Gabriel and Popescu, an Abelian category is Grothendieck if, and only if, it is of the form $\lmod {(R,\frak F_\tau)}$ for a ring $R$ and a hereditary torsion pair $\tau$ in $\lmod R$). In particular, even if we could try a similar approach for the Abelian category $\rre(S,\Comp)$, one is forced to turn to a different categorical machinery in the case of $\rre(S,\CompS)$. 
So, let us conclude this subsection by showing that the construction of the Ore localization can also be viewed as a suitable left Kan extension (see \cite[\S X.3]{maclane}). As Kan extensions can be defined for general functor categories, this second approach will be easy to adapt to $\rre(S,\CompS)$ (see \S\ref{categorical_coloc_subs} for more details). Indeed, observe that:
\begin{itemize}
     \item $S$ and $G$ can be viewed as categories with one object, say $\Ob(S)=\{\star\}=\Ob(G)$, and 
     \[
     \End_S(\star)=(S,\cdot,1)\subseteq (G,\cdot,1)=\End_G(\star),
     \] where the composition is defined by the multiplication in $S$ and $G$, respectively. Moreover, the inclusion $\iota\colon S\to G$ can be viewed as a (non-full) inclusion of a subcategory;
    \item $\lre(S,\Ab)$ and $\lre(G,\Ab)$ can be seen as categories of covariant functors $S\to \Ab$ and $G\to \Ab$, respectively. In fact, a functor $F\colon S\to \Ab$ is uniquely determined by the Abelian group $X=F(\star)$ and by the left $S$-action $S\la{\lambda}X$ such that $\lambda_s=F(s)$, for all $s\in S$. Similarly, one can view covariant functors $G\to \Ab$ as left $G$-actions. Therefore, the inclusion $\iota\colon S\to G$ induces a forgetful functor
\[
\iota^*\colon \lre(G,\Ab)\to \lre(S,\Ab),\quad\text{such that}\quad F\mapsto F\circ \iota.
\]
In the language of left actions, this means that $\iota^*$ sends a left $G$-action $G\la{\lambda}X$ to the left $S$-action $S\la{\lambda_{\restriction S}}X$ (so just ``forgetting'' part of the action), that is, $\iota^*$ is the inclusion $\lre(G,\Ab)\to \lre(S,\Ab)$  introduced in \S\ref{representations}.
\end{itemize}
When it exists, the left adjoint functor $\iota_!\colon \lre(S,\Ab)\to \lre(G,\Ab)$ is called the left Kan extension along $\iota$. In fact, it is well-known that, whenever the target category (which, in our case, is $\Ab$) is cocomplete, left Kan extensions always exist and they can be computed pointwise using colimits (see \cite[Theorem~3.7.2]{Borceux} or \cite[Theorem 1 in \S X.3]{maclane}). Unraveling all the definitions, one ends up seeing (again) that, given a functor $F\colon S\to \Ab$, which corresponds uniquely to the left $S$-action $S\la{\lambda}X=F(\star)$, and letting $\iota_!(F)\colon G\to \Ab$ be its left Kan extension along $\iota$, there is an isomorphism 
\[
\iota_!(F)(\star)\cong \varinjlim_{(G,\leq_S)}\frak X= X^*,
\] 
where $\frak X$ is the direct system of Definition~\ref{def_x_star}. We omit the details about this isomorphism as we will be much more explicit in \S\ref{categorical_coloc_subs}, when describing the Ore colocalization $\rre(S,\CompS)\to \rre(G,\CompS)$ as a right Kan extension.

\begin{remark}
Even if the categorical treatment in this subsection may seem overly complicated, especially when compared to the  natural constructions of \S\ref{reduction_to_injective} and \S\ref{Ore_invariance_subs}, it is actually an accurate illustration of the steps that we took in developing our research project. In fact, starting from Hanfeng Li's remark that a cancellative right amenable monoid $S$ is always left Ore, it did not take us too long to  conclude that it was possible to use the theory of modules of left fractions (which is, essentially, the above torsion theoretic approach) to associate to any given left $S$-actions on a discrete Abelian group a canonical left $G=S^{-1}S$-action (though it took longer to prove that the process was entropy-preserving). Moreover, just formally following the  dual steps (in the sense of Pontryagin-van Kampen), we readily produced a theory of Ore colocalizations for right actions on compact Abelian groups. Nevertheless, when we described the Ore localizations via suitable left Kan extensions, we obtained a purely categorical construction that applies to functor categories in general and, as usual in category theory, it comes with its formal dual. At that point, the task of defining the Ore colocalizations of actions on compact spaces was reduced to an exercise of formally inverting the direction of arrows. 

The very last step in refining our theory was a bit of a technicality: the first direct system we used to build the Ore localizations in the discrete case (which was the natural choice in the context of modules of left fractions) was indexed by $S$, so it was smaller than (even if cofinal in)  the one used in \S\ref{Ore_invariance_subs}, which is indexed by $G$, and  is the natural choice in the context of Kan extensions. This little change allowed for several simplifications in the arguments (e.g., in Lemma~\ref{def_lambda_star_lemma}), and it is ultimately responsible for the ``natural'' appearance of  \S\ref{Ore_invariance_subs}.
\end{remark}

\section{Ore colocalization of actions on compact spaces}\label{sec_compact}

Throughout this section we fix  a cancellative and right amenable monoid $S$, its group of left fractions $G=S^{-1}S$, a compact space $K$ and a right $S$-action $K\ra{\rho} S$.


Mirroring the results of \S\ref{sec_discrete} for left actions on discrete Abelian groups, in this section we introduce first the right $S$-action $\bar K\ra{\bar\rho} S$ by surjective continuous self-maps (see \S\ref{subs_red_top_to_surj}) and then the Ore colocalization $K^*\ra{\rho^*} G$ of $K\ra{\rho} S$  (see \S\ref{subs_dual_ore}). Moreover, we verify the equalities $h_\top(\rho)=h_\top(\bar\rho)=h_\top(\rho^*)$ (see Theorems~\ref{rhobarrho} and~\ref{Th2}).

\subsection{Inverse limits of compact groups and spaces}\label{topology_of_inverse_limit_subs}

We need to recall some useful facts about inverse limits in the category $\CompS$ of compact spaces.

Consider a directed preordered set $(I,\leq)$ and an inverse system in $\CompS$:
\begin{equation}\label{inv_sys_eq_gen_props}
\frak K=\{(K_i, \pi_{i,j}):i\geq j\ \text{in}\ I\},
\end{equation}
that is, all the maps $\pi_{i,j}\colon K_i\to K_j$, with $i\geq j$ in $I$, are supposed to be continuous and $\pi_{j,k}\circ\pi_{i,j}=\pi_{i,k}$, for all $i\geq j\geq k$ in $I$. 

In the rest of the section we analyze the structure, the topology and the open covers of the limit of ${\frak K}$ in $\CompS$ and prove some exactness-like property for the inverse limit functor in $\CompS$. 

First of all, let  $\prod_{i\in I}K_i$ be the cartesian product and denote by $p_j\colon \prod_{i\in I}K_i\to K_j$ the canonical projection, for each $j\in I$. By the Tychonoff Theorem, $(\prod_{i\in I}K_i,(p_i)_{i\in I})$ is a product in $\CompS$ when equipped with the product topology, generated by the pre-base:
\[\mathcal B=\{p^{-1}_i(U):i\in I,\, U\subseteq K_i\ \text{ open}\}.\]
\begin{lemma}\label{description_B*}
In the above setting, let 
$K^\sharp=\bigcap_{j\leq i\in I}\{x\in  \prod_{k\in I}K_k :p_j(x) = \pi_{i,j}(p_i(x))\} \leq \prod_{i\in I}K_i$ 
and $\pi_j=(p_j)_{\restriction K^\sharp}\colon K^\sharp\to K_j$, for each $j\in I$. The following statements hold true:
\begin{enumerate}[(1)] 
  \item $\pi_{i,j}\pi_i=\pi_j$, for all $j\leq i$ in $I$;
  \item $(K^\sharp,(\pi_i)_{i\in I})$ is an inverse limit for ${\frak K}$ in $\CompS$;
  \item the family $\mathcal B^*=\{\pi_{i}^{-1}(U): i\in I,\, U\subseteq K_i\ \text{ open}\}$ is a base for the topology of $K^\sharp$;
  \item every open cover of $ K^\sharp$ has a (finite) refinement consisting of elements of $\mathcal B^*$.
\end{enumerate}
\end{lemma}
\begin{proof} 
(1) is true by construction, (2) follows by \cite[\S2.5]{Engelking} and (3-4) by \cite[Proposition~2.5.5]{Engelking}.
\end{proof}

As a consequence of the above lemma, any open subset of $K^\sharp$ is a union (possibly infinite) of elements of $\mathcal B^*$.  

\begin{remark}\label{rem_lims_of_groups}
Consider the obvious forgetful functors $\CompG\to \CompS$ and $\Comp\to \CompS$, that forget the group structures and ``only retain the topological information'' about objects and morphisms in $\CompG$ and $\Comp$, respectively. Given an inverse system $\frak K$ like in \eqref{inv_sys_eq_gen_props} in $\CompG$ (resp., $\Comp$), then the product $\prod_{i\in I}K_i$ in $\CompG$ (resp., $\Comp$) is just the cartesian product with the product topology, so the above forgetful functors preserve products. Furthermore, defining $K^\sharp$ as in Lemma~\ref{description_B*}, this is a subgroup of $\prod_{i\in I}K_i$, so it is an object in $\CompG$ (resp., $\Comp$). Lifting the universal property along the forgetful functor, one concludes that $K^\sharp$ is an inverse limit in $\CompG$ (resp., $\Comp$). By this observation, all the results in this subsection apply also to inverse limits in $\CompG$ (resp., $\Comp$).
\end{remark}

By Lemma~\ref{description_B*}, we have an explicit description of the limit $K^\sharp$ of our inverse system $\frak K$ from \eqref{inv_sys_eq_gen_props}. On the other hand, to have an even better control on the structure of $K^\sharp$, it is often useful to slightly modify $\frak K$ in such a way that the connecting maps become all surjective. Indeed, consider the following inverse system:
\[\bar{\frak K}=\{(\bar K_i, \bar \pi_{i,j}):i\geq j\ \text{in}\ I\},\]
where $\bar K_i=\mathrm{Im}(\pi_i)\subseteq K_i$ and $\bar \pi_{i,j}=(\pi_{i,j})_{\restriction \bar K_i}$, for all $i\geq j$ in $I$. Then, all the connecting maps in $\bar{\frak K}$ become surjective and, furthermore, the limit of $\frak K$ and that of $\bar{\frak K}$ are the same:

\begin{lemma}\label{description_B**}
In the above setting, the following statements hold true:
\begin{enumerate}[(1)]
   \item $\bar\pi_{i,j}$ is surjective for all $i\geq j\in I$;
   \item let $\bar\pi_i$ be the corestriction to its image of $\pi_i\colon K^\sharp\to K_i$, for all $i\in I$. Then each $\bar\pi_i$ is surjective;
   \item $(K^\sharp, (\bar\pi_i)_{i\in I})$ is the limit of $\bar{\frak K}$.
\end{enumerate}
In particular, $\varprojlim {\frak K}\cong \varprojlim {\bar {\frak K}}$.\end{lemma}
\begin{proof} 
See \cite[Proposition~2.5.6]{Engelking}.
\end{proof}

The following corollary is essentially covered by~\cite[Proposition 5.9]{DGB4}; we give a direct proof for the reader's convenience.

\begin{corollary}\label{dual_finite_shift}
In the above setting, let $\mathcal U$ be a finite open cover of $K^\sharp$. Then there exist $i\in I$ and a finite open cover $\mathcal U_i$ of $K_i$ such that $\pi_i^{-1}(\mathcal U_i)$ is a refinement of $\mathcal U$. 
\end{corollary}
\begin{proof} 
By Lemma~\ref{description_B*}(4), we can suppose that $\mathcal U=\{U_1,\dots,U_n\}$, where $U_m=\pi_{i_m}^{-1}(U_{m,i_m})$ for some open  $U_{m,i_m}\subseteq K_{i_m}$, with $i_m\in I$, for all $m=1,\dots, n$. Using that $(I,\leq)$ is directed, we can choose $i\in I$ such that $i\geq i_m$ for all $m=1,\dots,n$. Let
\[\U_i=\{\pi_{i,i_m}^{-1}(U_{m,i_m}): m=1,\dots,n\}.\]
By Lemma~\ref{description_B*}(4), we deduce that $\pi_i^{-1}(\pi_{i,i_m}^{-1}(U_{m,i_m}))=\pi_{i_m}^{-1}(U_{m,i_m})=U_m$, for all $m=1,\dots,n$, that is, $\pi_i^{-1}(\U_i)=\U$. We have to verify that $\U_i$ covers $K_i$. Indeed, $\pi_i\colon K^\sharp\to K_i$ is surjective by Lemma~\ref{description_B**}(3), and so 
\[K_i=\pi_i(K^\sharp)=\pi_i(U_1\cup\ldots\cup U_n)=\pi_i(U_1)\cup\ldots\cup \pi_i(U_n)=\pi_{i,i_1}^{-1}(U_{1,i_1})\cup\ldots\cup \pi_{i,i_n}^{-1}(U_{n,i_n})=\bigcup \U_i.\qedhere\]
\end{proof}

We want now to analyze some ``exactness'' properties of inverse limits in $\CompS$. For this consider a directed preordered set $(I,\leq)$, two inverse systems $\frak K_1=\{(K^{(1)}_i, \pi^{(1)}_{i,j}):i\geq j\ \text{in}\ I\}$, $\frak K_2=\{(K^{(2)}_i, \pi^{(2)}_{i,j}):i\geq j\ \text{in}\ I\}$ and a compatible system of continuous maps $\{\phi_i\colon K^{(1)}_i\to K^{(2)}_i:i\in I\}$. Finally, denote by $\phi^\sharp\colon K_1^\sharp\to K_2^\sharp$ the continuous map induced on the inverse limits. 

\begin{lemma}\label{lim_is_exact_for_comp_lemma}
In the above notation, the following statements hold true:
\begin{enumerate}[\rm (1)]
   \item if $\phi_i$ is injective for all $i\in I$, then so is $\phi^\sharp$;
   \item if $\phi_i$ is surjective for all $i\in I$, then so is $\phi^\sharp$.
\end{enumerate}
In particular, inverse limits in $\CompG$ and $\Comp$ are exact (that is, they send inverse systems of short exact sequences to short exact sequences).
\end{lemma}
\begin{proof}
Part (1) follows by  \cite[Lemma~2.5.9]{Engelking} while (2) is \cite[Theorem~3.2.14]{Engelking}.
\end{proof}

Let us conclude this subsection by underlining a useful technical consequence of part (2) of the above lemma, that comes handy for computations:

\begin{corollary}\label{cont_func_on_comp_commutes_w_cap}
Let $K$ and $K'$ be compact Haudorff spaces, $\{H_i:i\in I\}$ a family of closed subspaces of $K$ that is downward directed by inclusion, and $\phi\colon K\to K'$ a continuous map. Then, $\phi(\bigcap_{i\in I}H_i)=\bigcap_{i\in I}\phi(H_i)$. 
\end{corollary}
\begin{proof}
Both $\{H_i:i\in I\}$ and $\{\phi(H_i):i\in I\}$ are inverse systems in $\CompS$, and their inverse limits can be identified with $\bigcap_{i\in I}H_i$ and $\bigcap_{i\in I}\phi(H_i)$, respectively. The family of surjective continuous maps $\{\phi_i=\phi_{\restriction H_i}\colon H_i\to \phi(H_i)\}$ is compatible and so, by Lemma~\ref{lim_is_exact_for_comp_lemma}(2), the induced map $\phi^\sharp=\phi_{\restriction \bigcap_{i\in I}H_i}\colon \bigcap_{i\in I}H_i\to \bigcap_{i\in I}\phi(H_i)$ is surjective.
\end{proof}

\subsection{Reduction to the case of actions by surjective maps}\label{subs_red_top_to_surj}

The {\em surjective core of $K\ra\rho S$} is the following closed {subspace of the compact space} $K$:
\[E(\rho) = \bigcap_{t\in S} \rho_t (K) \subseteq K.\]  
\begin{lemma}\label{bar_is_surj_lemma} In the above notation, $E(\rho)$ is closed in $K$ and $\rho_s(E(\rho))=E(\rho)$, for all $s\in S$.
\end{lemma}
\begin{proof}
Since $K$ is compact, each subset $\rho_s(K)$ is compact, so closed. Therefore, $E(\rho)$ is closed as well. Observe now that $Ss$ is cofinal in $(S,\leq)$ by (LO$''$), so that $E(\rho)=\bigcap_{t\in S}\ \rho_{ts}(K)$. Furthermore, by Corollary~\ref{cont_func_on_comp_commutes_w_cap},
\[\rho_s(E(\rho))=\rho_s\left(\bigcap_{t\in S}\rho_t(K)\right)\overset{}=\bigcap_{t\in S}\rho_s\rho_t(K)=\bigcap_{ts\in Ss}\rho_{ts}(K)=E(\rho).\qedhere\]
\end{proof}

By the above lemma, $E(\rho)$ is  a closed $S$-invariant  subspace of $K$. Let $\bar K=E(\rho)$, denote by $\varepsilon_K\colon E(\rho)\to K$ the inclusion map, and let 
\[ \bar K\overset{\bar\rho}{\curvearrowleft}S,\quad\text{such that}\ \bar\rho_s=(\rho_s)_{\restriction \bar K}\ \text{for all}\ s\in S.\] 

\begin{corollary}\label{coro_top_by_surj}
In the above setting, the following statements hold true:
\begin{enumerate}[(1)]
   \item $\bar \rho$ acts on $\bar K$ by surjective continuous self-maps, that is, $\bar\rho_s$ is surjective for all $s\in S$;
   \item given a second right $S$-action $K'\ra{\rho'}S$ on the compact space $K'$, and an $S$-equivariant continuous map $\phi\colon K\to K'$, there is a unique continuous map $\bar\phi\colon \bar K\to \bar K'$ such that $\phi\circ \varepsilon_K=\varepsilon_{K'}\circ\bar \phi$. Furthermore, $\bar\phi$ is $S$-equivariant and it is injective (resp., surjective), whenever $\phi$ is is injective (resp., surjective);
   \item $h_\top(\bar\rho)\leq h_\top(\rho)$.
\end{enumerate}
\end{corollary}
\begin{proof}
(1) It is clear from Lemma~\ref{bar_is_surj_lemma} that $\bar \rho$ acts by surjective self-maps. 

\smallskip
(2) Using  Corollary~\ref{cont_func_on_comp_commutes_w_cap}, we conclude that 
\begin{equation}\label{Eq:July8}
\phi(\bar K)=\phi\left(\bigcap_{s\in S}\rho_s(K)\right)=\bigcap_{s\in S}\phi(\rho_s(K))=\bigcap_{s\in S}\rho'_s(\phi(K'))\subseteq 
\bigcap_{s\in S}\rho'_s(K') = 
\bar K'.
\end{equation}
Hence, $\phi$ restricts to a unique continuous map $\bar\phi=\phi_{\restriction \bar K}\colon \bar K\to \bar K'$, that is clearly $S$-equivariant. Furthermore, if $\phi$ is injective, its restriction $\bar\phi$ is also injective. On the other hand, if $\phi$ is surjective, then for each $s\in S$ we have that $\phi(\rho_s(K))=\rho_s'(\phi(K))=\rho_s'(K')$. Exploiting this equality we can deduce that the last inclusion in \eqref{Eq:July8} becomes an equality, so $\bar \phi(\bar K)=\phi(\bar K)= \bigcap_{s\in S}\rho'_s(K')=\bar K'$.

\smallskip
(3) This inequality follows by Proposition~\ref{properties_h_top}(2), as the inclusion $\varepsilon_K\colon \bar K\to K$ is injective and $S$-equivariant.
\end{proof}

By part (2) of the above corollary, the assignment $(K,\rho)\mapsto (\bar K,\bar\rho)$ is part of a functor
\begin{equation}\label{EqMay21-}
\overline{(-)}\colon \rre(S,\CompS)\to \rre(S,\CompS),
\end{equation}
that sends continuous injective (resp., surjective)  $S$-equivariant maps to continuous injective  (resp., surjective) $S$-equivariant maps, respectively. Furthermore, this functor restricts to exact functors
\[\overline{(-)}\colon \rre(S,\CompG)\to \rre(S,\CompG)\quad\text{and}\quad \overline{(-)}\colon \rre(S,\Comp)\to \rre(S,\Comp).\]

Now our aim is to upgrade the inequality in Corollary~\ref{coro_top_by_surj}(3) to the equality $h_{\top}(\rho)=h_{\top}(\bar\rho)$. First we need the following technical lemma:

\begin{lemma}\label{claim1} 
In the above notation, given a finite open cover $\U$ of $K$, there exists $s\in S$ such that 
\begin{equation}\label{eq_claim1}
N_{E(\rho)}(\U) = N_{\rho_s(K)}(\U) = N_{K}(\rho_s^{-1}(\U))  .
\end{equation}
\end{lemma}
\begin{proof} 
The  second equality is clear, we prove the first one.
It is clear that $N_{\rho_t(K)}(\U) \leq  N_{K}(\U)$, for each $t\in S$. Therefore, $\{N_{\rho_t(K)}(\U) : t\in S \} \subseteq \{1,2,\ldots , N_{K}(\U)\}$ and so there exists $s\in S$ such that 
\begin{equation}\label{eq5.5}
N_{\rho_s (K)}(\U) = \min\{N_{\rho_t(K)}(\U) : t\in S \}.
\end{equation}
Let us show that this $s$ verifies \eqref{eq_claim1}. Indeed, pick a subfamily $\V$ of $\U$ that covers $E(\rho)$ and such that $N_{E(\rho)}(\U) =|\V|$. 
Then the subset $A=\bigcup \V$ of $K$ is open. Since $E(\rho) = \bigcap_{t\in S}\rho_t(K)\subseteq A$ and $K$ is compact, there is a finite subset $F\subseteq S$ such that $E(\rho)\subseteq\bigcap_{t\in F}\rho_t(K)\subseteq A$. As $\{\rho_t(K): t\in S\} $ is an inverse system, there exists a single $t\in S$ with 
$E(\rho)  \subseteq \rho_t(K) \subseteq A$.
Therefore, $N_{E(\rho)}(\U) \geq N_{\rho_t(K)}(\U) \geq N_{\rho_s(K)}(\U)$, by the choice of $\V$, $A$ and \eqref{eq5.5}. The converse inequality $N_{\rho_s(K)}(\U) \geq N_{E(\rho)}(\U)$ follows by the inclusion $E(\rho)  \subseteq \rho_s(K)$.
\end{proof}

We are ready to show that the computation of the topological entropy can be reduced to actions by surjective self-maps:

\begin{theorem}\label{Oldie}\label{E_rho_annihilator}\label{rhobarrho} 
In the above notation, $h_{\top}(\rho) = h_{\top}(\bar \rho)$.
\end{theorem}
\begin{proof} 
In view of Corollary~\ref{coro_top_by_surj}, it is enough to see that $h_{\top}(\rho) \leq h_{\top}(\bar\rho)$. Hence, let $\{F_i\}_{i\in I}$ be a right F\o lner net for $S$, $\U\in\cov(K)$ and, for every $i\in I$, let $\V_i= \U_{\rho, F_i}$. By Lemma~\ref{claim1} there exists $s_i\in S$ such that  
\begin{equation}\label{eq5}
N_{K}(\rho_{s_i}^{-1}(\V_i))  = N_{E(\rho)}(\V_i).
\end{equation}
Next we observe that  
\[
\rho_{s_i}^{-1}(\V_i)=\rho_{s_i}^{-1}\left(\bigvee_{f\in F_i}\rho_f^{-1} (\U)\right)=\bigvee_{f\in F_i}\rho_{s_i}^{-1}(\rho_f^{-1} (\U))=\bigvee_{f\in F_i}(\rho_f\rho_{s_i})^{-1} (\U)=\bigvee_{f\in F_i}\rho_{s_if}^{-1} (\U)= \U_{\rho, s_iF_i}.
\] 
The latter equality and (\ref{eq5}) give 
\begin{equation}\label{eq6}
N_{K}(\U_{\rho, s_iF_i})  = N_{E(\rho)}(\U_{\rho, F_i}).
\end{equation}
The net $\{s_iF_i\}_{i\in I}$ is right F\o lner by Lemma~\ref{fol} with $|s_iF_i| = |F_i|$ (as $S$ is cancellative), so, from (\ref{eq6}), we get:
\[
H_{\top}(\rho,\U) = \lim_{i\in I} \frac{\log N_{K}(\U_{\rho, s_iF_i})}{|s_iF_i| }= \lim_{i\in I} \frac{\log N_{E(\rho)}(\U_{\rho, F_i})}{|F_i| } = 
H_{\top}(\bar \rho,\U) \leq h_{\top}(\bar\rho).
\]
Since $\U\in\cov(K)$ was chosen arbitrarily, we deduce that $h_{\top}(\rho) \leq h_{\top}(\bar \rho)$, as desired.
\end{proof}

\begin{remark}\label{rem:Halmos}
According to the well-known Halmos' paradigm, an endomorphism of a compact group is measure-preserving with respect to the Haar measure if and only if it is surjective. 
Therefore, {when applied to a right action $K\ra{\rho}S$ on a compact group $K$}, the above theorem suggests how to pass from $\rho$ 
to the continuous and measure-preserving right $S$-action $E(\rho)\ra{\bar \rho}S$. So, one can also discuss the measure entropy of such an action; it is known that for actions of amenable countable groups on compact metrizable groups the topological and the measure entropy coincide \cite[Proposition 13.2]{li_kerr_book}.
\end{remark}

\subsection{Invariance under Ore colocalization for $h_\top$}\label{subs_dual_ore}


We are now going to modify the right $S$-action $\bar K\ra{\bar\rho}S$ introduced in the previous section to make it into a new action by invertible maps:

\begin{lemma}\label{def_of_k_star_lemma}
Consider the following inverse system ${\frak K}=\{(K_g, \bar\rho_s\colon K_{g}\to K_{gs}):g\in G,\ s\in S\}$, where $K_g=\bar K$ for all $g\in G$, and let $K^*=\varprojlim_{(G,\leq_S)}\mathfrak K$, with $\pi_g=\pi_g^K\colon K^*\to K_g$ the canonical map from the inverse limit, for all $g\in G$. Then:
\begin{enumerate}[(1)]
   \item $\pi_g$ is surjective and $\pi_{gs}=\bar\rho_s\circ\pi_g$, for all $g\in G$ and $s\in S$;
   \item let $\rho^*_g\colon K^*\to K^*$ be the unique continuous map such that the following diagram commutes for all $g,h\in G$:
\begin{equation}\label{univ_prop_act_comp_eq}
\xymatrix@C=50pt{
K_{gh}\ar[r]^{\id_{\bar K}}\ar@{<-}[d]_{\pi_{gh}}&K_{h}\ar@{<-}[d]^{\pi_{h}}\\
K^*\ar[r]_{\rho^*_g}&K^*.}
\end{equation}
Then, $K^*\ra{\rho^*} G$ is a right $G$-action;
   \item $\pi_1\colon K^*\to \bar K$ is (surjective and) $S$-equivariant, when $K^*$ is endowed with the restriction  $(\rho^*)_{\restriction S}$ of the action $\rho^*$ to $S$; 
   \item given a second right $S$-action $K'\ra{\rho'}S$ on a compact space $K'$ and an $S$-equivariant continuous map $\phi\colon K\to K'$, there is a unique continuous map $\phi^*\colon K^*\to (K')^*$ such that, for every $g\in G$, the following diagram commutes
\[
\xymatrix@C=50pt{K^*\ar[r]^{\phi^*}\ar[d]_{\pi^K_g}&(K')^*\ar[d]^{\pi^{K'}_g}\\
\bar K\ar[r]_{\bar\phi}&\bar K'.}
\]
Furthermore, $\phi^*$ is $G$-equivariant and if $\phi$ is injective (resp., surjective) then so is $\phi^*$.
\end{enumerate}
\end{lemma}
\begin{proof}
(1) follows by Lemma \ref{description_B*}.

\smallskip
(2) Given $g\in G$, the family of maps $\{\pi_{gh}\colon K^*\to K_h:h\in G\}$ is compatible with the transition maps in $\frak K$. Hence, there is a unique continuous map $\rho^*_g\colon K^*\to K^*$ such that 
\begin{equation}\label{equational_un_prop_rho_star_eq}
\pi_h\circ \rho^*_g=\pi_{gh},\quad\text{for all $h\in G$.} 
\end{equation}
Let us show that this defines a right action $K\ra{\rho^*}G$. Indeed, $\rho_1^*=\id_{K^*}$ since $\pi_h\circ \id_{K^*}=\pi_{h}$ holds for all $h\in G$, that is, $\id_{K^*}$ satisfies the universal property  \eqref{equational_un_prop_rho_star_eq} with $g=1$. Similarly, given $g_1,g_2\in G$, we have that $\rho_{g_1}^*\circ\rho_{g_2}^*=\rho_{g_2g_1}^*$, in fact, the equalities
\[
\pi_h\circ (\rho_{g_1}^*\circ\rho_{g_2}^*)=(\pi_h\circ \rho_{g_1}^*)\circ\rho_{g_2}^*=\pi_{g_1h}\circ\rho_{g_2}^*=\pi_{(g_2g_1)h}
\]
hold for all $h\in G$. In particular, $\rho_{g_1}^*\circ\rho_{g_2}^*$ satisfies the universal property \eqref{equational_un_prop_rho_star_eq} with $g=g_2g_1$.

\smallskip
(3) For each $s\in S$ and $x\in K^*$, we have that $(\pi_1\circ\rho^*_s)(x)=\pi_s(x)={(\bar \rho_s \circ \pi_1)}(x)$, where the former equality comes from the commutative diagram \eqref{univ_prop_act_comp_eq} with $h=1$ and $g=s$, while the latter comes from part (1) with $g=1$.

\smallskip
(4) The existence and uniqueness of $\phi^*$ are clear by the universal property of inverse limits, while $\phi^*$ is clearly \mbox{$G$-equivariant} by construction. Furthermore, by Corollary~\ref{coro_top_by_surj}(2), $\bar\phi$ is injective (resp., surjective) whenever $\phi$ has the same property. One can then conclude by Lemma~\ref{lim_is_exact_for_comp_lemma}.
\end{proof}

\begin{definition}\label{def_k_star}
The {\em (left) Ore colocalization} of $K\ra{\rho} S$ is the right $G$-action $K^*\ra{\rho^*} G$ on the compact space $K^*$, built in Lemma \ref{def_of_k_star_lemma}. 
\end{definition}

Let us remark that, by part (4) of the above lemma, the assignment $(K,\rho)\mapsto (K^*,\rho^*)$ is part of a functor
\begin{equation}\label{EqMay21}
{(-)^*}\colon \rre(S,\CompS)\to \rre(G,\CompS),
\end{equation}
that sends continuous injective (resp., surjective) $S$-equivariant maps to continuous injective (resp., surjective) $G$-equivariant maps. 
Furthermore, this functor restricts to exact functors
\[{(-)^*}\colon \rre(S,\CompG)\to \rre(G,\CompG)\quad\text{and}\quad {(-)^*}\colon \rre(S,\Comp)\to \rre(G,\Comp).\]
In fact, the Ore colocalization is even ``exact'' on ``generalized short exact sequence'' of the form $H\hookrightarrow K\twoheadrightarrow K/H$, where $K$ is a compact group with a right $S$-action, $H$ is an $S$-invariant closed subgroup, and $K/H$ is the $K$-space of left $H$-cosets.

\begin{corollary}\label{strong_exactness_coro}
Let $K\ra\rho S$ be a right $S$-action on a compact group $K$, let $H\leq K$ be a closed $S$-invariant subgroup and let $H\ra{\rho_H}S$ and $K/H\longra{\rho_{K/H}}S$ be the right $S$-actions (by continuous endomorphisms, and by continuous self-maps, respectively) induced by $\rho$ on $H$ and on $K/H$, respectively. If we denote by $\iota\colon H\to K$ the inclusion and by $\pi\colon K\to K/H$ the projection, then:
\begin{enumerate}[(1)]
\item $\iota^*\colon H^*\to K^*$ is an injective, $G$-equivariant, continuous and closed group homomorphism, i.e.,  the action $H^*\longra{\ \ (\rho_H)^*}G$ is conjugated to the action $\iota^*(H^*)\longra{(\rho^*)_{\iota^*(H^*)}}\ G$;
\item $\pi^*\colon K^*\to (K/H)^*$ is a surjective, $G$-equivariant, continuous and open map;
moreover, the action $(K/H)^*\longra{(\rho_{K/H})^*}G$ is conjugated to the action $K^*/H^*\longra{(\rho^*)_{K^*/H^*}}\ \ G$ induced by $\rho^*$ on the space of left $H^*$-cosets.
\end{enumerate} 
\end{corollary}

\begin{proof} Recall that $\iota\colon H\to K$ is an $S$-equivariant closed embedding and that $\pi\colon K\to K/H$ is an $S$-equivariant surjective continuous and open map.

\smallskip
(1) According to Lemma \ref{def_of_k_star_lemma}(4), the continuous $G$-equivariant homomorphism $\iota^*\colon H^*\to K^*$ is injective, since $\iota$ is injective. Closedness follows from the compactness of $H^*$.

\smallskip
(2)  According to Lemma \ref{def_of_k_star_lemma}(4), the continuous $G$-equivariant homomorphism $\pi^*$ is surjective, since $\pi$ is surjective.  
Openness of  $\pi^*$ follows from the open mapping theorem which applies to all compact groups \cite[Corollary~8.4.2]{Book}. 
The second part of (2) follows from the first one.
\end{proof}

Let us conclude this section by proving that the topological entropy ``is invariant under Ore colocalization'':

\begin{theorem}[Invariance under Ore colocalization]\label{Th2} 
In the above notation, $h_{\top}(\rho)=  h_{\top}(\rho^*)$.
\end{theorem}
\begin{proof}
By Theorem~\ref{E_rho_annihilator}, it is enough to verify that $h_{\top}(\bar\rho) =  h_{\top}(\rho^*)$.
Fix, all along this proof, a right F\o lner net $\{F_i\}_{i\in I}$ for $S$, so that, $\{F_i\}_{i\in I}$ is also a right F\o lner net for $G$, by Lemma~\ref{Lemma:ex-2.2}(2). 

Let $\U\in\cov(K^*)$. By Corollary~\ref{dual_finite_shift}, there exist $s\in S$ and a finite open cover $\V$ of $K_s= \bar K$ such that $\U=\pi_s^{-1}(\V)$. 
We then have the following chain of equalities: 
\begin{equation}\label{Eq**:Feb16}
\rho^*_{s}(\U)= \rho^*_{s}(\pi_s^{-1}(\V))= (\pi_s\circ \rho^*_{s^{-1}})^{-1}(\V) \overset{{(*)}}{=} \pi_1^{-1}(\V),
\end{equation}
where $(*)$ follows since $\pi_s\circ \rho^*_{s^{-1}} = \pi_1$ (see \eqref{univ_prop_act_comp_eq} with $g=s^{-1}$ and $h=s$).
By \cite[Lemma 2.7(b)]{DFGB},  $\{F_is^{-1}\}_{i\in I}$ is a right F\o lner net for $G$.  Now, for each $i\in I$,
\begin{align}\label{last-inch}
\U_{\rho^*,F_is^{-1}}&=\bigvee_{f\in F_i}(\rho^*_{fs^{-1}})^{-1}(\U)=\bigvee_{f\in F_i}(\rho^*_{f})^{-1}\circ(\rho^*_{s^{-1}})^{-1}(\U)=\bigvee_{f\in F_i}(\rho^*_{f})^{-1}(\rho^*_{s}(\U))=(\rho^*_{s}(\U))_{\rho^*,F_i}.
\end{align}
Furthermore, using that $\pi_1\circ\rho_s^*=\bar\rho_s\circ \pi_1$, for all $s\in S$ (by Lemma~\ref{def_of_k_star_lemma}(3)), we obtain that: 
\begin{align}\label{very_last-inch}
(\pi_1^{-1}(\V))_{\rho^*,F_i}=\bigvee_{f\in F_i}(\rho^*_f)^{-1}(\pi_1^{-1}(\V))=\bigvee_{f\in F_i}\pi_1^{-1}((\bar\rho_f)^{-1}(\V))=\pi_1^{-1}\left(\bigvee_{f\in F_i}(\bar\rho_f)^{-1}(\V)\right)=\pi_1^{-1}(\V_{\bar\rho,F_i}).
\end{align}
Now, the chains of equalities \eqref{very_last-inch} and \eqref{last-inch} can be connected via \eqref{Eq**:Feb16}:
\[
\U_{\rho^*,F_is^{-1}}\ \overset{\eqref{last-inch}}{=}\ (\rho^*_{s}(\U))_{\rho^*,F_i}\ \overset{\eqref{Eq**:Feb16}}=\ (\pi_1^{-1}(\V))_{\rho^*,F_i}\ \overset{\eqref{very_last-inch}}{=}\ \pi_1^{-1}(\V_{\bar\rho,F_i}).
\]
Hence, $N(\U_{\rho^*,F_is^{-1}})=N(\pi_1^{-1}(\V_{\bar\rho,F_i})
)=N(\V_{\bar \rho,F_i})$, where the second equality uses the surjectivity of $\pi_1$. By the above equalities,
\[
H_{\top}(\rho^*,\U)=\lim_{i\in I}\frac{\log N(\U_{\rho^*,F_is^{-1}})}{|F_is^{-1}|}=\lim_{i\in I}\frac{\log N(\V_{\bar\rho,F_i})}{|F_i|}=H_{\top}(\bar\rho, \V)\leq h_{\top}(\bar\rho).
\]
Therefore, one gets $ h_{\top}(\rho^*)\leq h_{\top}(\bar\rho)$. 

\smallskip 
By Lemma \ref{def_of_k_star_lemma}(3), $\pi_1\colon K^*\to \bar K$ is (surjective and) $S$-equivariant,  hence $h_{\top}(\rho^*_{\restriction S})\geq h_{\top}(\bar\rho)$ by Proposition~\ref{mono-top}(1). On the other hand, $h_{\top}(\rho^*) = h_{\top}(\rho^*_{\restriction S})$, by Remark \ref{RemarkMay16*}. This proves that $ h_{\top}(\rho^*)\geq h_{\top}(\bar\rho)$.
\end{proof}

\subsection{Categorical interpretation of the Ore colocalization}\label{categorical_coloc_subs}

As announced in \S\ref{categorial_ore_loc_subs}, it is possible to give a categorical interpretation of the Ore colocalization by using the machinery of right Kan extensions (see \cite[\S 3.7]{Borceux} or \cite[\S X.3]{maclane}). Indeed, consider $S$ and $G$ as categories with one object (called $\star$, as in \S\ref{categorial_ore_loc_subs}) and note that $\rre(S,\CompS)$ and $\rre(G,\CompS)$ can be seen as categories of contravariant functors $S\to \CompS$ and $G\to \CompS$, respectively. In fact, a functor $F\colon S\to \CompS$ is uniquely determined by the compact space $K=F(\star)$ and by the right $S$-action $K\ra{\rho}S$ such that $\rho_s=F(s)$, for all $s\in S$. Similarly, one can view any contravariant functor $F'\colon G\to \CompS$ as a right $G$-action on the compact space $F'(\star)$. 

As usual, it is convenient to identify the categories of contravariant functors $\rre(S,\CompS)$ and $\rre(G,\CompS)$ with the categories of covariant functors $\lre(S^{\op},\CompS)$ and $\lre(G^{\op},\CompS)$, respectively, where $\C^\op$ denotes the opposite of a given category $\C$. 

Therefore, the inclusion $\iota^{\op}\colon S^\op\to G^\op$ induces a forgetful functor
\[
\iota^*\colon \lre(G^\op,\CompS)\to \lre(S^\op,\CompS),\quad\text{such that}\quad F\mapsto F\circ \iota^\op.
\]
In the language of right actions, this means that $\iota^*$ sends a right $G$-action $K\ra{\rho}G$ to the right $S$-action $K\ra{\rho_{\restriction S}}S$ (which just ``forgets'' part of the action), that is, $\iota^*$ is naturally isomorphic to the inclusion $\rre(G,\CompS)\to \rre(S,\CompS)$.}

By the dual of \cite[Theorem~3.7.2]{Borceux} (see also \cite[Theorem~1 in \S X.3]{maclane}), and since $\CompS$ has all limits, the functor $\iota^*$ has a right adjoint functor:
\[
\iota_*\colon \lre(S^\op,\CompS)\to \lre(G^\op,\CompS).
\] 
This right adjoint $\iota_*$ to the forgetful functor is called the right Kan extension along $\iota^\op$. In fact, given a covariant functor $F\colon S^\op\to \CompS$, which corresponds uniquely to the right $S$-action $F(\star)=(K\ra{\rho}S)$, there is an explicit construction of the right Kan extension $\iota_*(F)\colon G^\op\to \CompS$ as the limit of a suitable inverse system, indexed by a specific comma category (for the general result, see \cite[Theorem~1 in \S X.3]{maclane}). In what follows we suggest how to recover our construction of $K^*\ra{\rho^*}G$ in \S\ref{subs_red_top_to_surj} and \ref{subs_dual_ore} from the general theory of the right Kan extensions. 

In other words, we show that the action $K^*\ra{\rho^*}G$ is the right Kan extension of $K\ra{\rho}S$ or, equivalently, that $K^*\ra{\rho^*}G$, together with the $S$-equivariant morphism $K^*\to K$ obtained as a composition of the projection $\pi_1\colon K^*\to \bar K$ and the inclusion $\bar K \to K$, is a coreflection of $(K\ra{\rho}S)\in \rre(S,\CompS)$ in $\rre(G,\CompS)$.

We start recalling the pointwise construction of the right Kan extension of a functor $F\colon S\to \CompS$ as a limit: 
\begin{itemize}
\item  consider the comma category $\star/\iota^\op$, that is a category with objects $\Ob(\star/\iota^\op)=\{(\star, g):g\in \End_{G^\op}(\star)\}$ and where a morphism $(\star,g)\to (\star,g')$ is just a morphism $s\in \End_{S^\op}(\star)$ such that $sg=g'$ (in $G^\op$);
\item we denote by $p\colon\star/\iota^\op\to S^{\op}$ the projection onto the first component of the objects in $\star/\iota^\op$, and we take the composition $F\circ p\colon \star/\iota^\op\to \CompS$;
\item one can then compute the right Kan extension via the following limit 
\[
\iota_*(F)(\star)\cong \lim{}_{\star/\iota^\op}(F\circ p).
\]
\end{itemize} 
Let us try to make the above construction more explicit: the category $\star/\iota^\op$ is easily seen to be equivalent to the category $(G,\leq_S)^\op$, such that $\Ob((G,\leq_S)^\op)=\{g:g\in G\}$ and 
\[
\hom_{(G,\leq_S)^\op}(g,g')=\begin{cases}\{s\}&\text{if, and only if, $g'=gs$ in $G$;}\\
\emptyset&\text{otherwise.}\end{cases}
\] 
In particular, $\hom_{(G,\leq_S)^\op}(g,g')$ has one element precisely when $g'\leq_Sg$, while it is empty otherwise. Furthermore, the functor $p\colon (G,\leq_S)^\op\to S^\op$ sends each object $g\in\Ob((G,\leq_S)^\op)$ to the unique object $\star\in \Ob(S^{\op})$, while it sends the unique morphism $s\in \hom_{(G,\leq_S)^\op}(g,gs)$ to the endomorphism $p(s)=s\in \End_{S^\op}(\star)$.

Therefore, the diagram $F\circ p\colon (G,\leq_S)^{\op}\to \CompS$ is such that $(F\circ p)(g)=K$ for all $g\in G$. Furthermore, for each $g\in G$ and $s\in S$, we have that $(F\circ p)(s\colon g\to gs)=(\rho_s\colon K\to K)$. Hence, $\iota_*(F)(\star)$ is the inverse limit  in $\CompS$ of an inverse system of copies of $K$, indexed by $(G,\leq_S)$, with connecting maps given by suitable $\rho_s$.

As discussed in Lemma~\ref{description_B**}, there is a canonical way to ``transform'' the diagram $F\circ p$ into a new diagram with surjective transition maps but with the same inverse limit. Indeed, one defines:
\[K_g=\bar K=\bigcap_{s\in S}\rho_s(K)\quad\text{and}\quad  \bar\rho_s=(\rho_s)_{\restriction \bar K}\colon K_{g}\to K_{gs},\]
for all $g\in G$ and $s\in S$. Then, the inverse system ${\frak K}=\{(K_g, \bar\rho_s\colon K_{g}\to K_{gs}):g\in G,\ s\in S\}$ is the same that we have used in Lemma~\ref{def_of_k_star_lemma} and, therefore, by the above discussion and by Lemmas~\ref{description_B**} and~\ref{def_of_k_star_lemma}, we obtain the desired isomorphisms:
\[K^*=\varprojlim{}_{(G,\leq_S)}{\frak K}\cong \varprojlim (F\circ p)\cong \iota_*(F)(\star).\]

\section{The Topological Addition Theorem}\label{AppB}

The aim of this section is to prove the Topological Addition Theorem, as stated in the introduction.
First we prove it for actions of amenable groups; we then deduce the general case by using the invariance of the topological entropy under Ore colocalization from \S\ref{subs_dual_ore}.

\subsection{Properties of open covers of compact spaces related to projections
}\label{N--}

 In this subsection we deal with an amenable group $G$, compact spaces $K$ and $Q$, two right $G$-actions $K\ra{\rho}G$ and $Q\ra{\rho_Q}G$, and a $G$-equivariant surjective continuous map $\pi\colon K\to Q$. To simplify our heavy notations, we let, for each $\U\in \cov(K)$, $\W\in \cov(Q)$ and $F\in \Pf(G)$,
\[\U_F=\U_{\rho,F}=\bigvee_{f\in F}\rho_{f}^{-1}(\U)\quad\text{and}\quad \W_F=\W_{\rho_Q,F}=\bigvee_{f\in F}(\rho_Q)_{f}^{-1}(\W);\]
 moreover, for $f\in F$, $\U_f=\rho_{f}^{-1}(\U)$ and $\W_f=(\rho_Q)_{f}^{-1}(\W)$.

The reader may keep in mind that the properties established in this section are needed in \S\ref{ProofAT:top} for the specific case when $K\in \CompG$, $H$ is a closed $G$-invariant subgroup of $K$, $Q=K/H$ is the left cosets space and $\pi\colon K \to Q$ is the $G$-equivariant projection. Nevertheless, we keep the more general setting here since we feel that also appropriate counterparts of  Propositions~\ref{Propo_next-to-last} and~\ref{Propo_last} for compact spaces can be proved using uniform covers in place of even ones in that more general setting.

\begin{definition}
Let  $K$ and $Q$ be compact spaces, $\pi\colon K\to Q$ a surjective continuous map, and $\W\in \cov(Q)$.  Define:
\[N(\U|\W)=\sup_{W\in \W}N_{\pi^{-1}(W)}(\U)\quad\text{and}\quad N(\U|\pi)=\sup_{q\in Q}N_{\pi^{-1}(q)}(\U).\]
\end{definition}

When $K$ is a compact group and $Q=K/H$, where $H$ is a closed subgroup of $K$, we get $N(\U|\pi)=\sup_{k\in K}N_{kH}(\U)$.

\begin{proposition}\label{tech_proposition_conditional} In the above setting, the following properties of $N(-|-)$ hold true:
\begin{enumerate}[(1)]
   \item $1 = N(\U_0|\W) \leq N(\U_1|\W)\leq N(\U_2|\W)$, if $\{K\}=\U_0\preceq \U_1\preceq \U_2$ in $\cov(K)$ and $\W\in\cov(Q)$;
   \item $N(\U|\W_2)\leq N(\U|\W_1)\leq N(\U|\W_0)= N(\U)$, if $\U\in \cov(K)$ and $\{Q\}=\W_0\preceq\W_1\preceq\W_2$ in $\cov(Q)$;
   \item $\log N(\U)\leq \log N(\W) + \log  N(\U|\W) $, for all $\U\in \cov(K)$ and $\W\in \cov(Q)$;
   \item $\log N(\U_1\vee\U_2|\W)\leq \log N(\U_1|\W) + \log N(\U_2|\W)$, for all $\U_1,\U_2\in \cov(K)$ and $\W\in\cov(Q)$. 
\end{enumerate}
Furthermore, $N(-|-)$ has the following relations with the invariant $N(-|\pi)$:
\begin{enumerate}[(1)]
\setcounter{enumi}{4}
   \item for each $\U\in \cov(K)$, there exists $q\in Q$, such that $N(\U|\pi)=N_{\pi^{-1}(q)}(\U)$, so $N(\U|\W)\geq N(\U|\pi)$, for all $\U\in\cov(K)$ and $\W\in\cov(Q)$;
   \item for every finite subset $\{\U_1, \ldots, \U_n\}\subseteq \cov(K)$,
there exists $\W\in \cov(Q)$ such that $N(\U_{t}|\W)=N(\U_{t}|\pi)$ for all $t=1,\dots,n$ (and so, $N(\U_{t}|\W')=N(\U_{t}|\pi)$ for all $\W\preceq\W'\in\cov(Q)$).
\end{enumerate}
\end{proposition}
\begin{proof} 
(1) Given $W\in \W$, it is clear that $\pi^{-1}(W)\subseteq K$, so that $N(\U_0|\W)=1$. Now, if $\U_2'=\{U_{2,j}:{j=1,\dots, n}\}\subseteq \U_2$ is a subset such that ${\pi^{-1}(W)}\subseteq \bigcup_{j=1}^nU_{2,j}$, there is $\U_1'=\{U_{1,j}:{j=1,\dots,n}\}\subseteq \U_1$ such that $U_{2,j}\subseteq U_{1,j}$ for all {$j=1,\dots, n$}, and therefore $N_{\pi^{-1}(W)}(\U_1)\leq N_{\pi^{-1}(W)}(\U_2)$. As this holds for all $W\in \W$, we deduce that $N(\U_1|\W)\leq N(\U_2|\W)$. 
 Since $\U_0\preceq \U_1$, this also shows that $N(\U_0|\W) \leq N(\U_1|\W)$.

\smallskip
(2) For each $W\in \W_2$, there exists $W'\in \W_1$ such that $W\subseteq W'$, and therefore, $N_{{\pi^{-1}(W)}}(\U)\leq N_{{\pi^{-1}(W')}}(\U)$. As this holds for all $W\in \W_2$, we deduce that $N(\U|\W_2)\leq N(\U|\W_1)$. Similarly, as $\W_0\preceq\W_1$, also $N(\U|\W_1)\leq N(\U|\W_0)$ and, by definition, $N(\U|\W_0)=N_{\pi^{-1}(Q)}(\U)=N_K(\U)=N(\U)$.

\smallskip
(3) Suppose that $N(\W)=n$ and let $\W'=\{W_1,\dots,W_n\}\subseteq \W$ be a minimal subcover. Choose, for  each {$k=1,\dots,n$}, a minimal subset $\U_i=\{U_{k,1},\dots,U_{k,m_k}\}\subseteq \U$ that covers $\pi^{-1}(W_k)$. Then, 
\[K=\pi^{-1}(Q)=\bigcup_{k=1}^n\pi^{-1}(W_k)=\bigcup_{k=1}^n\left(\bigcup_{j=1}^{m_k}U_{k,j}\right),\]
and therefore, $\U'=\bigcup_{k=1}^n\U_k$ is a subcover of $\U$, showing that $N(\U)\leq |\U'|\leq \sum_{k=1}^nm_k$. By construction, we have that $m_k=N_{\pi^{-1}({W}_k)}(\U)\leq N(\U|\W)$ and $n=N(\W)$, thus $N(\U)\leq n\cdot N(\U|\W)=N(\W)\cdot N(\U|\W)$. To conclude, apply logarithms on both sides of the equation.

\smallskip
(4) For each $W\in \W$, let $\U_{1,W}\subseteq \U_1$ and $\U_{2,W}\subseteq \U_2$ be two minimal subsets that cover $\pi^{-1}(W)$. Then, by definition, $\U_{1,W}\vee\U_{2,W}=\{U_1\cap U_2:U_1\in \U_{1,W},\ U_2\in \U_{2,W}\}$ is a subset of $\U_1\vee \U_2$ that covers $\pi^{-1}(W)$. In particular, $N_{\pi^{-1}(W)}(\U_1)\cdot N_{\pi^{-1}(W)}(\U_2)=|\U_{1,W}|\cdot |\U_{2,W}|\geq |\U_{1,W}\vee\U_{2,W}|\geq N_{\pi^{-1}(W)}(\U_{1}\vee\U_{2})$. As this holds for all $W\in\W$, we deduce that $N(\U_{1}\vee\U_{2}|\W)\leq N(\U_1|\W)\cdot N(\U_2|\W)$. To conclude, just take the logarithm of this inequality.

\smallskip
(5) Given $\U\in \cov(K)$, the set $\{N_{\pi^{-1}(q)}(\U):q\in Q\}\subseteq \{0,\dots,N(\U)\}$ is a finite subset of $\N$, so its supremum $N(\U|\pi)$ is a maximum, and we can choose $q \in Q$ such that $N(\U|\pi)=N_{\pi^{-1}(q)}(\U)$. Choose also $W\in \W$ such that $q \in W$, then $\pi^{-1}(q) \subseteq \pi^{-1}(W)$, and so $N(\U|\W)\geq N_{\pi^{-1}(W)}(\U) \geq N_{\pi^{-1}(q)}(\U) = N(\U|\pi)$.

\smallskip
(6) Consider first the case $n=1$, i.e., we have to see that for each $\U\in \cov(K)$, there exists $\W\in \cov(Q)$ such that $N(\U|\W)= N(\U|\pi)$. 
Let $\U\in \cov(K)$. For each $q\in Q$, choose a subset $\U_q\subseteq \U$ such that $|\U_q|=N_{\pi^{-1}(q)}(\U)$ and 
$\pi^{-1}(q)\subseteq \bigcup \U_q=A_q$; clearly, $A_q\subseteq K$ is open. We claim that there is an open neighborhood $W_q$ of $q\in Q$ such that  
\begin{equation}\label{Eq:June14}
\pi^{-1}(q)\subseteq \pi^{-1}(W_q)\subseteq A_q.
\end{equation}
Indeed, being $Q$ a regular space, there is a family of open neighborhoods $\{O_i:i\in I\}$ of $q$ such that 
$\bigcap_{i\in I} \cl (O_i) = \{q\}$. Then, $\bigcap_{i\in I} \pi^{-1}(\cl( O_i )) = \pi^{-1}(q)\subseteq A_q$. 
Since $K$ is compact, $A_q$ is open and $\pi^{-1}(\cl(O_i))$ is closed, for all $i\in I$, there is a finite subset $J \subseteq I$ such that $\pi^{-1}(q) \subseteq \bigcap_{i\in J} \pi^{-1}(\cl( O_i )) \subseteq A_q$. Then, $W_q= \bigcap_{i\in J} O_i $ is an open neighborhood of $q$ in $Q$ that  satisfies \eqref{Eq:June14}. 
Let $\W=\{W_q \colon q\in Q\}\in\cov(Q)$. Then, for each $q\in Q$, 
\[N_{\pi^{-1}(W_q)}(\U) \leq N_{A_q}(\U)\leq|\U_q|=N_{\pi^{-1}(q)}(\U)\leq N(\U|\pi).\]
As this holds for all $q\in Q$, we deduce that $N(\U|\W)\leq N{(\U|\pi)}$. Equality holds by item (5).

Now we consider the general case. By the first part of the argument and item (2), for each $t=1,\dots,n$, there exists $\W_t\in \cov(Q)$ such that $N(\U_{t}|\W')=N(\U_{t}|\pi)$, for any $\W'$ that refines $\W_t$. Hence, letting $\W=\bigvee_{t=1}^n\W_t$, we have that
$N(\U_t|\W)=N(\U_t|\pi)$ for all $t=1,\dots,n$.

For the last assertion, combine with (2) and (5).
\end{proof}

Consider the following consequences of Proposition~\ref{tech_proposition_conditional}:

\begin{corollary}\label{tech_corollary_conditional}
Let $G$ be an amenable group, $K$ and $Q$ compact spaces, $K\ra{\rho}G$ and $Q\ra{\rho_Q}G$ two right $G$-actions, and suppose that the surjective continuous map $\pi\colon K\to Q$ is $G$-equivariant. Let $\U\in\cov(K)$ and $\W\in \cov(Q)$. Then:
\begin{enumerate}[(1)]
   \item $N(\U_{gF}|\pi)=N(\U_{F}|\pi)$ and $N(\U_{gF_1}|\W_{F_2})=N(\U_{F_1}|\W_{g^{-1}F_2})$, for all $F, F_1,F_2\in \Pf(G)$ and $g\in G$;
   \item $\log N(\U_F|\W) \leq  |F| \cdot \log N(\U)$, for all $F \in  \Pf(G)$;
   \item $\log N(\U_F|\W_F)\leq \log N(\U_{F_1}|\W_F)+\log N(\U_{F_2}|\W_F)\leq \log N(\U_{F_1}|\W_{F_1})+\log N(\U_{F_2}|\W_{F_2})$, for all $F_1,F_2\in\Pf(G)$ and $F=F_1\cup F_2$;
   \item $\log N(\U_F|\pi)\leq \log N(\U_{F_1}|\pi)+\log N(\U_{F_2}|\pi)$, for all $F_1,F_2,F=F_1\cup F_2\in\Pf(G)$.
\end{enumerate}
\end{corollary}
\begin{proof}
(1) First we prove the second assertion: take the bijection $\W_{F_2}\to \W_{g^{-1}F_2}$ given by $W\mapsto (\rho_Q)_g(W)$, for all $W$ in $\W_{F_2}$. Furthermore, given $W\in \W_{F_2}$,  $\U'\subseteq \U_{gF_1}$ covers $\pi^{-1}(W)$ if and only if $\U''=\rho_g(\U')=\{\rho_g(U):U\in \U'\}\subseteq \U_{F_1}$ covers $\pi^{-1}((\rho_Q)_{g}(W))$. In fact, 
\[\pi^{-1}(W)\subseteq \bigcup\U'\quad \text{if and only if}\quad \pi^{-1}((\rho_Q)_{g}(W))=\rho_g(\pi^{-1}(W))\subseteq \rho_g\left( \bigcup_{U\in \U'}U\right)=\bigcup_{U\in\U'}\rho_g(U)=\bigcup\U''.\]

The proof of the first assertion follows the same line. In fact, the map $\U_{F}\to \U_{gF}$ such that $U\mapsto \rho_{g^{-1}}(U)$, for all $U\in\U_F$, is a bijection. Moreover, given $q\in Q$, a subset $\U'\subseteq \U_F$ covers $\pi^{-1}(q)$ if and only if $\U''=\rho_{g^{-1}}(\U')\subseteq \U_{gF}$ covers $\pi^{-1}((\rho_Q)_{g^{-1}}(q))$; and $(\rho_Q)_{g^{-1}}\colon Q\to Q$ is a bijection.

\smallskip
(2) By Proposition~\ref{tech_proposition_conditional}(4), we deduce that $\log N(\U_F|\W) \leq \sum_{f\in F}\log N(\U_f|\W)$. Furthermore, by part (1) and Proposition~\ref{tech_proposition_conditional}(2), $N(\U_f|\W)=N(\U|\W_{f^{-1}})\leq N(\U)$, for all $f\in F$. Hence, 
\[
\log N(\U_F|\W) \leq \sum_{f\in F}\log N(\U|\W_{f^{-1}})\leq |F|\cdot \log N(\U).
\]

\smallskip
(3) The statement follows by the following formula, where we apply Proposition~\ref{tech_proposition_conditional}(4) and (2) in the inequalities:
\[\log N(\U_F|\W_F)=\log N(\U_{F_1}\vee\U_{F_2}|\W_F)\leq \log N(\U_{F_1}|\W_F)+\log N(\U_{F_2}|\W_F)\leq \log N(\U_{F_1}|\W_{F_1})+\log N(\U_{F_2}|\W_{F_2}).\]

(4) By Proposition~\ref{tech_proposition_conditional}(6), there exists $\W\in\cov(Q)$ such that $N(\U_F|\W)=N(\U_F|\pi)$, {$N(\U_{F_1}|\W)=N(\U_{F_1}|\pi)$ and $N(\U_{F_2}|\W)=N(\U_{F_2}|\pi)$}. Hence, making use of Proposition~\ref{tech_proposition_conditional}(4), we get
\begin{equation*}
\log N(\U_F|\pi)=\log N(\U_F|\W) =\log N(\U_{F_1}\vee\U_{F_2}|\W) 
\leq \log N(\U_{F_1}|\W)+\log N(\U_{F_2}|\W) 
= \log N(\U_{F_1}|\pi)+\log N(\U_{F_2}|\pi).  \qedhere 
\end{equation*}
\end{proof}

The above corollary implies that, given $\U\in\cov(K)$ and $\W\in \cov(Q)$, the  two functions $\Pf(G)\to \R_{\geq0}$, defined by 
$$F\mapsto \log N(\U_F|\W_F)\ \mbox{ and } \ F\mapsto \log N(\U_F|\pi),$$
are both subadditive and left-invariant. So, they satisfy the hypotheses of the Ornstein-Weiss Lemma, hence, for every right F\o lner net $\s=\{F_i\}_{i\in I}$ for $G$, these two nets converge:
$$\left\{\frac{\log N(\U_{F_i}|\pi)}{|F_i|}\right\}_{i\in I} \quad \text{and}\quad \left\{\frac{\log N(\U_{F_i}|\W_{F_i})}{|F_i|}\right\}_{i\in I}.$$

\begin{lemma}\label{lemma_very_close_cotrajectories}
In the above setting, let $\s=\{F_i\}_{i\in I}$ be a right F\o lner net for $G$ and  $\varepsilon>0$. Then,  there exists $\bar\imath\in I$ such that
\begin{equation}\label{EqJune22}
\log N(\U_{F_j}|\pi) \leq  \frac{|F_j|}{|F_i|}  \log N(\U_{F_i}|\pi) + \varepsilon\ |F_j|,\ \ \text{for all $i,j\geq\bar\imath'$ in $I$}.
\end{equation}
\end{lemma}
\begin{proof} 
Being convergent, $\{\log N(\U_{F_k}|\pi)/|F_k|\}_{k\in I}$ is a Cauchy net in $\R_{\geq0}$. Therefore, there is some $\bar\imath\in I$ such that 
\[\left|\frac{\log N(\U_{F_j}|\pi)}{|F_j|}-\frac{\log N(\U_{F_i}|\pi)}{|F_i|}\right| \leq \varepsilon, \ \ \text{for all $\bar\imath\leq i, j\in I$.}\]
We then obtain the following inequality, which clearly implies \eqref{EqJune22}:
\[\frac{\log N(\U_{F_j}|\pi)}{|F_j|}  \leq \frac{\log N(\U_{F_i}|\pi)}{|F_i|} + \varepsilon, \ \ \text{for all $\bar\imath\leq i, j\in I$.}  \qedhere\]
\end{proof}

\subsection{Even covers and their properties}\label{evensec}

For the rest of this section, fix a compact  group $K$ with a right $G$-action $K\ra{\rho} G$ and a closed $S$-invariant subgroup $H\leq K$. Necessarily, $\rho_g(H)=H$ for every $g\in G$, since $\rho_g(H)\leq H$ and $\rho_{g^{-1}}(H)\leq H$ for every $g\in G$ implies $H\leq\rho_g(H)$ for every $g\in G$.
Denote by $\pi\colon K\to K/H$ the canonical projection on the space of left cosets $K/H$, and by $H\ra{\rho_H}\ G$ and $K/H\longra{\rho_{K/H}} G$ the induced right $G$-actions on $H$ and $K/H$, respectively.

\smallskip
An open cover $\U\in\cov(K)$ is said to be ({\em left}) {\em even} if  $\U=\{xU:x\in K\}$ for some open neighborhood $U$ of $1\in  K$. Let
\begin{equation}\label{EvenCov}
\U_{ K}[U]=\{xU:x\in K\}\quad\text{and}\quad \U_{K/H}[U]=\{\pi(xU):x\in K\}
\end{equation}
the even cover of $K$ associated with $U$ and the even cover of $K/H$ associated with $U$, respectively. When it is clear from the context, we simply write $\U[U]$ instead of $\U_K[U]$.

\smallskip
First, we see that for even covers $\U\in\cov(K)$ the quantity $N(\U|\pi)$ introduced in \S\ref{N--} has a simple and clear meaning with respect to $H$:

\begin{lemma}\label{NpiN} 
In the above setting, let $\U=\U[V]\in \cov(K)$ be an even cover and $F\in\Pf(G)$. Then, $N(\U_F|\pi)=N_H(\U_F)$.
\end{lemma}
\begin{proof} 
Recall that $N(\U|\pi)=\sup_{k\in K}N_{kH}(\U)$, so it suffices to prove that $N_{kH}(\U_F)\leq N_H(\U_F)$ for all $k\in K$.

Fix $k\in K$. Let $F=\{g_1,\ldots,g_n\}$ and $\rho_m=\rho_{g_m}$ for every {$m=1,\ldots,n$}. Assume that 
$$\V=\{V_j:j=1,\ldots,\ell\},\quad \text{with}\ V_j=\rho_1^{-1}(x_{j,1}V)\cap\ldots\cap\rho_n^{-1}(x_{j,n}V)\ \mbox{ and }\ x_{j,1},\ldots,x_{j,n}\in K,$$ 
is a finite subset of $\U_F$ with $H\subseteq\bigcup\V=V_1\cup\ldots\cup V_\ell$ and $N_H(\U_F)=|\V|=\ell$.
The containment $H\subseteq\bigcup\V$ implies that 
$$kH\subseteq k\bigcup\V=  k(V_1\cup\ldots\cup V_\ell)=kV_1\cup\ldots \cup kV_\ell=\bigcup k\V.$$
Now, for each {$m=1,\ldots,n$}, let $k_m=\rho_m(k)$, so that $k=\rho_m^{-1}(k_m)$ and hence, for every $j=1,\ldots,\ell$,
$$kV_j=k (\rho_1^{-1}(x_{j,1}V)\cap\ldots\cap\rho_n^{-1}(x_{j,n}V))=k\rho_1^{-1}(x_{j,1}V)\cap\ldots\cap k\rho_n^{-1}(x_{j,n}V)=\rho_1^{-1}(k_1 x_{j,1}V)\cap\ldots\cap\rho_n^{-1}(k_n x_{j,n}V)=W_j.$$
Let $\W=\{W_j:j=1,\ldots,\ell\}$; then $k\V=\W\subseteq \U_F$, $kH\subseteq\bigcup\W$ and $|\W|\leq\ell=N_H(\U_F)$, so $N_{kH}(\U_F)\leq N_H(\U_F)$.
\end{proof}

Furthermore, extending the notation from Example~\ref{Gactionsextop2}, for any subset $V$ of $K/H$ and $F\in\Pf(G)$, we let 
\[C_F(\rho_{K/H},V)=\bigcap_{f\in F}(\rho_{K/H})_f^{-1}(V) = T_F((\rho_{K/H})_{\frak U},V).\]

In the sequel we use the fact that every compact group is SIN (i.e., has small invariant neighborhoods) in the sense of the following: 

\begin{fact}[{\cite[Corollary 1.12]{HM}}]\label{SIN}
Every compact group $K$ has a local base of invariant (under conjugation) neighborhoods of $1$,
i.e., $x^{-1}Vx \subseteq V$ for every $x\in K$.
\end{fact}  

If the neighborhood $V$ of $1\in K$ is invariant, then $xV=Vx$ for every $x\in K$ and  $C_F(\rho,V)$ is invariant for every $F\in\Pf(G)$.

\begin{lemma}\label{tech_AT}
In the above setting, the following statements hold true:
\begin{enumerate}[(1)]
   \item each $\U\in \cov(K/H)$ has an even refinement, that is, $\U\preceq \U_{K/H}[U]$ for some open neighborhood $U$ of $1\in K$ (in particular, $N(\U)\leq N(\U_{K/H}[U])$);
   \item for each open neighborhood $U$ of $1\in K$ and $F\in\Pf(G)$, we have $\U_{K/H}[C_F(\rho,U)]\subseteq T_F((\rho_{K/H})_\cov,\U_{K/H}[U])$ (in particular, $N(T_F((\rho_{K/H})_\cov,\U_{K/H}[U]))\leq N(\U_{K/H}[C_F(\rho,U)])$);
   \item  given two open neighborhoods $U$ and $V$ of $\ 1\in  K$ with $V^{-1}V\subseteq U$, and $F\in\Pf(G)$, then $T_F((\rho_{K})_\cov,\U_{K}[V])$ refines $\U_{K}[C_F( \rho,U)]$  (in particular, $N(\U_{K}[C_F(\rho,U)])\leq N(T_F((\rho_{K})_\cov,\U_{K}[V]))$);
   \item given two open neighborhoods $U$ and $V$ of $1\in K$ such that $V=V^{-1}$ is invariant and $VV\subseteq U$, and $F\in\Pf(G)$, let $D$ be a subset of $K$ which is maximal with respect to the following property:
\begin{enumerate}[\rm ($*$)] 
  \item $\pi(d_1C_F(\rho,V))\cap \pi(d_2C_F( \rho,V))=\emptyset$, for all $d_1\neq d_2$ in $D$.
\end{enumerate}
Then, $\{\pi(dC_F(\rho,U)):d\in D\}$ is a subcover of $\U_{K/H}[C_F(\rho,U)]$, so that $N(\U_{K/H}[C_F(\rho,U)])\leq |D|$.
\end{enumerate}
\end{lemma}
\begin{proof} 
(1) follows directly from a general version of the Lebesgue Covering Lemma (see~\cite[Theorem~27]{Kelley}).

\smallskip
(2) Here we use the left action $\tau$ of $K$ on itself by left traslations $_xt\colon k\mapsto xk$ and the induced left action $\bar \tau$ of $K$ on $K/H$ by left traslations {$_x\bar t\colon kH\mapsto xkH$. For $A \subseteq K/H$, we shall briefly write $xA$ in place of $_x\bar t(A)$}.  Since the projection $\pi\colon K\to K/H$ is $K$-equivariant, for $x\in K$ and $B \subseteq K$ one has $x\pi(B)= \pi(xB)=\pi(x)\pi(B)$. Now it only remains to note that, for every $x\in K$, 
\begin{align*}
\pi(xC_F(\rho,U)) &=x\pi(C_F(\rho,U))= x\pi\left(\bigcap_{f\in F}\rho_f^{-1}(U)\right)\subseteq x \bigcap_{f\in F}\pi(\rho_f^{-1}(U))=\bigcap_{f\in F} x \pi(\rho_f^{-1}(U))= \\
&= \bigcap_{f\in F}  \pi(x \rho_f^{-1}(U))= \bigcap_{f\in F}\pi( \rho_f^{-1}(\rho_f(x)U))= \bigcap_{f\in F}(\rho_{K/H})_f^{-1}(\pi(\rho_f(x)U))\in T_F((\rho_{K/H})_\cov,\U_{K/H}[U]).
\end{align*}

(3)  It is enough to verify that, if $B\in T_F(\rho_\cov,\U_K[V])$ and $y\in B$, then $B\subseteq y C_F(\rho,U)\in \U_K[C_F(\rho,U)]$.
Indeed, take $x,y\in B$; for each $f\in F$ there exists $x_f\in K$ such that $B= \bigcap_{f\in F}\rho_f^{-1}(x_fV)$; hence, 
$x = \rho_f^{-1}(x_fw_f)$ and $y=\rho_f^{-1}(x_fv_f)$ for appropriate $v_f, w_f \in V$. Then, $y^{-1}x=\rho_f^{-1}(v_f^{-1}w_f)\in \rho_f^{-1}(V^{-1}V)\subseteq \rho_f^{-1}(U)$ for all $f\in F$, and so $x=y(y^{-1}x)\in  y C_F(\rho,U)$.

\smallskip
(4) It is enough to show that $ K/H\subseteq \bigcup_{d\in D}\pi(dC_F(\rho,U))$. Assume for a contradiction that $\pi(k)\notin \bigcup_{d\in D}\pi(dC_F(\rho,U))$ for some $k\in K$, so that $k\not\in D$, and we verify that $D\cup\{k\}$ satisfies $(*)$. 
In fact, if  $\pi(dt)=\pi(ks)\in \pi(dC_F(\rho,V))\cap \pi(kC_F(\rho,V))$ for some $t,s\in C_F(\rho,V)$ and $d\in D$, then there exists $h\in H$ such that $ks=dth$. Let $s'=hs^{-1}h^{-1}\in C_F(\rho,V)$ (by using that $V=V^{-1}$ is invariant), then $k=kss^{-1}=dths^{-1}=dts'h$, with $ts'\in C_F(\rho,V) C_F(\rho,V) \subseteq  C_F(\rho,VV)\subseteq C_F(\rho, U)$. Hence, $\pi(k)\in \pi(dC_F(\rho,U))$, which contradicts our assumption that $\pi(k)\notin \bigcup_{d\in D}\pi(dC_F(\rho,U))$. Therefore, $D\cup\{k\}$ satisfies $(*)$, against our assumption that $D$ is maximal for $(*)$.
\end{proof}

\subsection{Proof of the Topological Addition Theorem}\label{ProofAT:top}

The following proposition follows by the results of \cite{CCK}, that extend to the setting of cancellative and right amenable monoids the formalism of quasi-tilings of amenable groups introduced by Ornstein and Weiss in \cite{OW}. The reason to refer to \cite{CCK} for a statement about amenable groups is that it is a very accessible and detailed source for such a technical argument like quasi-tilings. Moreover, contrary to what we need in this section, many standard sources for amenable group theory in the literature are restricted to the countable case.

\begin{proposition}\label{tilings_theo}
Let $G$ be an amenable group and $\s=\{F_i\}_{i\in I}$ a right F\o lner net for $G$. Then, for each $\varepsilon\in (0,1/2)$ and $\bar \jmath \in I$, there exist $n\in\N_+$, $\bar\jmath <i_1<\ldots <i_n$ in $I$ and $\bar \imath \in I$ such that the family $\mathcal T=\{F_{i_1},\dots, F_{i_n}\}\subseteq \Pf(G)$ 
$\varepsilon$-quasi-tiles $F_i$, for all $i\geq \bar\imath$ in $I$, that is, there is a family $\{C_1,\dots,C_n\}\subseteq\Pf(G)$ such that:
\begin{enumerate}[({QT.}1)]
   \item $C_tF_{i_t}\subseteq F_i$ and $|C_tF_{i_t}|\geq (1-\varepsilon)|C_t||F_{i_t}|$, for all $t=1,\dots,n$;
   \item the family $\{C_tF_{i_t}:t=1,\dots,n\}$ is pairwise disjoint;
   \item $|F_i\setminus \bigcup_{t=1}^nC_tF_{i_t}|\leq \varepsilon |F_i|$.
\end{enumerate}
\end{proposition}

The elements $F_{i_t}$ of $\mathcal T$ are called \emph{shapes} or \emph{tiles}, while the $C_j$ are called \emph{$(\mathcal T,\varepsilon)$-tiling centers} for $F_i$.

We apply the above proposition in the proof of the next one.
	
\begin{proposition}\label{Propo_next-to-last} 
$(\rho_H)_\cov\oplus(\rho_{K/H})_\cov\colon G\la{} \cov(H)\oplus\cov( K/H)$ weakly asymptotically dominates $G\la{\rho_\cov} \cov(K)$.
\end{proposition}
\begin{proof}
Fix $\varepsilon\in(0,1/2)$ and let $\U\in \cov(K)$. Up to taking a refinement, we can suppose that $\U$ is even (see Lemma~\ref{tech_AT}(1)). 
Let $\s=\{F_i\}_{i\in I}$ be a right F\o lner net for $G$. By Proposition~\ref{tech_proposition_conditional}(3), for every $i\in I$,
\begin{equation}\label{firstleq}
\log N(\U_{F_i})\leq \log N_{K/H}(\W_{F_i})+\log N(\U_{F_i}|\W_{F_i}).
\end{equation}

Next we consider $\log N(\U_{F_i}|\W_{F_i})$.  Observe that:
\begin{itemize}
  \item by Lemma~\ref{lemma_very_close_cotrajectories}, there exists $\bar\jmath_0\in I$ such that \eqref{EqJune22} holds with $\bar\imath = \bar\jmath_0$;
  \item by Proposition~\ref{tilings_theo}, there are $\bar\jmath_0< i_1<\dots<i_n$ in $I$ and $\bar\jmath_1\in I$ such that for all $i\geq \bar\jmath_1$ in $I$,  $\mathcal T=\{F_{i_1},\dots,F_{i_n}\}$ $\varepsilon$-quasi-tiles $F_i$;
  \item by Proposition~\ref{tech_proposition_conditional}(6), we can choose $\W\in \cov(Q)$ such that $N(\U_{F_{i_t}}|\W)=N(\U_{F_{i_t}}|\pi)$ for all $t=1,\dots,n$;
  \item by  Lemma~\ref{NpiN}, for any $i\in I$, $N(\U_{F_{i}}|\pi)=N_H(\U_{F_{i}})=N_H(\V_{\rho_H,F_{i}})=N_H(\V_{F_i})$, with $\V=\{U\cap H:U\in \U\}\in \cov(H)$, as $\U_{F_{i_t}}\cap H=(\U\cap H)_{\rho_H,F_i}=(\U\cap H)_{F_i}$ since $\rho_g^{-1}(H)=H$ for every $g\in G$.
\end{itemize}
Choose $\bar\jmath \in I$ such that $\bar\jmath\geq \bar\jmath_k$, with $k=0,1$. Fix $i\geq\bar\jmath$ in $I$, and choose a family $\{C_1,\dots,C_n\}\subseteq \Pf(G)$ of $(\mathcal T,\varepsilon)$-tiling centers for $F_i$. This gives some basic partitions: 
\begin{equation}\label{basic_partitions_from_tiles_eq}
F_i=F_i^*\sqcup F_i^\#,\mbox{ with }\ F_i^*= \bigcup_{t=1}^nC_tF_{i_t} = \bigsqcup_{t=1}^nC_tF_{i_t}, \ F_i^\# = F_i \setminus F_i^* \ \mbox{ and }\ |F_i^\#|\leq \varepsilon |F_i|.
\end{equation}
Apply Corollary~\ref{tech_corollary_conditional}(2) and (3) to get the following inequalities, where the last one makes use also of the last inequality in \eqref{basic_partitions_from_tiles_eq}:
\begin{equation}\label{first_deco_eq*}
\log N(\U_{F_i}|\W_{F_i})\leq \log N(\U_{F_i^*}|\W_{F_i}) + \log N(\U_{F_i^\#}|\W_{F_i}) \leq \log N(\U_{F_i^*}|\W_{F_i}) + \varepsilon |F_i| \log N(\U).
\end{equation}
To better understand the summand $\log N(\U_{{F_i^*}}|\W_{F_i})\leq \sum_{t=1}^n \log N(\U_{C_tF_{i_t}}|\W_{F_i})$
(see Corollary~\ref{tech_corollary_conditional}(3) and \eqref{basic_partitions_from_tiles_eq}), we can study each $N(\U_{C_tF_{i_t}}|\W_{F_i})$ separately, for $t=1,\dots,n$. In particular, by Corollary~\ref{tech_corollary_conditional}(1) and (3) and Proposition~\ref{tech_proposition_conditional}(2),
\begin{align*}
\log N(\U_{C_tF_{i_t}}|\W_{F_i}) &
\leq \sum_{c\in C_t}\log N(\U_{cF_{i_t}}|\W_{F_i})=\sum_{c\in C_t}\log N(\U_{F_{i_t}}|\W_{c^{-1}F_i})\leq |C_t|\log N(\U_{F_{i_t}}|\W)=|C_t| \log N(\U_{F_{i_t}}|\pi),
\end{align*}
where the last equality holds by the choice of $\W$. Hence, using \eqref{EqJune22} with $\bar\imath = \bar\jmath$,
\begin{align*}
\log N(\U_{C_tF_{i_t}}|\W_{F_i})
& \leq \frac{|F_{i_t}||C_t|}{|F_i|} \log N(\U_{F_i}|\pi) + {\varepsilon |F_{i_t}| |C_t|}.
\end{align*}
Taking the sum over $t$ of these last inequalities, we finally get the estimate we were missing in \eqref{first_deco_eq*}:
\[\begin{split}\log N(\U_{{F_i^*}}|\W_{F_i})&\leq \sum_{t=1}^n \left(\frac{|C_t||F_{i_t}|}{|F_i|} \log N(\U_{F_i}|\pi)+ {\varepsilon |C_t||F_{i_t}|}\right)
\leq \\ &\leq \left( \sum_{t=1}^n\frac{|C_t||F_{i_t}|}{|F_i|}\right) \log N(\U_{F_i}|\pi)+ {\varepsilon\sum_{t=1}^n|C_t||F_{i_t}|}
\leq \frac{1}{1-\varepsilon} \log N(\U_{F_i}|\pi) + {\frac{\varepsilon|F_i|}{1-\varepsilon}}.\end{split}\]
Combining this inequality with \eqref{first_deco_eq*}, and recalling that $N(\U_{F_i}|\pi)=N_H(\V_{F_i})$,
\[\log N(\U_{F_i}|\W_{F_i})\leq\frac{1}{1-\varepsilon} \log N_H(\V_{F_i}) + \frac{\varepsilon|F_i|}{1-\varepsilon} + \varepsilon|F_i| \log N(\U).\]
So, using also \eqref{firstleq}, we get that
\begin{align*}\log N(\U_{F_i})&\leq \log N_{K/H}(\W_{F_i})+ \frac{1}{1-\varepsilon} \log N_H(\V_{F_i}) + \frac{\varepsilon|F_i|}{1-\varepsilon} + \varepsilon|F_i| \log N(\U)\\
&\leq \frac{1}{1-\varepsilon}(\log N_{K/H}(\W_{F_i})+  \log N_H(\V_{F_i})) + \frac{\varepsilon|F_i|}{1-\varepsilon} + \varepsilon|F_i| \log N(\U).
\end{align*}
Dividing by $|F_i|$, and considering the last inequality with $\varepsilon=\frac{1}{n}$, we get
\[\frac{\log N(\U_{F_i})}{|F_i|}\leq \frac{n}{n-1}\left(\frac{\log N_H(\V_{F_i})}{|F_i|}+ \frac{\log N_{K/H}(\W_{F_i})}{|F_i|}\right)+\frac{1}{n-1}+\frac{1}{n}\log N(\U).\]
This means that, with $f_n(r)=\frac{n}{n-1} r+\frac{1}{n-1}+\frac{1}{n}\log N(\U)$, for $r\in \R$, and the constant sequence $(\V, \W)\in \cov (H)\oplus\cov(K/H)$, the inequality
$$\frac{v_{\cov}(T_{F_i}( \rho_\cov, \U))}{|F_i|}\leq f_n\left(\frac{v_{\cov}(T_{F_i}((\rho_H)_\cov\oplus (\rho_{K/H})_\cov, (\V,\W)))}{|F_i|}\right),$$
 witnesses the  weak asymptotic domination, since $f_n(r)-r$ uniformly converges to 0 on every bounded interval $[0,C]$.
\end{proof}

\begin{proposition}\label{Propo_last}
$G\la{\rho_\cov} \cov(K)$ dominates $(\rho_H)_\cov\oplus(\rho_{K/H})_\cov\colon G\la{} \cov(H)\oplus\cov( K/H)$.
\end{proposition}
\begin{proof}
Given $\U\in \cov(K)$ and $\V\in\cov(H)$, by Lemma~\ref{tech_AT}(1), there is an open neighborhood $U$ of $1\in K$ such that $\V\preceq \U_{H}[U\cap H]$ and $\pi(\U)\preceq \U_{K/H}[U]$. Choose a symmetric and invariant neighborhood $V$ of $1\in K$ such that, letting $W=VVVV$, we have $WW \subseteq U$. Then $W$ is invariant in $K$, so $W\cap H$ is invariant in $H$.  Given $F\in \Pf(G)$, select  $D\subseteq H$ and $E\subseteq K$ to be maximal with respect to the following properties:
\begin{enumerate}
   \item[\rm ($*_H$)] $d_1C_F(\rho_H,W\cap H)\cap d_2C_F(\rho_H,W\cap H)=\emptyset$, for all $d_1\neq d_2$ in $D$;
   \item[\rm ($ *_{K/H}$)] $\pi(e_1C_F(\rho,W))\cap \pi(e_2C_F( \rho,W))=\emptyset$, for all $e_1\neq e_2$ in $E$.
  \end{enumerate}
The existence of such a subset $D$ of the compact group $H$ follows from the fact that its open subset $A=C_F(\rho_H,W\cap H)\ne \emptyset$ has a positive Haar measure, say $a>0$, while $H$ has Haar measure 1. Therefore, a set $D$ with ($*_H$) must have size at most $1/a$, so there exists a maximal one with  ($*_H$). As far as the  existence of $E$ is concerned, consider the open subset  $B=\pi(C_F(\rho_{K/H},W)\ne \emptyset$ of $K/H$. If $E\subseteq K$ satisfies ($ *_{K/H}$), then $e_1 \pi^{-1}(B) \cap e_2 \pi^{-1}(B)=\emptyset$, for all $e_1\neq e_2$ in $E$. So by  the previous argument, applied this time to the compact group $K$ and its open subset $\pi^{-1}(B)$, we conclude that there is a finite upper bound for the cardinalities of $E$ satisfying ($ *_{K/H}$). In particular, there is a maximal one. 

\smallskip
By Lemma~\ref{tech_AT}(2) and (4),  we deduce that:
\[N(T_F((\rho_H)_\cov,\V))\leq N(T_F((\rho_H)_\cov,\U_{H}[U\cap H]))\leq N(\U_{H}[C_F(\rho_H,U\cap H)])\leq |D|.\]
Similarly, by Lemma~\ref{tech_AT}(2) and (4), we obtain that:
\[N(T_F((\rho_{K/H})_\cov,\pi(\U)))\leq N(T_F((\rho_{K/H})_\cov,\U_{ K/H}[U])\leq N(\U_{K/H}[C_F(\rho,U)])\leq |E|.\]
In particular,
\[N(T_F((\rho_H)_\cov\oplus(\rho_{K/H})_\cov,(\V,\pi(\U)))=N(T_F((\rho_H)_\cov,\V))\cdot N(T_F((\rho_{K/H})_\cov,\pi(\U)))\leq |D|\cdot|E|.\]
Moreover, the following property holds:
\begin{enumerate}
\item[\rm ($*$)] 
$(e_1C_F(\rho,VV)d_1)\cap (e_2C_F(\rho,VV)d_2)=\emptyset$, for all $(e_1,d_1)\neq (e_2,d_2)$ in $E\times D$.
\end{enumerate}
In fact, assume that {$e_1c_1d_1=e_2c_2d_2$} for some $c_1,c_2\in C_F(\rho,VV)$, $e_1,e_2\in E$, and $d_1,d_2\in D$. Then,
$$\pi(e_1c_1)=\pi(e_1c_1d_1)=\pi(e_2c_2d_2)=\pi(e_2c_2)\in \pi(e_1C_F(\rho,VV))\cap \pi(e_2C_F( \rho,VV))\subseteq \pi(e_1C_F(\rho,W))\cap \pi(e_2C_F( \rho,W)),$$
and so  $e_1=e_2$ by ($ *_{K/H}$). Hence, $c_1d_1=c_2d_2$ and, therefore, 
$$c=c_1^{-1}c_2=d_1d_2^{-1}\in H\cap C_F(\rho,W)= C_F(\rho_H,W\cap H),$$ 
showing that $d_1=cd_2\in C_F(\rho_H,W\cap H)d_2 = d_2C_F(\rho_H,W\cap H)$, as $C_F(\rho_H,W\cap H)$ is invariant in $H$ 
(being $W\cap H$ invariant in the $S$-invariant subgroup $H$). This gives $d_1=d_2$ by ($*_H$). Hence, $\bigcup_{(e,d)\in E\times D}eC_F(\rho,VV)d\subseteq K$ is a disjoint union of $|E|\cdot|D|$-many measurable subsets, each of which has the same measure of $C_F(\rho,VV)$ (use that, being $K$ compact, it is unimodular and, therefore, $\mu$ is both left and right invariant), and so 
\[|E|\cdot|D|\cdot \mu(C_F(\rho,VV))\leq \mu(K)\leq N(\U_K[C_F(\rho,VV)])\cdot \mu(C_F(\rho,VV)).\]
Hence, $|E|\cdot|D| \leq N(\U_K[C_F(\rho,VV)])$. By Lemma~\ref{tech_AT}(3),  $N(\U_K[C_F(\rho,VV)])\leq N(T_F(\rho_\cov,\U_K[V]))$ and, combining all these estimates, we get:
\[N(T_F((\rho_H)_\cov\oplus(\rho_{K/H})_\cov,(\V,\pi(\U)))\leq |D|\cdot |E|\leq N(T_F(\rho_\cov,\U_K[V]).\]
Taking logarithms gives 
\(v_{\cov}(T_F((\rho_H)_\cov \oplus (\rho_{K/H})_\cov, (\V, \pi(\U)))) \leq v_{\cov}(T_F( \rho_\cov, \U_K[V]))\).
\end{proof}

As a consequence of the above propositions and of the results of \S\ref{subs_dual_ore}, we obtain the proof of (\text{\sc at$_\top$}):

\begin{proof}[{\bf Proof of the Topological Addition Theorem}]  First assume that $K\ra{\rho}G$ is a right action of an amenable group $G$ on a compact group $K$ and that $H$ is an $S$-invariant closed subgroup of $K$. Fix a right F\o lner net $\s$ for $G$. Since $\rho_{\cov}$ and $(\rho_H)_{\cov}\oplus(\rho_{K/H})_\cov$ are weakly asymptotically equivalent by Propositions~\ref{Propo_next-to-last} and~\ref{Propo_last}, 
\begin{equation}\label{ATfinaleq}
h(\rho_{\cov},\s)=h((\rho_H)_\cov\oplus(\rho_{K/H})_\cov,\s) = h((\rho_H)_{\cov},\s) + h({(\rho_{K/H})_\cov},\s),
\end{equation} 
by Propositions~\ref{wad} and~\ref{wAT}. Hence, \eqref{ATfinaleq} gives $h_{\top}(\rho)=h_{\top}(\rho_{{H}})+h_{\top}(\rho_{K/H})$ as already now that,  by Example~\ref{Gactionsextop}, $h_\top(\rho)=h(\rho_{\cov},\s)$, $h_\top(\rho_H)=h((\rho_H)_{\cov},\s)$ and $h_\top(\rho_{K/H})=h({(\rho_{K/H})_\cov},\s)$. 

\smallskip
Now consider the general case of a right $S$-action $K\ra{\rho}S$ of a cancellative and right amenable monoid $S$ on a compact group $K$ and let $H$ be a closed $S$-invariant (not necessarily normal) subgroup of $K$.  As usual, let $G=S^{-1}S$ be the group of left fractions of $S$.  
By Corollary~\ref{strong_exactness_coro}(1), we can identify $H^*$ with a closed $S$-invariant subgroup of $K^*$ (so that it makes sense to consider the space of left $H^*$-cosets $K^*/H^*$), and by Corollary~\ref{strong_exactness_coro}(2), we can identify $K^*/H^*$ with $(K/H)^*$. This allows us to apply the previous case  and write   
\begin{equation}\label{AT_top_eq_coro}
h_{\top}(\rho^*)=h_{\top}((\rho^*)_{H^*})+h_{\top}((\rho^*)_{K^*/H^*}). 
\end{equation}
In view of the above identifications,
$$h_{\top}((\rho^*)_{H^*})=h_\top((\rho_H)^*)\quad\text{and}\quad h_\top((\rho^*)_{K^*/H^*})=h_\top((\rho_{K/H})^*).$$
Furthermore, by Theorem~\ref{Th2}, we deduce that $h_{\top}(\rho^*)=h_{\top}(\rho)$, $h_{\top}((\rho_{ H})^*)=h_{\top}(\rho_{H})$ and $h_{\top}( (\rho_{K/H})^*)=h_{\top}( \rho_{K/H})$. Making these substitutions in \eqref{AT_top_eq_coro}, we get
\[h_{\top}(\rho)=h_{\top}(\rho_{ H})+h_{\top}( \rho_{K/H}).\qedhere\]
\end{proof}

\begin{question}
Is it possible to prove the general version of {AT}$_\top$ without making recurse to the Ore colocalization?
\end{question}

\section{Proof of the Bridge Theorem and the Algebraic Addition Theorem}\label{proofs}

The aim of this last section, is to prove the Bridge Theorem, as announced in the introduction. We start with a lemma showing how,  for a cancellative and right amenable monoid $S$, the kernel $\Ker(\lambda)$ of a left $S$-action $S\la{\lambda}X$ on a discrete Abelian group (from \S\ref{reduction_to_injective}) and the surjective core $E(\rho)$ of a right $S$-action $K\ra{\rho}S$ on a compact Abelian group (from \S\ref{subs_red_top_to_surj}) are both the dual construction of each other.  Recall that $\bar X=X/\Ker(\lambda)$ and $\bar K=E(\rho)$, and that $S\la{\bar\lambda}\bar X$ and $\bar K\ra{\bar\rho}S$ are the induced $S$-actions by $\lambda$ and $\rho$, respectively.

\begin{lemma}\label{first_tech_lemma_for_duality}
The following statements hold true:
\begin{enumerate}[(1)]
   \item Given a left $S$-action $S\la{\lambda}X$ on a discrete Abelian group $X$, let $K=\dual X$ and $K\longra{\rho=\dual{\lambda}}S$. Then, $\Ker(\lambda)^{\perp}=E(\rho)$ and the dual of $S\la{\bar\lambda}\bar X$ is conjugated to $\bar K\ra{\bar \rho}S$.
   \item Given a right $S$-action $K\ra{\rho}S$ on a compact Abelian group $K$, let $X=\dual K$ and $S\longla{\lambda=\dual{\rho}}X$. Then, $E(\rho)^{\perp}=\Ker(\lambda)$ and the dual of $\bar K\ra{\bar\rho}S$ is conjugated to $S\la{\bar\lambda}\bar X$.
\end{enumerate}
\end{lemma}
\begin{proof}
Let us just prove (1), as part (2) follows similarly. Consider $(S,\leq)$ preordered by its right divisibility relation, then $\Ker(\lambda)$ and $E(\rho)$ are, respectively, the supremum in $\L(X)$ of the directed family $\{\Ker(\lambda_s):s\in S\}$ and the infimum in $\L(K)$ of the downward directed family $\{\rho_s(K):s\in S\}$. By the properties of the lattice anti-isomorphism $(-)^\perp\colon \L(X)\to \L(K)$, one deduces immediately that $\Ker(\lambda)^\perp=E(\rho)$ and, therefore, 
\begin{equation}\label{iso_ker_epicore}
\dual{\bar X}=\dual{(X/\Ker(\lambda))}\cong \Ker(\lambda)^\perp=E(\rho)=\bar K.
\end{equation}
Finally, for each $s\in S$, we have that $\lambda_s(\Ker(\lambda))=\Ker(\lambda)$ and $\bar\lambda_s$ is the injective endomorphism induced by $\lambda_s$ on the quotient $\bar X=X/\Ker(\lambda)$. As the Pontryagin duality sends (endomorphisms of) quotient groups to (endomorphisms of) closed subgroups of the dual, $\dual{{\bar\lambda}_{s}}$ can be identified with the restriction $\bar\rho_s$ of $\rho_s=\dual{\lambda_s}$ to $\bar K$. In other words, the isomorphism in \eqref{iso_ker_epicore} is $S$-equivariant.
\end{proof}

In the following lemma we show that ``the Ore colocalization of a right action on a compact Abelian group is the dual of the Ore localization of its dual action'' and, vice versa, that ``the Ore localization of a left action on a discrete Abelian group is the dual of the Ore colocalization of its dual action'' (see \S\ref{Ore_invariance_subs} and \S\ref{subs_dual_ore} for the definitions of the Ore localization $G\la{\lambda^*}X^*$ and the Ore colocalization $K^*\ra{\rho^*}G$, respectively).

\begin{lemma}\label{dual_ore_is_dual_of_ore}
The following statements hold true:
\begin{enumerate}[(1)]
\item Given a left $S$-action $S\la{\lambda}X$ on a discrete Abelian group $X$, let $K=\dual X$ and $K\longra{\rho=\dual{\lambda}}S$. 
 Then the dual of $G\la{\lambda^*}X^*$ is conjugated to $K^*\ra{\rho^*}G$.
\item Given a right $S$-action $K\ra{\rho}S$ on a compact Abelian group $K$, let $X=\dual K$ and $S\longla{\lambda=\dual{\rho}}X$. 
 Then the dual of $K^*\ra{\rho^*}G$ is conjugated to $G\la{\lambda^*}X^*$.
\end{enumerate}
\end{lemma}
\begin{proof}
We just prove part (1), as (2) follows similarly. In Definition~\ref{def_x_star} we have introduced $X^*$ as a direct limit of a direct system $\mathfrak X$ of copies of $\bar X$, indexed by $(G,\leq_S)$, and with transition maps given by suitable maps of the form $\bar\lambda_s$ for some $s\in S$. Dually, in Definition~\ref{def_k_star}, we have introduced $K^*$ as the inverse limit of an inverse system $\mathfrak K$ of copies of $\bar K$, indexed by $(G,\leq_S)$, and with transition maps given by suitable maps of the form $\bar\rho_s$ for some $s\in S$. Now, since Lemma~\ref{first_tech_lemma_for_duality}(1) allows us to identify $\dual{\bar X}\ra{\dual{\bar\lambda}}S$ with $\bar K\ra{\bar\rho}S$, we deduce that the two inverse systems $\dual{\mathfrak X}$ and $\mathfrak K$ are isomorphic and, therefore, so are their limits: $\dual{(X^*)}\cong K^*$. Furthermore, this isomorphism is $S$-equivariant since the commutative diagrams of Lemma~\ref{def_lambda_star_lemma}(2) are sent, via Pontryagin duality, to the commutative diagrams that appear in Lemma~\ref{def_of_k_star_lemma}(2). In fact, $\rho^*$ is, by definition, the unique right $S$-action that makes these diagrams commute; as $\dual{(\lambda^*)}$ makes the same diagrams commute, the two actions have to be conjugated.
\end{proof}

\begin{proof}[\bf{Proof of the Bridge Theorem}]
(1) Let $K\ra{\rho}S$ be a right $S$-action on a compact Abelian group $K$ and $S\longla{\lambda=\dual\rho}X=\dual K$.
With the notation of \S\ref{sec_discrete} and \S\ref{sec_compact}, we have proved in Theorems~\ref{Th1} and~\ref{Th2}, respectively, that
\[h_\top(\rho)=h_\top(\rho^*) \quad\text{and}\quad h_{\alg}(\lambda)=h_\alg(\lambda^*).\]
Furthermore, $\rho^*$ is conjugated to $\dual{(\lambda^*)}$ by Lemma \ref{dual_ore_is_dual_of_ore}(1) and so $h_\top(\rho^*)= h_\top(\dual{(\lambda^*)})$ by Proposition~\ref{mon}(1), while $h_\top(\dual{(\lambda^*)})=h_\alg(\lambda^*)$ by Theorem~\ref{BT_for_invertible_actions}. From the above equalities we conclude that  \[
h_{\top}(\rho)=h_\top(\rho^*)=h_\top(\dual{(\lambda^*)})=h_\alg(\lambda^*)=h_{\alg}(\lambda).\] 

\smallskip
(2) Let $S\la{\lambda}X$ be a left $S$-action on a discrete Abelian group $X$ and $K=\dual X\longra{\rho=\dual\lambda} S$.
Since the left $S$-action $\dual\rho={\lambda}^{\wedge\wedge}$ is conjugated to $\lambda$ (as noted in the introduction), we deduce that $h_{\alg}(\lambda) = h_{\alg}({\lambda}^{\wedge\wedge})$ by Proposition~\ref{mon}(1). By part (1), $h_{\top}(\rho)=h_{\alg}(\dual\rho)=h_{\alg}(\lambda)$, as required.
\end{proof}

The above proof shows in particular that items (1) and (2) in the Bridge Theorem are equivalent. This makes it natural to state the Bridge Theorem asymmetrically, that is, in our case from the point of view of $K\ra{\rho} S$.

\begin{proof}[{\bf Proof of the Algebraic Addition Theorem}] 
Given a left $S$-action $S\la{\lambda} X$ on a discrete Abelian group $X$ and letting $Y$ be an \mbox{$S$-invariant} subgroup of $X$, denote by $\lambda_{Y}$ and $ \lambda_{X/Y}$ the obvious left $S$-actions induced by $\lambda$ on $Y$ and $X/Y$, respectively. Furthermore, consider the compact Abelian group $K=\dual X$, its closed subgroup $H=Y^\perp$ and its quotient group  $K/H$, and  the right $S$-action $\rho=\dual{\lambda}$. Then $H$ is $S$-invariant and the action $\rho_{H}$ induced by $\rho$ on $H$ by restriction is conjugated to $\dual{(\lambda_{X/Y})}$, while the right $S$-action $\rho_{K/H}$ induced by $\rho$ on $K/H$ is conjugated to $\dual{(\lambda_{Y})}$. We can now conclude via the following series of equalities:
{\belowdisplayskip=-14pt
\begin{align*}
h_\alg(\lambda)&=h_\top(\rho)&\text{by the Bridge Theorem;}\\ 
&=h_\top(\rho_H)+h_\top(\rho_{K/H}) & \text{by the Topological Addition Theorem;} \nonumber \\
&=h_\top(\dual{(\lambda_{X/Y})}\,)+h_\top({\dual{(\lambda_{Y})}}) &\text{by Proposition~\ref{properties_h_top}(1);} \nonumber \\
&=h_\alg( \lambda_{X/Y})+h_\alg(\lambda_{ Y})&\text{by the Bridge Theorem}.  
\end{align*}}
\end{proof}

\begin{question}
Is it possible to give a direct proof of $(\text{\sc at}_\alg)$ instead of deducing it from $(\text{\sc at}_\top)$ through the Bridge Theorem?
\end{question}

\appendix

\section{Bridge Theorem for the entropies of amenable group actions}\label{AppC}

In this appendix, we consider a discrete Abelian group $X$, its dual  compact Abelian group $K=\dual X$ (identifying $X$ with $\dual K$) and the following actions of an infinite amenable group $G$: $G\la{\lambda}X$ and $K\longra{\rho=\dual\lambda}G$.
Our goal is to prove the Bridge Theorem, that is, that $h_\alg(\lambda) = h_\top(\rho)$. 

\begin{remark}
 When $X$ is torsion (i.e., $\L^\fin(X)$ is cofinal in $\Pf^0(X)$ -- see Example~\ref{nmexalg}(3)) and so equivalently $K$ is profinite (i.e., $\L^{\op}(K)$ is cofinal in $\cov(K)$ -- see Example~\ref{nmextop2}(6)),
we have seen in Examples~\ref{Gactionsexalg}(4) and~\ref{Gactionsextop2}(4) that  the algebraic and the topological entropies may be written, respectively, only in terms of $\L^\fin(X)$ and $\L^{\op}(K)$, namely, for a right F\o lner net $\s$ for $G$, $h_\alg(\lambda)=h((\lambda_{\mathcal P})_{\restriction \L^\fin(X)},\s)$ and $h_\top(\rho)=h((\rho_\cov)_{\restriction \L^{\op}(K)},\s)$.
By using the isomorphism of normed monoids $\L^\fin(X)\cong \L^{\op}(K)$ given by $F\mapsto F^\perp$ from Example~\ref{nmextop2}(4), one can see that $(\lambda_{\mathcal P})_{\restriction \L^\fin(X)}$ and  $(\rho_\cov)_{\restriction \L^{\op}(K)}$ are conjugated and so conclude that $h_\alg(\lambda)=h_\top(\rho)$ by Proposition~\ref{wad}. 
This is done in~\cite{DGB-BT,GB,Weiss}.

Unfortunately, in general  this is not enough (think for example of the case when $X$ is torsion-free, where $\L^\fin(X)=\{0\}$ and $\L^{\op}(K)=\{K\}$) but the idea at the base of the above  strategy is still valid: to prove the Bridge Theorem one may try to suitably ``approximate'' the normed monoids ${\frak F}(X)$ and $\cov(K)$ by some other normed monoids: following an idea of \mbox{J.\ Peters}, we consider the monoids $\M_\alg(X)$ and $\M_\top(K)$ in \S\ref{Petersmon}, and verify that they are isomorphic.
\end{remark}

This appendix is organized as follows: 
\begin{enumerate}[\rm --]
   \item in \S\ref{harmonic_background} we recall the needed definitions and results from harmonic analysis;
   \item in \S\ref{Petersmon} we introduce the algebraic and the topological Peters monoids $\M_\alg(X)$ and $\M_\top(K)$, and the canonical left $G$-actions $\lambda_\alg$ and $\rho_\top$ on them induced by $\lambda$ and $\rho$ respectively, and we verify the isomorphism of normed monoids $\M_\alg(X)\cong \M_\top(K)$ that permits to see that $\lambda_\alg$ and $\rho_\top$ are conjugated; 
   \item in \S\ref{algebraic_approximation_part}  and \S\ref{topological_approximation_part} we prove that, respectively,
the left $G$-action $\lambda_\alg$ on $\M_\alg(X)$ and the left $G$-action $\lambda_{\mathfrak F}$ on ${\frak F}(X)$
are asymptotically equivalent, and the same holds true for   the left $G$-actions $\rho_\top$ on $\M_\top(K)$ and $\rho_\cov$ on $\cov(K)$;
   \item finally, in \S\ref{proof_BT}, we put all the pieces together to prove the Bridge Theorem.
\end{enumerate}

\subsection{Needed facts from harmonic analysis}\label{harmonic_background}

For the first part of this subsection let $\Gamma$ be an infinite\footnote{The results in this subsection remain true for finite locally compact Abelian groups (i.e., those that are both compact and discrete). Assuming $\Gamma$ to be infinite is just a technical simplification. Furthermore, the entropy (algebraic or topological) of $G$-actions on finite groups is always $0$, so this  restriction does not exclude any non-trivial case.} locally compact Abelian group and we fix a Haar measure $\mu$ on $\Gamma$. We denote by ${\frak U}(\Gamma)$ the family of symmetric compact neighborhoods of $0$ in $\Gamma$.
We are mainly interested in the case when $\Gamma=X$ is discrete, so 
$\mu$ is the counting measure $\mu(Y)=\sup\{|F|:F\in\Pf(Y)\}$ for $Y\subseteq X$ and $\mathfrak U(X)=\mathfrak F(X)$, and when $\Gamma=K$ is compact, in which case there is a unique Haar measure such that $\mu(K)=1$.

\smallskip
In this appendix we use the unit circle $\mathbb S$ in $\mathbb C$ in place of $\T$ when we consider the dual group $\dual\Gamma$, so that each $\gamma\in\dual\Gamma$ is a continuous complex-valued function $\gamma\colon\Gamma\to\mathbb S\subseteq\mathbb C$. This motivates us to use the multiplicative notation for the group $\dual\Gamma$,  with the sole exception of its neutral element that we write $ 0_{\dual\Gamma}$ (it should be kept in mind that this is the constant function $1$ on $\Gamma$).  

\medskip
Given a subset $E\subseteq \Gamma$, we denote by $\cl(E)$ its closure in $\Gamma$. For a Borel function $\phi\colon \Gamma\to \mathbb C$, we define its  {\em support} by: $\supp(\phi)=\cl{\{x\in \Gamma:\phi(x)\neq 0\}}\subseteq \Gamma$. Furthermore, we say that $\phi$ is:
\begin{enumerate}[\rm --]
   \item {\em absolutely integrable} if $\Vert\phi\Vert_1=\int_{x\in\Gamma}|\phi(x)| \d\mu(x)<\infty$. We denote by $L^1(\Gamma)$ the space of absolutely integrable functions, where we have identified those that coincide almost everywhere. By~\cite[\S{E7}]{Rudin}, $\Vert-\Vert_1$ induces a norm on $L^1(\Gamma)$ (so the triangular inequality holds);
   \item {\em positive} if it is real-valued (i.e., $\phi(\Gamma)\subseteq \R$) and $\phi(x)\geq 0$ (in $\R$), for all $x\in \Gamma$. Given a set $V$ of complex-valued functions on $\Gamma$, we denote by $V^+$ the subset of all positive functions in $V$;
   \item {\em positive-definite} if, for all $n\in\N_+$, $x_1,\dots,x_n\in \Gamma$ and $c_1,\dots,c_n\in\mathbb C$, the following holds:
\(\sum_{i,j=1}^nc_i\overline{c_j}\phi(x_i-x_j)\in\R_{\geq 0}\).
We let $\P(\Gamma)$ be the set of continuous and positive-definite functions on $\Gamma$. 
As every positive-definite function coincides with a continuous one almost everywhere (see~\cite[(3.21) and (3.35)]{Fol}), it is not restrictive to assume continuity.
\end{enumerate}
Obviously, $\dual\Gamma \subseteq \P(\Gamma)$. 

\smallskip
For the proof of the following properties of positive-definite functions we refer to~\cite[\S1.4.1]{Rudin}.

\begin{lemma}\label{positive_def}
In the above notation, let $\phi\in \P(\Gamma)$. Then $\phi^\star$, defined  by $\phi^\star(x)=\phi(-x)$ for all $x\in \Gamma$,  belongs to $\P(\Gamma)$ and 
\begin{enumerate}[(1)] 
    \item $\phi^\star(x)=\overline{\phi(x)}$  and $|\phi(x)|\leq \phi(0)$, for all $x\in \Gamma$. In particular, $\phi(0)=\overline {\phi(0)}\in \R_{\geq0}$;
    \item if $\phi\in\P(\Gamma)^+$, then $\phi=\phi^\star$. In particular, $\phi(0)$ is a maximum for $\phi$;
    \item if $\phi\in \P(\Gamma)^+$ has compact support, then $\Vert\phi\Vert_1\leq\phi(0)\cdot\mu(\supp(\phi))$.
\end{enumerate}
\end{lemma}

If $\phi,\psi\in L^1(\Gamma)$, then $\int_{y\in \Gamma}|\phi(y)\psi(x-y)|\d\mu(y)<\infty$ for almost all $x\in \Gamma$ (see~\cite[\S 1.1.6(e)]{Rudin}), and so the {\em convolution} $\phi*\psi\colon \Gamma\to\mathbb C$ is defined almost everywhere by the following formula:
\[(\phi*\psi)(x)=\int_{y\in \Gamma}\phi(y)\psi(x-y)\d\mu(y).\]
 By~\cite[\S1.1.6]{Rudin}, if $\phi,\psi\in L^1(\Gamma)$, then $\Vert\phi*\psi\Vert_1\leq \Vert\phi\Vert_1\Vert\psi\Vert_1<\infty$, so that $\phi*\psi\in L^1(\Gamma)$.

Defining $\phi_x\colon \Gamma\to \mathbb C$ as $\phi_x(y)=\phi(xy)$ for all $y\in \Gamma$, we get $(\phi*\psi)(x)=\int_{y\in \Gamma}\phi(y)\psi_{-y}(x)\d\mu(y)$. For the proof of the following properties of convolutions we refer to~\cite[\S1.1.6, \S1.1.7]{Rudin}:

\begin{lemma}\label{norma_1}
In the above notation, let $\phi,$ $\psi,$ $\xi\in L^1(\Gamma)$. Then: 
\begin{enumerate}[ (1)]
   \item $(\phi*\psi)*\xi=\phi*(\psi*\xi)$, \ \ $\phi*\psi=\psi*\phi$\ and \ $\xi*(\phi+\psi)=\xi*\phi+\xi*\psi$;
   \item $\phi_x*\psi=\phi*\psi_x=(\phi*\psi)_x$, for all $x\in \Gamma$; in particular, 
$(\phi_1*\ldots * \phi_k)_x = (\phi_1)_{x_1}*\ldots * (\phi_k)_{x_k}$ whenever $x = x_1 + \ldots + x_k$, where $\phi_1, \ldots, \phi_k\in  L^1(\Gamma)$ and $x_1, \ldots, x_k\in \Gamma$;
   \item $\supp(\phi*\psi)\subseteq \supp(\phi)+\supp(\psi)$, so $\phi*\psi$ has compact support, provided $\phi$ and $\psi$ have compact support.
\end{enumerate}
\end{lemma}

For the sake of completeness we include a proof of the following properties needed in the sequel.

\begin{lemma}\label{norma-}
Let $\phi,$ $\psi\in L^1(\Gamma)^+$. Then:
\begin{enumerate}[ (1)]
   \item $\Vert\phi-\psi\Vert_1\geq \Vert\phi\Vert_1-\Vert\psi\Vert_1$;
   \item $\Vert\phi*\psi\Vert_1=\Vert\phi\Vert_1\Vert\psi\Vert_1$;
   \item $\phi*\phi^\star$ is positive-definite, provided $\phi\in L^2(\Gamma)$ (that is, $\int_{x\in\Gamma}|\phi(x)|^2\d\mu(x)<\infty$);
   \item $(\phi*\psi)(x)\leq \phi(0)\cdot \Vert\psi\Vert_1$, provided $\phi\in\mathfrak P(\Gamma)$; 
   \item if  $\phi_1, \ldots, \phi_k\in  L^1(\Gamma)^+$ and $\psi_1, \ldots, \psi_k\in  L^1(\Gamma)^+$, with $\phi_j \leq \psi_j$ for {$j=1,\ldots, k$}, then $\phi_1*\ldots * \phi_k  \leq \psi_1*\ldots * \psi_k $.
\end{enumerate}
\end{lemma}
\begin{proof}
(1) For all $x\in\Gamma$, we have $|\phi(x)-\psi(x)|\geq |\phi(x)|-|\psi(x)|$. Furthermore, by linearity of integration, 
$$\Vert\phi-\psi\Vert_1=\int_{x\in\Gamma}|\phi(x)-\psi(x)|\d\mu(x)\geq\int_{x\in\Gamma}|\phi(x)|\d\mu(x)-\int_{x\in\Gamma}|\psi(x)|\d\mu(x)=\Vert\phi\Vert_1-\Vert\psi\Vert_1.$$

(2) By the proof of~\cite[Theorem 1.1.6(e)]{Rudin} we have that $\int_{x\in \Gamma} \left(\int_{y\in \Gamma}|\phi(x-y)\psi(y)|\d\mu(y)\right)\d\mu(x) = \Vert\phi\Vert_1\Vert\psi\Vert_1$ (this is a consequence of the Fubini Theorem). It remains to take into account that, when $\phi$ and $\psi$ are positive, then 
\begin{align*}
 \int_{x\in \Gamma} \left(\int_{y\in \Gamma}|\phi(x-y)\psi(y)|\d\mu(y)\right) \d\mu(x) &=  \int_{x\in \Gamma} \left( \int_{y\in \Gamma} \phi(x-y)\psi(y) \d\mu(y)\right) \d\mu(x)= \\
& = \int_{x\in \Gamma} \left| \int_{y\in \Gamma} \phi(x-y)\psi(y) \d\mu(y)\right| \d\mu(x) = \Vert\phi*\psi\Vert_1.
\end{align*}

\smallskip
(3) follows by~\cite[\S1.4.2]{Rudin}  and (5) is clear.

\smallskip
(4) By Lemma~\ref{positive_def}(1), we have $\phi(\Gamma)\subseteq [0,\phi(0)]\subseteq \R_{\geq0}$, so by Lemma~\ref{norma_1}(1),
$$(\phi*\psi)(x)=(\psi*\phi)(x)=\int_{y\in \Gamma}\phi(x-y)\psi(y)\d\mu(y)\leq \phi(0)\cdot\int_{y\in \Gamma}\psi(y)\d\mu(y)=\phi(0)\cdot \Vert\psi\Vert_1.\qedhere$$
\end{proof}

As a consequence of the above lemmas, we get:

\begin{corollary}\label{prescribed_support}
For each $U\in {\frak U}(\Gamma)$, there is a non-zero $\phi\in L^1(\Gamma)^+\cap \P(\Gamma)$ with $\supp(\phi)\subseteq U$.
\end{corollary}
\begin{proof}
Choose $V\in{\frak U}(\Gamma)$ such that $V-V\subseteq U$ and an open subset $V$ of $\Gamma$ with $0\in V'\subseteq V$. As $\Gamma$ is Tychonoff, there is a continuous function $f\colon \Gamma\to [0,1]$ such that $f(0)=1$ and $\supp(f)\subseteq \cl(V')\subseteq V$, so that $\supp( f^\star)\subseteq -V$.  Letting $\phi=f* f^\star$, we have $\supp(\phi)\subseteq V-V\subseteq U$ and $\phi\in L^1(\Gamma)$ by Lemmas~\ref{norma_1}(3) and~\ref{norma-}(2), $\phi\in\P(\Gamma)$ by Lemma~\ref{norma-}(3), and it is easy to see that $\phi$ is non-zero and positive.
\end{proof}
 
For a subset $C$ of $\Gamma$ we denote by $\chi_C\colon\Gamma\to \mathbb C$ the {\em characteristic function} of $C$ (i.e., $\chi_C(x)=1$ if $x\in C$ and $\chi_C(x)=0$ otherwise). If $C$ is compact and symmetric, then $\chi_C\in L^1(\Gamma)^+$ and $\chi_C=(\chi_C)^\star$. By Lemma~\ref{norma-}(1) and (4), we get item (1) in the following corollary. Item (2) follows from Lemma~\ref{norma-}(2) and $\Vert\chi_C\Vert_1=\mu(C)$.

\begin{corollary}\label{chichi}
Let $C\in {\frak U}(\Gamma)$, then:
\begin{enumerate}[ (1)] 
   \item $\chi_C*\chi_C\in L^1(\Gamma)^+\cap \P(\Gamma)$;
   \item $\Vert\chi_C*\chi_C\Vert_1=\mu(C)^2$.
\end{enumerate}
\end{corollary}

\subsubsection{The Fourier transform}

The {\em Fourier transform} of $\phi\in L^1(\Gamma)$ is defined as 
\[
\widehat\phi\colon\dual\Gamma\to \mathbb C,\quad\text{with }\quad\widehat\phi(\gamma)= (\phi*\gamma)(0) =\int_{y\in \Gamma}\phi(y)\gamma(-y) \d\mu(y)=\int_{y\in \Gamma}\phi(y)\gamma^{-1}(y) \d\mu(y)= \int_{y\in \Gamma}\phi(y)\overline{\gamma(y)} \d\mu(y).
\]
Then $\widehat\phi$ is a bounded continuous function on $\dual \Gamma$ that vanishes at infinity (see~\cite[(4.13)]{Fol}). In case $\phi$ is positive-definite, we can say more:

\begin{lemma}\label{first} 
If $\phi\in L^1(\Gamma)^+\cap \mathfrak P(\Gamma)$, then $\widehat\phi\in L^1(\dual\Gamma)^+\cap \mathfrak P(\dual\Gamma)$.
\end{lemma}
\begin{proof} Pick $\phi \in L^1(\Gamma)^+\cap \mathfrak P(\Gamma)$. Then $\widehat\phi\in L^1(\dual\Gamma)$ by the Fourier Inversion Theorem (see~\cite[page 22]{Rudin}) and $\widehat\phi$ is positive by~\cite[Corollary 4.23]{Fol}. 

Let now $\mu_\phi$ be the non-negative and bounded (as $\phi\in L^1(\Gamma)^+$) regular measure defined on a generic Borel subset $E$ of $\Gamma$ by $\mu_\phi(E)=\int_{x\in E}\phi(x)\d\mu(x)$. According to~\cite[\S1.3.4]{Rudin}, for every $\gamma \in \dual\Gamma$, 
$$\widehat\phi(\gamma) =\int_{x\in\Gamma}\gamma(-x)\d\mu_\phi(x) = {\int_{x\in\Gamma}\gamma(x)^{-1}\d\mu_\phi(x)}= \int_{x\in\Gamma}\gamma^{-1}(x)\d\mu_\phi(x).$$
Trading $\gamma$ for $\gamma^{-1}$ gives  $(\widehat\phi)^\star (\gamma) =\widehat\phi(\gamma^{-1}) =   \int_{x\in\Gamma}\gamma(x)\d\mu_\phi(x).$
By the Bochner Theorem (see~\cite[page 19]{Rudin}), $({\widehat\phi})^\star\in \mathfrak P(\dual\Gamma)$.  By Lemma \ref{positive_def}, this proves that ${\widehat\phi}\in \mathfrak P(\dual\Gamma)$. 
\end{proof}

As a consequence of the above result, and in particular of the fact that $\widehat\phi$ is positive and positive-definite, we get that $\widehat\phi$ is symmetric (i.e., $\widehat\phi=({\widehat\phi})^\star$) by Lemma~\ref{positive_def}(2).

\smallskip
A proof of the following fact can be found in \cite[Theorem (32.9)]{HR}, but we offer here a self-contained proof for the sake of completeness, since we use it in Lemma~\ref{Peters-top}.

\begin{lemma}\label{posdef}
If $\phi,\psi\in L^1(\Gamma)$ are positive-definite, then $\phi\cdot\psi$ is positive-definite.
\end{lemma}
\begin{proof}
By the Bochner Theorem, there exists non-negative measures $\mu$ and $\nu$ on $\dual\Gamma$ such that, for every $x\in\Gamma$,
\begin{equation}\label{Eq:July2}
\phi(x)=\int_{\gamma\in\dual\Gamma}\gamma(x) \d\mu(\gamma) \ \mbox{ and } \ \psi(x)=\int_{\gamma\in\dual\Gamma}\gamma(x) \d\nu(\gamma).
\end{equation}
%
After multiplication in \eqref{Eq:July2} one obtains
$$
\phi(x)\psi(x)=\int_{\gamma\in\dual\Gamma}\gamma(x) \d\mu(\gamma)\int_{\gamma'\in\dual\Gamma}\gamma'(x) \d\nu(\gamma')
=\int_{\gamma\in\dual\Gamma} \left(\int_{\gamma'\in\dual\Gamma}\gamma(x)\gamma'(x)\d\nu(\gamma')\right) \d\mu(\gamma)
=\int_{\gamma\in\dual\Gamma}\left(\int_{\gamma'\in\dual\Gamma}(\gamma\cdot\gamma')(x)\d\nu(\gamma') \right)\d\mu(\gamma).
$$
This chain of equalities can extend by making use of the standard topological isomorphism $\omega_\Gamma \colon \Gamma \to \Gamma^{\wedge\wedge}$: 
$$
\int_{\gamma\in\dual\Gamma}\left( \int_{\gamma'\in\dual\Gamma}(\gamma\cdot\gamma')(x)\d\nu(\gamma') \right)\d\mu(\gamma)
=\int_{\gamma\in\dual\Gamma}\left(\int_{\gamma'\in\dual\Gamma} \omega_\Gamma(x)(\gamma\cdot\gamma')) \d\nu(\gamma')\right)\d\mu(\gamma).
$$
At this point we can make use of the convolution $\mu * \nu$, that is again a non-negative measure (see equation (2) from~\cite[\S1.3.1]{Rudin}), which allows us to write (see (4) from~\cite[\S1.3.1]{Rudin})
$$
\int_{\gamma\in\dual\Gamma}\left(\int_{\gamma'\in\dual\Gamma} \omega_\Gamma(x)(\gamma\cdot\gamma')) \d\nu(\gamma')\right)\d\mu(\gamma)
=\int_{\gamma\in\dual\Gamma}\omega_\Gamma(x)(\gamma)\d(\mu*\nu)(\gamma) =\int_{\gamma\in\Gamma}\gamma(x)\d(\mu*\nu)(\gamma).
$$
This gives,  for every $x\in \Gamma$, 
$$(\phi\cdot\psi)(x)=\int_{\gamma\in\dual\Gamma}\gamma(x)\d(\mu*\nu)(\gamma).$$
From the opposite implication of the Bochner Theorem we deduce that that $\phi\cdot\psi$ is positive-definite.
\end{proof}

 Next we see how the Fourier transform behaves with respect to product and convolution.

\begin{lemma}\label{cdot*} Let $\phi,\psi\in L^1(\Gamma)$. Then:
\begin{enumerate}[(1)]
  \item $\widehat{\phi * \psi} =  \widehat \phi \cdot  \widehat \psi$;
  \item $\widehat{\phi \cdot \psi} = \widehat \phi  * \widehat \psi$ if moreover $\phi,\psi\in L^2(\Gamma)$.
\end{enumerate}
\end{lemma}
\begin{proof} (1) is  \cite[Theorem 1.2.4(a),(b)]{Rudin} and (2) is \cite[(4.36)]{Fol}.
\end{proof}

 Finally, in our case of interest the Fourier transform is involutive:

\begin{lemma}\label{hathat}
If $\phi\in L^1(\Gamma)\cap\mathfrak P(\Gamma)^+$, then $\widehat{\widehat \phi}= \phi$.
\end{lemma}
\begin{proof}
By Lemma~\ref{positive_def}(2), $ \phi^\star = \phi$. Hence, $\widehat{\widehat \phi} = \phi^\star = \phi$, by \cite[(4.32)]{Fol}.
\end{proof}

\subsection{Algebraic and topological Peters monoids}\label{Petersmon}

For the compact Abelian group $K$, denote by $0\colon K\to \mathbb C$ the $0$ function and define the \emph{topological Peters monoid} 
$$
\M_\top(K)=(L^1(K)^+\cap \P(K))\setminus \{0\},\quad \text{with}\ w_\top\colon \M_\top(K)\to \R_{\geq0},\ \text{such that}\ w_\top(\phi)=\log \frac{\phi(0)}{\Vert\phi\Vert_1}.
$$
The definition of $w_\top$ is correct since $\phi(0)\geq \Vert\phi\Vert_1$ by Lemma~\ref{positive_def}(3) and $\Vert\phi\Vert_1\neq 0$, having excluded those $\phi$ that are $0$ almost everywhere.

\begin{lemma}\label{Peters-top}
In the above notation, the following statements hold true:
\begin{enumerate}[(1)]
  \item $\widehat{\phi \cdot \psi} = \widehat \phi  * \widehat \psi$ \  for \ $\phi , \psi \in \M_{\top}(K)$; 
  \item $(\M_{\top}(K),\cdot)$ is a commutative monoid, with unit $\chi_K$;
  \item the norm $w_\top\colon \M_{\top}(K)\to \R_{\geq0}$  is monotone.
\end{enumerate}
\end{lemma}
\begin{proof} 
(1) Since $\phi , \psi$ are continuous (hence, $\phi , \psi \in L^2(K)$), it is enough to apply Lemma~\ref{cdot*}(2).

\smallskip
(2) Assume, that $\phi , \psi \in  \M_{\top}(K)$. Clearly, they commute as they are complex-valued functions. Moreover, $\phi$ and $\psi$ are non-zero continuous function in $L^1(K)^+$, and hence, $\phi\cdot \psi$ is continuous, so $\phi\cdot \psi \in L^1(K)^+$. Finally, $\phi \cdot  \psi$ is positive-definite by Lemma~\ref{posdef}. 

\smallskip
(3) Let $\phi,\psi\in\M_\top(K)$. For every $x\in K$, $|\phi(x)\psi(x)|=\phi(x)\psi(x)\leq \phi(x)\psi(0)$, by Lemma~\ref{positive_def}(1). Then, we have that $\Vert \phi\cdot\psi\Vert_1\leq \Vert\phi\Vert_1 \psi(0)$, and so $$\frac{\phi(0)}{\Vert \phi\Vert_1}\leq \frac{\phi(0)\psi(0)}{\Vert\phi\cdot\psi\Vert_1}=\frac{(\phi\cdot\psi)(0)}{\Vert\phi\cdot\psi\Vert_1},$$ which gives $w_\top(\phi)\leq w_\top(\phi\cdot\psi)$. 
\end{proof}

Dually, for the discrete Abelian group $X$, denote by $0\colon X\to \mathbb C$ the $0$-function and define the \emph{algebraic Peters monoid}
$$\M_\alg(X)=(L^1(X)^+\cap \P(X))\setminus\{0\},\quad \text{with}\ w_\alg\colon \M_\alg(X)\to \R_{\geq0},\ \text{such that}\ w_\alg(\phi)=\log \frac{\Vert \phi\Vert_1}{\phi(0)}.$$
This makes sense since $\Vert\phi\Vert_1=\sum_{x\in X}\phi(x)\geq \phi(0)\neq 0$.

\begin{lemma}\label{Peters-alg}
In the above notation, the following statements hold true:
\begin{enumerate}[ (1)]
    \item if $\phi  \in \M_{\alg}(X)$ (resp., $\M_{\top}(K)$), then $\widehat \phi \in \M_{\top}(K)$ (resp., $\M_{\alg}(X)$);
    \item $(\M_{\alg}(X),*)$ is a commutative monoid, with unit $\chi_{\{0\}}$;
    \item the norm $w_\alg\colon \M_{\alg}(X)\to \R_{\geq0}$ is monotone.
\end{enumerate}
\end{lemma}
\begin{proof} 
(1)  follows from Lemma~\ref{first}.

\smallskip
(2) Assume, that $\phi , \psi \in  \M_{\alg}(X)$. Then $\phi * \psi \in L^1(X)^+$, by  Lemma~\ref{norma-}(2), and $\phi*\psi=\psi*\phi$ by Lemma~\ref{norma_1}(1). 

It remains to check that  $\phi *  \psi$ is positive-definite. By Lemma~\ref{hathat}, $\widehat{\widehat \phi} = \phi$ and $\widehat{\widehat \psi} = \psi$.
This gives $\phi * \psi = \widehat{\widehat \phi} * \widehat{\widehat \psi} =\widehat{\widehat \phi \cdot \widehat{\psi}}$, by Lemma \ref{Peters-top}(1), using the fact that $\widehat \phi , \widehat{\psi} \in \M_{\top}(K)$ by (1). Since $\widehat \phi \cdot \widehat{\psi} \in  \M_{\top}(K)$, by Lemma \ref{Peters-top}(2), using again item (1), we deduce that $\phi * \psi =\widehat{\widehat \phi \cdot \widehat{\psi}}  \in  \M_{\alg}(X)$.

\smallskip
(3) For $\phi,\psi\in\M_\alg(X)$, using Lemma~\ref{norma-}(2) and (4), we get
$$
\frac{\Vert \phi*\psi\Vert_1}{(\phi*\psi)(0)}\geq\frac{\Vert \phi\Vert_1\Vert\psi\Vert_1}{\phi(0)\Vert\psi\Vert_1}=\frac{\Vert \phi\Vert_1}{\phi(0)},
$$
that is, $w_\alg(\phi*\psi)\geq w_\alg(\phi)$.
 \end{proof}

 Now we are in position to prove the following fundamental isomorphism between the algebraic and the topological Peters monoid. From now on, $K=\dual X$.

\begin{theorem}\label{inversion}
In the above setting, $\widehat{(-)}\colon \M_\alg(X)\to\M_\top(K)$, such that $\phi\mapsto\widehat\phi$ 
is an isomorphism of normed monoids. In particular, $\widehat{\chi_{\{0\}}}=\chi_K$, $\widehat{\phi*\psi}=\widehat\phi\cdot\widehat\psi$ and $w_\alg(\phi)=w_\top(\widehat\phi)$, for all $\phi,\psi\in \M_\alg(X)$.
\end{theorem}
\begin{proof} By Lemma~\ref{Peters-alg}(1), the function is well-defined, and by Lemma~\ref{cdot*}(1), it is a homomorphism of monoids.
Moreover, for every $\phi\in\M_\alg(X)$, $w_\alg(\phi)=w_\top(\widehat \phi)$, as 
\begin{align*}
&\widehat{\phi}(0_K)=\int_{x\in X}\phi(x)\overline{0_K(x)}\d\mu(x)=\int_{x\in X}\phi(x)\d\mu(x)=\Vert\phi\Vert_1,\\
&\Vert{\widehat{\text{} \phi}}\Vert_1=\int_{x\in K}\widehat\phi(x)\d\mu(x)=\int_{x\in K}\widehat\phi(x)\overline{0_{\dual K}(x)}\d\mu(x)=\widehat{\widehat{\text{}\phi}}({0_{K^\wedge}})=\phi({0_X}),
\end{align*}
 where we use Lemma~\ref{hathat} for the last equality.
 Since $\phi=\widehat{\widehat{\phi}}$ by Lemma~\ref{hathat}, and $\widehat\phi\in \M_\alg(\dual K)$ by Lemma~\ref{Peters-alg}(1), we conclude that $\widehat{(-)}\colon \M_\alg(X)\to\M_\top(K)$ has $\widehat{(-)}\colon \M_\top(K)\to\M_\alg(\dual K)=\M_\alg(X)$ as inverse function and so it is bijective.
\end{proof}

\subsubsection{Transfering the actions to Peters monoids}

A given automorphism $\alpha\colon X\to X$ of the discrete Abelian group $X$ induces a map:
\[(-)\circ\alpha^{-1}\colon\M_{\alg}(X)\to \M_\alg(X),\quad\text{such that}\quad \phi\mapsto\phi\circ\alpha^{-1}.\]
It is not difficult to verify the properties in the following lemma, showing that $(-)\circ \alpha^{-1}$ is an isomorphism of normed monoids. In item (4) the need to take $C$ symmetric, comes from the fact that for a $C \in \Pf(X)$ one has $\chi_C $ and/or $\chi_{\alpha(C)}\in \M_\alg(X)$ precisely when $C \in \mathfrak F(X)$.

\begin{lemma}\label{Malg}
Given an automorphism $\alpha\colon X\to X$ of the discrete Abelian group $X$ and $\phi,\psi\in \M_\alg(X)$,
\begin{enumerate}[ (1)]
   \item $\phi\circ\alpha^{-1}\in \M_\alg(X)$;
   \item $\Vert\phi\circ\alpha^{-1}\Vert_1=\Vert\phi\Vert_1$ \ and \ $(\phi\circ\alpha^{-1})(0)=\phi(0)$;
   \item $(\phi*\psi)\circ\alpha^{-1}=(\phi\circ\alpha^{-1})*(\psi\circ\alpha^{-1})$ \ and \ $\chi_{\{0\}}\circ\alpha^{-1}=\chi_{\{0\}}$;
   \item if $C\in{\frak F}(X)$, then $\chi_{\alpha(C)}=\chi_C\circ \alpha^{-1}$  and so $\chi_{\alpha(C)}*\chi_{\alpha(C)}=(\chi_C*\chi_C)\circ {\alpha^{-1}}$.
\end{enumerate}
\end{lemma}

The following left $G$-action is well defined  by Lemma~\ref{Malg}:
\[G\la{\lambda_\alg} \M_\alg(X),\quad\text{such that}\ \ (\lambda_\alg)_g(\phi)=\phi\circ\lambda_g^{-1},\ \text{for all}\ \phi\in \M_{\alg}(X),g\in G.\]

Dually, a topological automorphism $\alpha\colon K\to K$ of the compact Abelian group $K$ induces a map
\[(-)\circ\alpha\colon\M_{\top}(K)\to \M_\top(K),\quad\text{such that}\quad \phi\mapsto\phi\circ\alpha.\]
It is not hard to check that $(-)\circ\alpha$ is an isomorphism of normed monoids:

\begin{lemma}\label{Mtop}
Given a topological automorphism $\alpha\colon K\to K$ of the compact Abelian group $K$ and $\phi,\psi\in \M_\top(K)$,
\begin{enumerate}[ (1)]
   \item $\phi\circ\alpha\in \M_\top(K)$;
   \item $\Vert\phi\circ\alpha\Vert_1=\Vert\phi\Vert_1$ and $(\phi\circ\alpha)(0)=\phi(0)$;
   \item $(\phi\cdot\psi)\circ\alpha=(\phi\circ\alpha)\cdot(\psi\circ\alpha)$ and $\chi_K\circ\alpha=\chi_K$;
   \item if $U\in{\frak U}(K)$, then $\chi_{\alpha^{-1} (U)}=\chi_U\circ \alpha$.
\end{enumerate}
\end{lemma}

The following left $G$-action is well-defined by Lemma~\ref{Mtop}:
\[G\la{\rho_\top} \M_\top(K),\quad\text{such that}\ \ (\rho_\top)_g(\phi)=\phi\circ\rho_g,\ \text{for all}\ \phi\in \M_{\top}(X),g\in G.\]

\begin{proposition}\label{conjugated}
The left $G$-actions $G\la{\lambda_\alg} \M_\alg(X)$ and $G\la{\rho_\top} \M_\top(K)$ are conjugated via the isomorphism of  normed monoids induced by the Fourier transform $f=\widehat{(-)}\colon \M_\alg(X)\to\M_\top(K)$. 
Thus, $h(\lambda_\alg, \s)) = h(\rho_\top, \s)$ for any right F\o lner net $\s$ for $G$.

\end{proposition}
\begin{proof} We have to verify the formula $f\circ (\lambda_\alg)_g=(\rho_\top)_g\circ f$,  for all $g\in G$. Indeed, fix $g\in G$ and $\phi\in \M_\alg(X)$, then 
$(f\circ (\lambda_\alg)_g)(\phi)=\widehat{\phi\circ\lambda_g^{-1}}$ and $((\rho_\top)_g\circ f)(\phi)=\widehat\phi\circ\rho_g$. Now, for $\gamma\in K=\dual X$, since $\rho_g=\dual{(\lambda_g)}$,
\[\begin{split}\widehat{\phi\circ\lambda_g^{-1}}(\gamma)&=((\phi\circ\lambda_g^{-1})*\gamma)(0)=\sum_{x\in X}\phi(\lambda_g^{-1}(x))\gamma(x)=\\&= \sum_{x\in X}\phi(x)\gamma(\lambda_g(x))=(\phi*(\gamma\circ\lambda_g))(0)=(\phi*\rho_g(\gamma))(0)=(\widehat\phi\circ\rho_g)(\gamma),\end{split}\]
and this gives the required equality.
The last assertion follows from Proposition~\ref{conjumon}.
\end{proof}

\subsection{Approximating the algebraic side}\label{algebraic_approximation_part}

The goal of this subsection is to prove that the left $G$-actions $G\la{\lambda_{\frak F}} {\frak F}(X)$ (see Example~\ref{Gactionsexalg}) and $G\la{\lambda_\alg} \M_\alg(X)$ are asymptotically equivalent.


We use several times the following immediate consequence of Lemmas~\ref{norma-}(2) and~\ref{Malg}(2).

\begin{corollary}\label{several}
For every $\phi\in\M_\alg(X)$ and $F\in\Pf(G)$, $\Vert T_F(\lambda_\alg,\phi)\Vert_1=\Vert\phi\Vert_1^{|F|}$ and $\Vert T_F(\lambda_\alg,\phi*\phi)\Vert_1=\Vert\phi\Vert^{2|F|}_1$.
\end{corollary}

\begin{proposition}\label{first_half_alg_dom}
$G\la{\lambda_\alg} \M_\alg(X)$ asymptotically dominates $G\la{\lambda_{\frak F}} {\frak F}(X)$.
\end{proposition}
\begin{proof}
Given $H\in {\frak F}(X)$, let $H^{(n)}=\{h_1+\ldots+h_n:h_j\in H,{j=1,\ldots,n}\}$ for all $n\in\N_+$, and define
\[\phi_n=\chi_{H^{(n)}}*\chi_{H^{(n)}}\quad\text{and}\quad \varepsilon_n=2\cdot\log\frac{|H^{(n+1)}|}{|H^{(n)}|}.\]
As Abelian groups have polynomial growth, $\{\varepsilon_n\}_{n\in\N_+}$ converges to $0$. Choose $n\in \N_{+}$ and  $F=\{f_1,\dots,f_k\}\in\Pf(G)$. 
First, with Lemma~\ref{Malg}(3) and (4), compute 
\begin{align*}
T_F(\lambda_\alg,\phi_{n})&= (\lambda_\alg)_{f_1}(\phi_n)*\ldots*(\lambda_\alg)_{f_k}(\phi_n)\\
&= (\phi_n\circ\lambda_{f_1}^{-1}) *\ldots * (\phi_n\circ\lambda_{f_k}^{-1})\\
&=((\chi_{H^{(n)}}*\chi_{H^{(n)}})\circ\lambda_{f_1}^{-1}) *\ldots * ((\chi_{H^{(n)}}*\chi_{H^{(n)}})\circ\lambda_{f_k}^{-1}) \\
&= ((\chi_{H^{(n)}}\circ \lambda_{f_1}^{-1})*(\chi_{H^{(n)}}\circ\lambda_{f_1}^{-1})) *\ldots * ((\chi_{H^{(n)}}\circ \lambda_{f_k}^{-1})*(\chi_{H^{(n)}}\circ\lambda_{f_k}^{-1}))\\
&=(\chi_{\lambda_{f_1}(H^{(n)})})*(\chi_{\lambda_{f_1}(H^{(n)})})*\ldots*(\chi_{\lambda_{f_k}(H^{(n)})})*(\chi_{\lambda_{f_k}(H^{(n)})});
\end{align*}
and also, by Corollaries~\ref{several} and~\ref{chichi}(2),
\begin{equation}\label{vert}
\Vert T_F(\lambda_\alg,\phi_n)\Vert_1=|H^{(n)}|^{2|F|}.
\end{equation}
Given $x\in T_F(\lambda,H^{(2)})$, write  $x=\sum_{j=1}^k(h_j+h_j')$ with $h_j,h_j'\in \lambda_{f_j}(H)$. By Lemmas~\ref{norma_1}(2) and~\ref{Malg}(3) and (4),
\begin{align}\label{Eq:July5*}
T_F(\lambda_\alg,\phi_{n+1})_{x}&=((\lambda_\alg)_{f_1}(\phi_{n+1}))_{h_1+h'_1}*\ldots*((\lambda_\alg)_{f_k}(\phi_{n+1}))_{h_k+h'_k}\\
\notag&=(\chi_{\lambda_{f_1}(H^{(n+1)})})_{h_1}*(\chi_{\lambda_{f_1}(H^{(n+1)})})_{h'_1}*\ldots*(\chi_{\lambda_{f_k}(H^{(n+1)})})_{h_k}*(\chi_{\lambda_{f_k}(H^{(n+1)})})_{h'_k} .
\end{align}
Since $(\chi_{\lambda_f(H^{(n+1)})})_h\geq \chi_{\lambda_f(H^{(n)})}$, for all $f\in F$ and $h\in \lambda_f(H)$, we deduce, with Lemma~\ref{norma-}(5),  that 
$$(\chi_{\lambda_{f_1}(H^{(n+1)})})_{h_1}*(\chi_{\lambda_{f_1}(H^{(n+1)})})_{h'_1}*\ldots*(\chi_{\lambda_{f_k}(H^{(n+1)})})_{h_k}*(\chi_{\lambda_{f_k}(H^{(n+1)})})_{h'_k} 
\! \geq \! (\chi_{\lambda_{f_1}(H^{(n)})})*(\chi_{\lambda_{f_1}(H^{(n)})})*\ldots*(\chi_{\lambda_{f_k}(H^{(n)})})*(\chi_{\lambda_{f_k}(H^{(n)})})\! =\! T_F(\lambda_\alg,\phi_{n}).$$
Using \eqref{Eq:July5*} and the above inequality,  after evaluation at $0$,  we get 
\begin{equation}\label{first_eq_appendix}
T_F(\lambda_\alg,\phi_{n+1})(x) = T_F(\lambda_\alg,\phi_{n+1})_{x}(0)\geq T_F(\lambda_\alg,\phi_{n})(0).
\end{equation}
Hence,
\begin{equation*}
\Vert T_F (\lambda_{\alg}, \phi_{n+1})\Vert_1=\sum_{x\in X}T_F (\lambda_{\alg}, \phi_{n+1})(x)\overset{(*)}\geq \sum_{x\in T_{F}(\lambda,H^{(2)})}T_F (\lambda_{\alg}, \phi_{n+1})(x)\overset{(**)}\geq |T_{F}(\lambda,H^{(2)})|\cdot T_F(\lambda_\alg,\phi_n)(0),
\end{equation*}
where $(*)$ follows as we are shrinking the indexing family, while $(**)$ follows by \eqref{first_eq_appendix}. We conclude that:
\begin{equation*}
|T_{F}(\lambda,H^{(2)})|\leq \frac{\Vert T_F (\lambda_{\alg}, \phi_{n+1})\Vert_1}{T_F(\lambda_\alg,\phi_n)(0)}=\frac{\Vert T_F (\lambda_{\alg}, \phi_{n})\Vert_1}{T_F(\lambda_\alg,\phi_n)(0)}\cdot\frac{\Vert T_F (\lambda_{\alg}, \phi_{n+1})\Vert_1}{\Vert T_F (\lambda_{\alg}, \phi_{n})\Vert_1}=\frac{\Vert T_F(\lambda_\alg,\phi_n)\Vert_1}{T_F(\lambda_\alg,\phi_n)(0)}\cdot\left(\frac{|H^{(n+1)}|}{|H^{(n)}|}\right)^{2|F|},
\end{equation*}
where the last equality follows from \eqref{vert}. Taking logarithms, one gets the desired inequality: 
\[\log|T_{F}(\lambda,H)|\leq \log|T_{F}(\lambda,H^{(2)})|\leq w_\alg(T_F(\lambda_\alg,\phi_n))+|F|\cdot\varepsilon_n.\qedhere\]
\end{proof}

The converse of Proposition~\ref{first_half_alg_dom} is technically a bit more challenging, as it relies on a series of successive reductions, that correspond to the various parts of the following lemma:

\begin{lemma}\label{tech_for_second_alg_domination} 
Let $\phi\in \M_\alg(X)$ and $F\in\Pf(G)$. Then:
\begin{enumerate}[ (1)]
  \item $w_\alg(T_F(\lambda_\alg,\phi))\leq w_\alg(T_F(\lambda_\alg,\phi*\phi))$;
  \item for $\psi=\frac{\phi}{\Vert\phi\Vert_1}$, $\Vert\psi\Vert_1=1$ and $w_\alg(T_F(\lambda_\alg,\phi))= w_\alg(T_F(\lambda_\alg,\psi))$;
  \item given $n\in\N_{+}$, there is $H\in {\frak F}(X)$ such that $\Vert\phi-\phi\cdot \chi_H\Vert_1\leq \frac{1}{2n}$;
  \item suppose that $\Vert\phi\Vert_1=1$, take $n\in \N$ with $n>1$, and let  $\xi=\phi\cdot\chi_H$ be as in (3), then $\xi*\xi\in \M_\alg(X)$ and
\[w_\alg(T_F(\lambda_\alg,\phi*\phi))\leq w_\alg(T_F(\lambda_\alg,\xi*\xi))+|F|\log\frac{n}{n-1}.\]
\end{enumerate}
\end{lemma}
\begin{proof}
(1) By Corollary~\ref{several}, $\Vert T_F(\lambda_\alg,\phi*\phi)\Vert_1=\Vert\phi\Vert^{2|F|}_1$. Furthermore,
\begin{align*}
T_F(\lambda_\alg,\phi*\phi)(0)&=\sum_{x\in X}T_F(\lambda_\alg,\phi)(x)\cdot T_F(\lambda_\alg,\phi)(-x)\overset{(*)}\leq \sum_{x\in X}T_F(\lambda_\alg,\phi)(x)\cdot T_F(\lambda_\alg,\phi)(0)\\
&= \Vert T_F(\lambda_\alg,\phi)\Vert_1\cdot T_F(\lambda_\alg,\phi)(0)= \Vert\phi\Vert^{|F|}_1\cdot T_F(\lambda_\alg,\phi)(0),
\end{align*}
where  $(*)$ holds by Lemma~\ref{positive_def}(1).
Using  both equalities in Corollary~\ref{several} and the above computation, we get:
\[\frac{\Vert T_F(\lambda_\alg,\phi)\Vert_1}{T_F(\lambda_\alg,\phi)(0)}=\frac{\Vert\phi\Vert_1^{|F|}}{T_F(\lambda_\alg,\phi)(0)}=\frac{\Vert\phi\Vert_1^{2|F|}}{\Vert\phi\Vert_1^{|F|}\cdot T_F(\lambda_\alg,\phi)(0)}\leq \frac{\Vert T_F(\lambda_\alg,\phi*\phi)\Vert_1}{T_F(\lambda_\alg,\phi*\phi)(0)}.\]

\smallskip
(2) Let $\varepsilon=\Vert\phi\Vert_1^{-1}\in \R_{>0}$. Then
$\Vert\psi\Vert_1=\sum_{x\in X}\psi(x)=\varepsilon\cdot \sum_{x\in X}\phi(x)=\varepsilon\cdot \Vert\phi\Vert_1=1$. Similarly, if $F=\{f_1,\dots,f_k\}$, then
\[T_F(\lambda_{\alg},\psi)(0)=(\psi^{\lambda_{f_1}}*\ldots*\psi^{\lambda_{f_k}})(0)=\varepsilon^k\cdot(\phi^{\lambda_{f_1}}*\ldots*\phi^{\lambda_{f_k}})(0)=\varepsilon^k\cdot T_F(\lambda_{\alg},\phi)(0).\]
Now, by Corollary~\ref{several}, 
$\Vert T_F(\lambda_{\alg},\psi)\Vert_1=\Vert\psi\Vert_1^k=1$ and $\Vert T_F(\lambda_{\alg},\phi)\Vert_1=\Vert\phi\Vert_1^k=\varepsilon^{-k}$. 
Combining the above computations we can easily conclude:
\[
w_\alg(T_F(\lambda_{\alg},\phi))=\frac{\Vert\phi\Vert_1^k}{T_F(\lambda_{\alg},\phi)(0)}=\frac{1}{\varepsilon^{k}\cdot T_F(\lambda_{\alg},\phi)(0)}=\frac{1}{T_F(\lambda_{\alg},\psi)(0)}=w_\alg(T_F(\lambda_{\alg},\psi)).
\]

\smallskip
(3) follows since $\Vert\phi\Vert_1=\sum_{x\in X}\phi(x)=\sup\{\sum_{x\in H}\phi(x):H\in{\frak F}(X)\}<\infty$.

\smallskip
(4)  Since $\xi\in\M_\alg(X)$ by definition, also $\xi*\xi\in\M_\alg(X)$. Furthermore, by the triangular inequality and  Lemma~\ref{norma-}(2) and (5),
\begin{align*}
\Vert\xi*\xi-\phi*\phi\Vert_1&=\Vert\xi*\xi-\xi*\phi+\xi*\phi-\phi*\phi\Vert_1\leq\Vert\xi*\xi-\xi*\phi\Vert_1+\Vert\xi*\phi-\phi*\phi\Vert_1\leq\\
&\leq \Vert\xi\Vert_1\Vert\xi-\phi\Vert_1+\Vert\phi\Vert_1\Vert\xi-\phi\Vert_1\overset{(*)}{\leq} 2\Vert\xi-\phi\Vert_1<\frac{1}{n},
\end{align*}
where $(*)$ follows since $\Vert\xi\Vert_1\leq \Vert\phi\Vert_1=1$. 
 As $\Vert\phi*\phi\Vert_1=\Vert\phi\Vert_1^2=1$ by Lemma~\ref{norma-}(2), we get 
\begin{equation}\label{22722}
\Vert\phi*\phi\Vert_1-\Vert\xi*\xi-\phi*\phi\Vert_1\geq 1-\frac{1}{n}=\frac{n-1}{n}.
\end{equation}
Furthermore, the fact that $\xi\leq \phi$ implies that $\xi*\xi\leq \phi*\phi$ by Lemma~\ref{norma-}(5), and  so $(\lambda_\alg)_g(\xi*\xi)\leq (\lambda_\alg)_g(\phi*\phi)$ for each $g\in G$; thus,  $T_F(\lambda_\alg,\xi*\xi)(0)\leq T_F(\lambda_\alg,\phi*\phi)(0)$.  Using the latter inequality, \eqref{22722}, Corollary~\ref{several} (to see that $\Vert T_F(\lambda_\alg,\xi*\xi)\Vert_1=\Vert\xi*\xi\Vert^{|F|}_1$)  and that $\Vert T_F(\lambda_\alg,\phi*\phi)(0)\Vert_1=\Vert \phi\Vert_1^{2|F|}=1$ by Corollary~\ref{several}, we deduce that:
\begin{equation*}\begin{split}
\frac{\Vert T_F(\lambda_\alg,\xi*\xi)\Vert_1}{T_F(\lambda_\alg,\xi*\xi)(0)}&\geq \frac{\Vert\xi*\xi\Vert^{|F|}_1}{T_F(\lambda_\alg,\phi*\phi)(0)}
\geq\frac{(\Vert\phi*\phi\Vert_1-\Vert\xi*\xi-\phi*\phi\Vert_1)^{|F|}}{T_F(\lambda_\alg,\phi*\phi)(0)}>\\
&>\frac{((n-1)/n)^{|F|}}{T_F(\lambda_\alg,\phi*\phi)(0)}=\frac{\Vert T_F(\lambda_\alg,\phi*\phi)(0)\Vert_1}{T_F(\lambda_\alg,\phi*\phi)(0)}\cdot \left(\frac{n-1}{n}\right)^{|F|}.
\end{split}\end{equation*}
To conclude, divide both sides by $((n-1)/n)^{|F|}$ and apply logarithms.
\end{proof}

\begin{proposition}
$G\la{\lambda_{\frak F}} {\frak F}(X)$ asymptotically dominates $G\la{\lambda_\alg} \M_\alg(X)$.
\end{proposition}
\begin{proof}
Given $\phi\in \M_\alg(X)$,  let $\psi=\frac{\phi}{\Vert\phi\Vert_1}$ and choose, for each $1<n\in\N$, $H_n\in{\frak F}(X)$ such that $\Vert\psi-\psi\cdot\chi_{H_n}\Vert_1<1/(2n)$ (we can do that by Lemma~\ref{tech_for_second_alg_domination}(3)). Letting $\psi_n=\psi\cdot \chi_{H_n}$, we have that $\psi_n*\psi_n$ has finite support by Lemma~\ref{norma_1}(3). Choose $E_n\in{\frak F}(X)$ containing $\supp(\psi_n*\psi_n)$, set $\varepsilon_n=\log|n/(n-1)|$, and let us verify that, for each $F\in\Pf(G)$,
\[w_{\alg}(T_F(\lambda_\alg,\phi))\leq \log|T_F(\lambda,E_n)|+|F|\cdot\varepsilon_n.\]
Using Lemma~\ref{tech_for_second_alg_domination}(2), (1) and (4), we can simplify this as follows 
\[w_\alg(T_F(\lambda_\alg,\phi))=w_\alg(T_F(\lambda_\alg,\psi)\leq w_\alg(T_F(\lambda_\alg,\psi*\psi))\leq w_\alg(T_F(\lambda_\alg,\psi_n*\psi_n))+|F|\cdot\varepsilon_n . \]
Hence, we are left with the proof of the following inequality:
\begin{equation}\label{eq_second_alg_dom}
w_{\alg}(T_F(\lambda_\alg,\psi_n*\psi_n))\leq \log|T_F(\lambda,E_n)|.
\end{equation}
For this, consider $\chi_{T_F(\lambda,E_n)}*T_F(\lambda_\alg,\psi_n*\psi_n)(0)$. Then,
\[
\chi_{T_F(\lambda,E_n)}*T_F(\lambda_\alg,\psi_n*\psi_n)(0)=\sum_{x\in X}\chi_{T_F(\lambda,E_n)}(x)\cdot T_F(\lambda_\alg,\psi_n*\psi_n)(-x)=\Vert T_F(\lambda_\alg,\psi_n*\psi_n)\Vert_1,
\]
since $\supp(T_F(\lambda_\alg,\psi_n*\psi_n))\subseteq T_F(\lambda,E_n)$, which is symmetric. Furthermore, by Lemma~\ref{positive_def}(2),
\[\chi_{T_F(\lambda,E_n)}*T_F(\lambda_\alg,\psi_n*\psi_n)(0)=\sum_{x\in X}\chi_{T_F(\lambda,E_n)}(x)\cdot T_F(\lambda_\alg,\psi_n*\psi_n)(-x)\leq |T_F(\lambda,E_n)|T_F(\lambda_\alg,\psi_n*\psi_n)(0).\spadesuit\]
Now \eqref{eq_second_alg_dom} is obtained by combining the above computations and taking logarithms.
\end{proof}

The above two propositions give immediately the following:

\begin{corollary}\label{cor1}
$G\la{\lambda_{\frak F}} {\frak F}(X)$ and $G\la{\lambda_\alg} \M_\alg(X)$ are asymptotically equivalent. 
Hence, $h(\lambda_{\frak F}, \s) = h(\lambda_\alg , \s)$ for every right F\o lner net $\s$ for $G$.
\end{corollary}

\subsection{Approximating the topological side}\label{topological_approximation_part}

The goal of this subsection is to verify the asymptotic equivalence of two pairs of left $G$-actions: first we show that $G\la{\rho_\top} \M_\top(K)$ and $G\la{\rho_{\frak U}}{\frak U}(K)$ (see Example~\ref{Gactionsextop2}) are asymptotically equivalent and, after that, we show that $G\la{\rho_{\frak U}}{\frak U}(K)$ and $G\la{\rho_{\cov}} {\cov}(K)$ (see Example~\ref{Gactionsextop}) are equivalent. 

\begin{proposition}\label{first_alg_dom}\label{cor2}
$G\la{\rho_{\frak U}}{\frak U}(K)$ and $G\la{\rho_\top}\M_\top(K)$ are asymptotically equivalent.  
Hence, $h(\rho_{\frak U},\s ) = h(\rho_\top, \s)$ for every right F\o lner net $\s$ for $G$.
\end{proposition}
\begin{proof} 
We show first that $G\la{\rho_\top}\M_\top(K)$ dominates $G\la{\rho_{\frak U}} {\frak U}(K)$.
Let $U\in {\frak U}(K)$. By  Corollary~\ref{prescribed_support}, there is $\phi\in \M_\top(K)$  such that $\supp(\phi)\subseteq U$ and, up to rescaling, we may suppose that $\phi(0)=1$. Then, ${\Vert T_F(\rho_\top,\phi)\Vert_1}\leq\mu(T_F(\rho,U))$, by Lemma~\ref{positive_def}(3).
Therefore, as $T_F(\rho_\top,\phi)(0)=\prod_{g\in F}(\phi\circ\rho_g)(0)=\phi(0)^{|F|}$,
\[
w_\top(T_F(\rho_\top,\phi))=\log\left(\frac{\phi(0)^{|F|}}{\Vert T_F(\rho_\top,\phi)\Vert_1}\right)\geq-\log\mu(T_F(\rho,U))=v_{\frak U}(T_F(\rho,U)).
\]

\smallskip
Next we check that $G\la{\rho_{\frak U}}{\frak U}(K)$ asymptotically dominates $G\la{\rho_\top}\M_\top(K)$.
Let $\phi \in \M_\top(K)$ and set $c=\phi(0)>0$. For each $r\in\mathbb R_{>0}$, we let $B_0(r)\subseteq \mathbb C$ be the open ball of radius $r$ around $0$. Furthermore, for each $\varepsilon>0$ we set:
\[
V(\phi,\varepsilon)=K\setminus \phi^{-1}(B_0(c/(1+\varepsilon)))\subseteq K.
\]
Since $\phi$ is continuous and each $B_0(r)$ is open, also $\phi^{-1}(B_0(r))$ is open, showing that $V(\phi,\varepsilon)$ is closed (and, therefore, compact). Furthermore, $\phi(0)\notin \cl(B_0(c/(1+\varepsilon))))$, so $V(\phi,\varepsilon)$ contains an open neighborhood of $0$. 

Now, for each $g\in G$, we have that  $\rho_g^{-1}(V(\phi,\varepsilon))=V(\phi\circ\rho_g,\varepsilon)$, and so, for each $k\in K$,  
\begin{equation}\label{chiE}
\frac{(1+\varepsilon)}{\phi(0)}\cdot(\phi\circ\rho_g)(k)\geq \chi_{\rho_g^{-1}(V(\phi,\varepsilon))}(k),\quad\text{for every $k\in K$.}
\end{equation}
For each $n\in\N_{+}$, choose $E_n\in {\frak U}(K)$ such that $E_n\subseteq V(\phi,1/n)$. Pick an arbitrary $F\in\Pf(G)$ and let $\varepsilon=1/n$.  With $g$ running in $F$, multiply all inequalities \eqref{chiE} and apply the norm $\Vert-\Vert_1$ to obtain: 
\[\begin{split}
\left(\frac{n+1}{n}\right)^{\vert F\vert}\frac{\Vert T_F(\rho_\top,\phi)\Vert_1}{T_F(\rho_\top,\phi)(0)}&= \left(\frac{n+1}{n}\right)^{|F|}\frac{\lVert\prod_{g\in F}\phi\circ\rho_g\rVert_1}{\phi(0)^{|F|}}\geq \\
&\geq \left\Vert\prod_{g\in F}\chi_{\rho_g^{-1}(V(\phi,1/n))}\right\Vert_1 \geq  \left\Vert\prod_{g\in F}\chi_{\rho_g^{-1}(E_n)} \right\Vert_1
\! =\mu\left(\bigcap_{g\in G}\rho_g^{-1}(E_n)\right)=\mu(T_F(\rho_{\frak U},E_n)).
\end{split}\]
Let $\varepsilon_n=\log((n+1)/n)$. Taking logarithms we get the desired inequality $w_\top(T_F(\rho_\top,\phi))\leq -\log\mu(T_F(\rho_{\frak U},E_n))+|F|\cdot \varepsilon_n.$ 

The last assertion follows from Proposition~\ref{conjumon}.
\end{proof}

Next we show that $G\la{\rho_{\frak U}} {\frak U}(K)$ is equivalent to $G\la{\rho_{\cov}} {\cov}(K)$ assuming that $K$ is an arbitrary compact group.

\begin{proposition}\label{prop_cov_doms_neig}\label{cor3}
$G\la{\rho_{\frak U}}{\frak U}(K)$ and $G\la{\rho_{\cov}}{\cov}(K)$ are equivalent. Hence, $h(\rho_{\frak U}, \s ) = h(\rho_{\cov}, \s)$
for every right F\o lner net $\s$.
\end{proposition}
\begin{proof} 
First we show that $G\la{\rho_{\cov}} {\cov}(K)$ dominates $G\la{\rho_{\frak U}}{\frak U}(K)$. To this end, let $U\in {\frak U}(K)$, choose an open subset $V$ of $K$ such that $1\in V\subseteq V^{-1}V\subseteq U$, and consider the open cover $\V=\U_K[V] \in\cov(K)$. Fix $F\in\Pf(G)$.
By Lemma~\ref{tech_AT}(3), $\W=T_F(\rho_\cov,\V)$ refines the even cover $\U_K[C_F(\rho,U)]$. Therefore, $K$ can be covered by a family of $N(\W)$-many members of $\W$, each of which is contained in a translate of $C_F(\rho,U)$, hence $1=\mu(K)\leq N(\W)\cdot \mu(C_F(\rho,U))$. This proves that
\begin{equation*}\label{May16}
-\log\mu(C_F(\rho,U))\leq \log N(T_F(\rho_\cov,\V)).
\end{equation*} 

To see that $G\la{\rho_{\frak U}}{\frak U}(K)$ dominates $G\la{\rho_{\cov}} {\cov}(K)$, let $\V\in \cov(K)$. By Lemma~\ref{tech_AT}(1), there exists an open neighborhood $W$ of $1$ such that $\V \preceq \U_K[W]$. Then, for each $F\in\Pf(G)$, 
\begin{equation}\label{last_prop_dom_app_eq1}
\log N(T_F(\rho_\cov,\V))\leq \log N(T_F(\rho_\cov,\U_K[W])).
\end{equation}
Choose $U\in {\frak U}(K)$ such that $UU\subseteq W$. Recall that there exists a finite subset $D$ of $K$ such that $\mathcal D=\{dC_F(\rho,W):d\in D\}$ is a subcover of $\U_K[C_F(\rho,W)]$ (by Lemma~\ref{tech_AT}(4)) and  the union $\bigcup_{d\in D} (dC_F(\rho,U))$ is disjoint; so,
$$N(\U_K[C_F(\rho,W)]) \leq |D|\ \mbox{ and } \ |D|\cdot \mu(C_F(\rho,U)) \leq \mu(K)\leq 1.$$ 
This gives $N(\U_K[C_F(\rho, W)]) \leq \mu(C_F(\rho,U))^{-1}$. As $N(T_F(\rho_\cov,\U_K[W])) \leq N(\U_K[C_F(\rho,W)])$ by Lemma~\ref{tech_AT}(2),
we conclude that $N(T_F(\rho_\cov,\U_K[W]))\leq  \mu(C_F(\rho,U))^{-1}$.  Taking logarithms, together with \eqref{last_prop_dom_app_eq1}, we get 
$$\log N(T_F(\rho_\cov,\V))\leq\log N(T_F(\rho_\cov,\U_K[W]))\leq -\log\mu(C_F(\rho,U)).$$

 This proves that $G\la{\rho_{\frak U}}{\frak U}(K)$ and $G\la{\rho_{\cov}} {\cov}(K)$ are equivalent.  The last assertion follows from Proposition~\ref{conjumon}. 
\end{proof}

\subsection{The proof of the Bridge Theorem for amenable group actions}\label{proof_BT}

\begin{theorem}[Bridge Theorem]\label{BT_for_invertible_actions} Let $G$ be an amenable group, $K$ a compact Abelian group, $X$ a discrete Abelian group and let $K\ra{\rho} S$ and $S\la{\lambda}X$ be a right and a left $G$-action, respectively.
Then: 
\begin{enumerate}[(1)]
\item $h_\top(\rho)=h_\alg(\dual\rho)$; 
\item $h_\alg(\lambda)=h_\top(\dual\lambda)$. 
\end{enumerate}
\end{theorem}
\begin{proof} 
Since (1) and (2) are equivalent, we choose to prove (2). Let $\rho=\dual\lambda$. For a fixed right F\o lner net $\s$ for $G$, we have, 
by Example~\ref{Gactionsexalg} and Example~\ref{Gactionsextop}, 
\[h_\alg(\lambda)=h(\lambda_{\frak F},\s)\quad\text{and}\quad h_\top(\rho)=h(\rho_{\cov},\s).\]
Furthermore,  $h(\lambda_{\frak F},\s)=h( \lambda_\alg,\s)$ by Corollary~\ref{cor1}, $h(\lambda_\alg,\s) = h(\rho_\top,\s)$ by Proposition~\ref{conjugated}, 
$h(\rho_\top,\s) = h(\rho_{\frak U},\s) $ by Proposition~\ref{cor2}, and $h(\rho_{\frak U},\s) =h(\rho_\cov,\s)$  by Proposition~\ref{cor3}. 
We conclude that 
\[h_\alg(\lambda)=h(\lambda_{\frak F},\s)=h(\lambda_\alg,\s)=h(\rho_\top,\s)=h(\rho_{\frak U},\s)= h(\rho_\cov,\s) =h_\top(\rho). \qedhere\]
\end{proof}

---------------------------------------------------------

\bigskip
\noindent
Department of Mathematics, Computer Science and Physics, University of Udine, \\
Via Delle Scienze 206, 33100 Udine, Italy

\medskip
\noindent Department of Mathematics, Computer Science and Physics, University of Udine, \\
Via Delle Scienze
206, 33100 Udine, Italy

\medskip
\noindent 
Departament de Matem\`atiques, Facultat de Ci\`encies, Universitat Aut\`onoma de Barcelona,\\
Facultat de Ci\`encies, Edifici C, 08193, Bellaterra (Barcelona), Spain

\end{document}